\documentclass[preprint]{elsarticle}
\usepackage{amsmath,amssymb,mathrsfs}
\usepackage{graphicx,color,url}
\usepackage{epsfig}
\usepackage{keyval}
\usepackage{algorithm}
\usepackage{algorithmic}
\usepackage{hyperref}
\usepackage{amsthm}
\usepackage{subfig}
\usepackage{listings}
\usepackage{color}
\usepackage{courier}

\newcommand{\OMIT}[1]{}

\newcommand{\pd}[2]{\frac{\partial #1}{\partial #2}}
\newtheorem{theorem}{Theorem}[section]

\newtheorem{lemma}{Lemma}[section]
\newtheorem{myremark}{Remark}[section]

\usepackage[numbers]{natbib}

\definecolor{mygreen}{RGB}{28,172,0} % color values Red, Green, Blue
\definecolor{mylilas}{RGB}{170,55,241}

\allowdisplaybreaks

\begin{document}

\lstset{language=Matlab,%
    basicstyle=\footnotesize\ttfamily,
    breaklines=true,%
    morekeywords={matlab2tikz},
    keywordstyle=\color{blue},%
    morekeywords=[2]{1}, keywordstyle=[2]{\color{black}},
    identifierstyle=\color{blue},%
    stringstyle=\color{mylilas},
    commentstyle=\color{mygreen},%
    showstringspaces=false,%without this there will be a symbol in the places where there is a space
    numbers=left,%
    numberstyle={\tiny \color{black}},% size of the numbers
    numbersep=9pt, % this defines how far the numbers are from the text
    emph=[1]{for,end,break},emphstyle=[1]\color{red}, %some words to emphasise
}

\begin{frontmatter}

\journal{International Journal for Numerical Methods in Engineering}

\title{Tuned Hybrid Non-Uniform Subdivision Surfaces with Optimal Convergence Rates}

\author[2]{Xiaodong Wei}
\author[1]{Xin Li\corref{cor1}}
\author[3]{Yongjie Jessica Zhang}
\author[4]{Thomas J.R. Hughes}
\cortext[cor1]{Corresponding author, lixustc@ustc.edu.cn, tel: +86-551-63607202}
\address[1]{School of Mathematical Sciences, University of Science and Technology of China, Hefei, Anhui, China}
\address[2]{Institute of Mathematics, \'Ecole Polytechnique F\'ed\'erale de Lausanne, 1015 Lausanne, Switzerland}
\address[3]{Department of Mechanical Engineering, Carnegie Mellon University, Pittsburgh, PA 15213, USA}
\address[4]{Oden Institute, The University of Texas at Austin, Austin, TX 78712, United States}

\begin{abstract}
This paper presents an enhanced version of our previous work, hybrid non-uniform subdivision surfaces~\cite{xin18}, to achieve optimal convergence rates in isogeometric analysis. We introduce a parameter $\lambda$ ($\frac{1}{4}<\lambda<1$) to control the rate of shrinkage of irregular regions, so the method is called tuned hybrid non-uniform subdivision (tHNUS). Our previous work corresponds to the case when $\lambda=\frac{1}{2}$. While introducing $\lambda$ in hybrid subdivision significantly complicates the theoretical proof of $G^1$ continuity around extraordinary vertices, reducing $\lambda$ can recover the optimal convergence rates when tuned hybrid subdivision functions are used as a basis in isogeometric analysis. From the geometric point of view, the tHNUS retains comparable shape quality as \cite{xin18} under non-uniform parameterization. Its basis functions are refinable and the geometric mapping stays invariant during refinement. Moreover, we prove that a tuned hybrid subdivision surface is globally $G^1$-continuous. From the analysis point of view, tHNUS basis functions form a non-negative partition of unity, are globally linearly independent, and their spline spaces are nested. We numerically demonstrate that tHNUS basis functions can achieve optimal convergence rates for the Poisson's problem with non-uniform parameterization around extraordinary vertices.
\end{abstract}


\end{frontmatter}

%%%%%%%%%%%%%%%%%%%%%%%%%\

\section{Introduction}

Isogeometric analysis (IGA) has emerged as a powerful technology to unify geometric modeling and numerical
simulation~\cite{ref:hughes05, ref:cottrell09}, which employs the same basis functions used in computer-aided
design (CAD) and simulations. IGA has grown into a large family of numerical methods incorporating various spline techniques, such as NURBS (Non-Uniform Rational B-Splines)~\cite{ref:hughes05}, hierarchical B-splines~\cite{ref:vuong11}, T-splines~\cite{ref:sederberg04, Li10, ScLiSeHu10, ASTDual12, xinli_arbitrary, xin_as++, ref:wei17, xin_refinement}, polynomial splines over T-meshes~\cite{PHT08}, and locally refinable B-splines~\cite{LRBspline13}.

The study of extraordinary vertices\footnote{An interior vertex in a quadrilateral mesh is called an extraordinary vertex if it is shared by other than four faces.} has been one of the most active research directions in IGA because they are inevitable in complex watertight geometric representations. Along this direction, simultaneously fulfilling the requirements from both design and analysis is a significant challenge. Numerous methods have been developed over the past few years, but among them, only a few constructions can achieve optimal convergence rates in IGA, such as geometrically smooth multi-patch construction~\cite{ref:collin16, ref:kapl17}, degenerated B\'{e}zier construction~\cite{ref:tnguyen16, c0spline, ref:utspline19}, manifold-based construction~\cite{manifold}, and blended $C^0$ construction for unstructured hexahedral meshes~\cite{c03d}. A common simplification in all these constructions is to adopt uniform parameterization around extraordinary vertices, i.e., the surrounding knot intervals are assumed to be the same. While the support of non-uniform parameterization is a necessary step forward to be compatible with the current industry standard in CAD, i.e., NURBS, the related study on the above-mentioned constructions has not been reported in the literature.

On the other hand, subdivision methods, as a generalization of splines, provide a flexible means to deal with extraordinary vertices, where an infinite series of spline patches are smoothly joined around extraordinary vertices. The combination of flexibility and global smoothness makes them not only the standard in the computer animation industry but also a promising candidate for IGA. Indeed, some of the subdivision methods have been studied in the context of IGA, such as the use of Loop subdivision in thin-shell analysis~\cite{CiOrShr00} and the development of Catmull-Clark solids~\cite{burck2010}. However, several challenging problems need to be carefully investigated before we can fully leverage the power of subdivision methods, such as developing efficient quadrature rules to integrate infinite piecewise polynomials around extraordinary vertices~\cite{ref:juttler16, ref:barendrecht18}, supporting non-uniform parameterizations to be compatible with NURBS~\cite{Sederberg98, Muller06, ref:cashman09, Muller10, xin16}, and recovering optimal convergence rates \cite{ref:ma19}. This paper intends to address both non-uniform parameterization and optimal convergence behavior at the same time.

The present work is a follow-up of our preceding work on hybrid non-uniform subdivision (HNUS)~\cite{xin18}, which generalizes bicubic NURBS to arbitrary topology with proved $G^1$ continuity around extraordinary vertices. HNUS features high quality in geometric modeling under non-uniform parameterization. When applied to IGA, HNUS basis functions are not optimal but lead to improved convergence rates compared to Catmull-Clark subdivision.

Motivated by the idea of tuned Catmull-Clark subdivision \cite{ref:ma19} under uniform parameterization, we introduce a parameter $\lambda\in (\frac{1}{4},1)$ in HNUS to control the shrinkage rate in irregular regions such that we can recover optimal convergence. The enhanced version of HNUS is therefore called \emph{tuned hybrid non-uniform subdivision} (tHNUS). In fact, the parameter $\lambda$ is the subdominant eigenvalue (the 2nd and 3rd eigenvalues which are equal) of the tHNUS subdivision matrix, that plays a crucial role in surface continuity \cite{Reif95} as well as the convergence performance \cite{ref:ma19}. Note that tHNUS coincides with the original HNUS when $\lambda=\frac{1}{2}$. From the geometric point of view, tHNUS retains comparable shape quality as HNUS. Its basis functions are refinable and the geometric mapping stays invariant during refinement. Moreover, we prove that the tHNUS surface is globally $G^1$-continuous. From the analysis point of view, tHNUS basis functions form a non-negative partition of unity, are globally linearly independent, and their spline spaces are nested. Moreover, we numerically demonstrate that tHNUS can achieve optimal convergence rates in the Poisson's problem by reducing $\lambda$, regardless of whether parameterization around extraordinary vertices is uniform or not. As an interesting side product, we also show that simply applying the standard Gauss quadrature rule to every element (close to or far away from extraordinary vertices) in tHNUS does not influence simulation accuracy or convergence.

The reminder of the paper is organized as follows. Section \ref{sec:ehs} presents the subdivision rules of tHNUS. The proof of $G^1$ continuity for tHNUS surfaces is given in Section \ref{sec:proof}. The tHNUS basis functions are derived and their properties are discussed in Section \ref{sec:basis}. In Section \ref{sec:result}, we present numerical tests of both geometric modeling and IGA. Section \ref{sec:con} concludes the paper and discusses the future work.

\section{Tuned hybrid non-uniform subdivision surfaces}
\label{sec:ehs}

Our discussion assumes that the input control mesh is a regular manifold mesh where all the faces are quadrilaterals. If initially a mesh has polygonal faces, we apply a single NURSS (Non-Uniform Recursive Subdivision Surface) refinement~\cite{Sederberg98,xin16} to obtain an all-quadrilateral mesh. A non-negative scalar, which is called the \emph{knot interval}, is assigned to each edge of the control mesh. We further assume that in each face, the knot intervals on the opposite edges coincide. A non-uniform parameterization is obtained by assigning different knot intervals to different edges as long as the assumption for knot intervals holds.

The tHNUS consists of two sets of rules: the topological rules to manipulate mesh connectivity, and the geometric rules to update the coordinates of involved control points. Each set of rules can be further divided into the first level and the subsequent levels. All the rules of tHNUS coincide with those of the original HNUS~\cite{xin18} except for the geometric rule corresponding to the subsequent  levels. We will concisely cover all the rules in the following to keep the explanation self-contained. One may refer to~\cite{xin18} for more details.

We start with the topological rules of tHNUS, which consist of rules for the first level and the subsequent levels, as illustrated in Figure~\ref{fig:topo}. The rule corresponding to the first level converts the input quadrilateral mesh to its hybrid counterpart. Each extraordinary vertex is replaced by a polygonal face, whereas each spoke edge\footnote{A spoke edge is an edge touching a certain extraordinary vertex.} is replaced by a quadrilateral face. To make the resulting mesh conforming, additional vertices and edges are further replaced by certain faces; see Figure~\ref{fig:topo}(a). Note that all the edges of a newly added polygonal face have a zero knot interval. Under the assumption of knot intervals, this means that all the newly added faces have a zero (parametric) measure. In regular regions, introducing zero-knot-interval edges leads to a reduction in continuity of basis functions from $C^2$ to $C^1$.

\begin{figure}[htbp]
\begin{center}
\begin{tabular}{cc}
\includegraphics[width=0.45\textwidth]{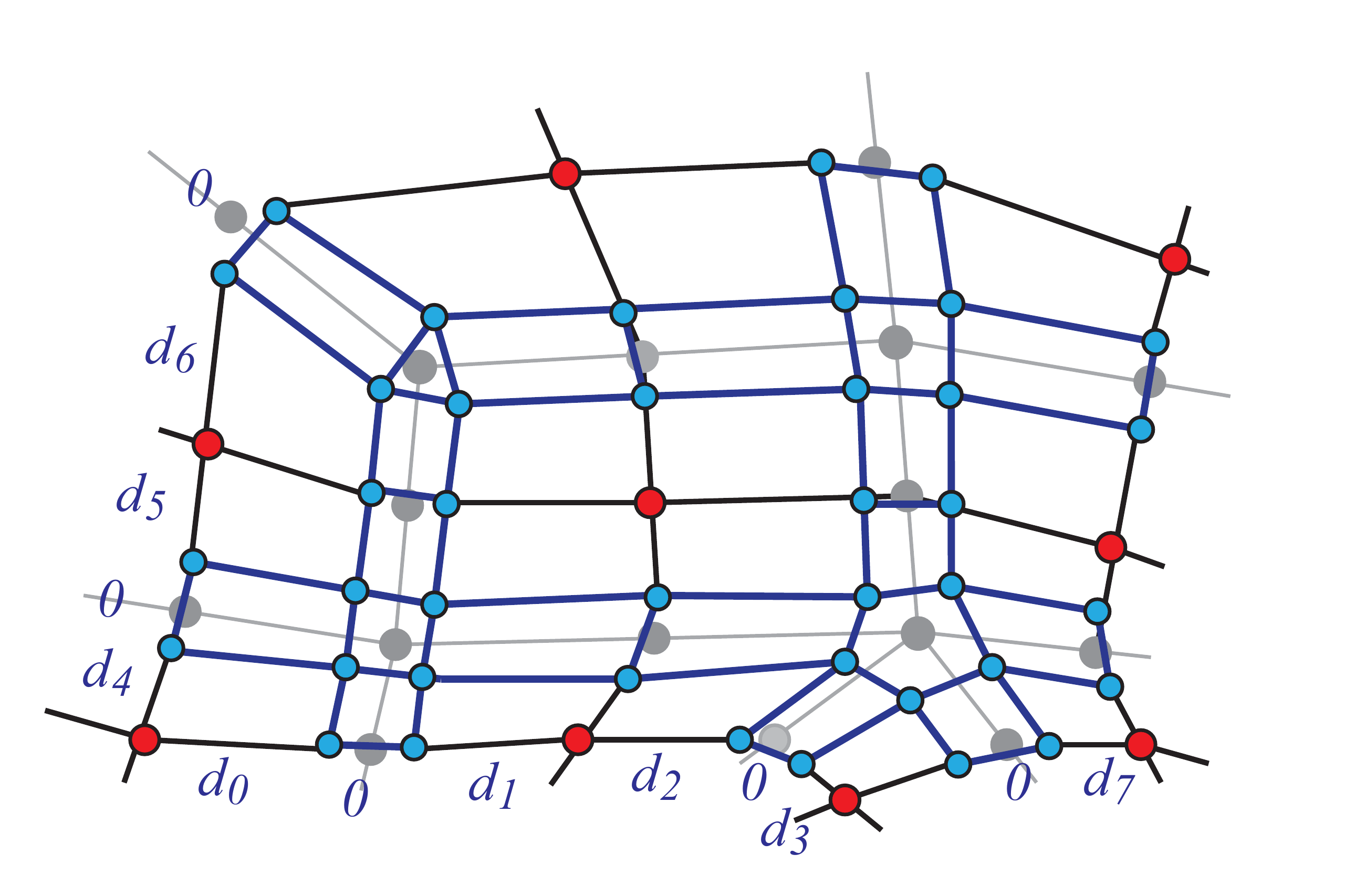}&
\includegraphics[width=0.45\textwidth]{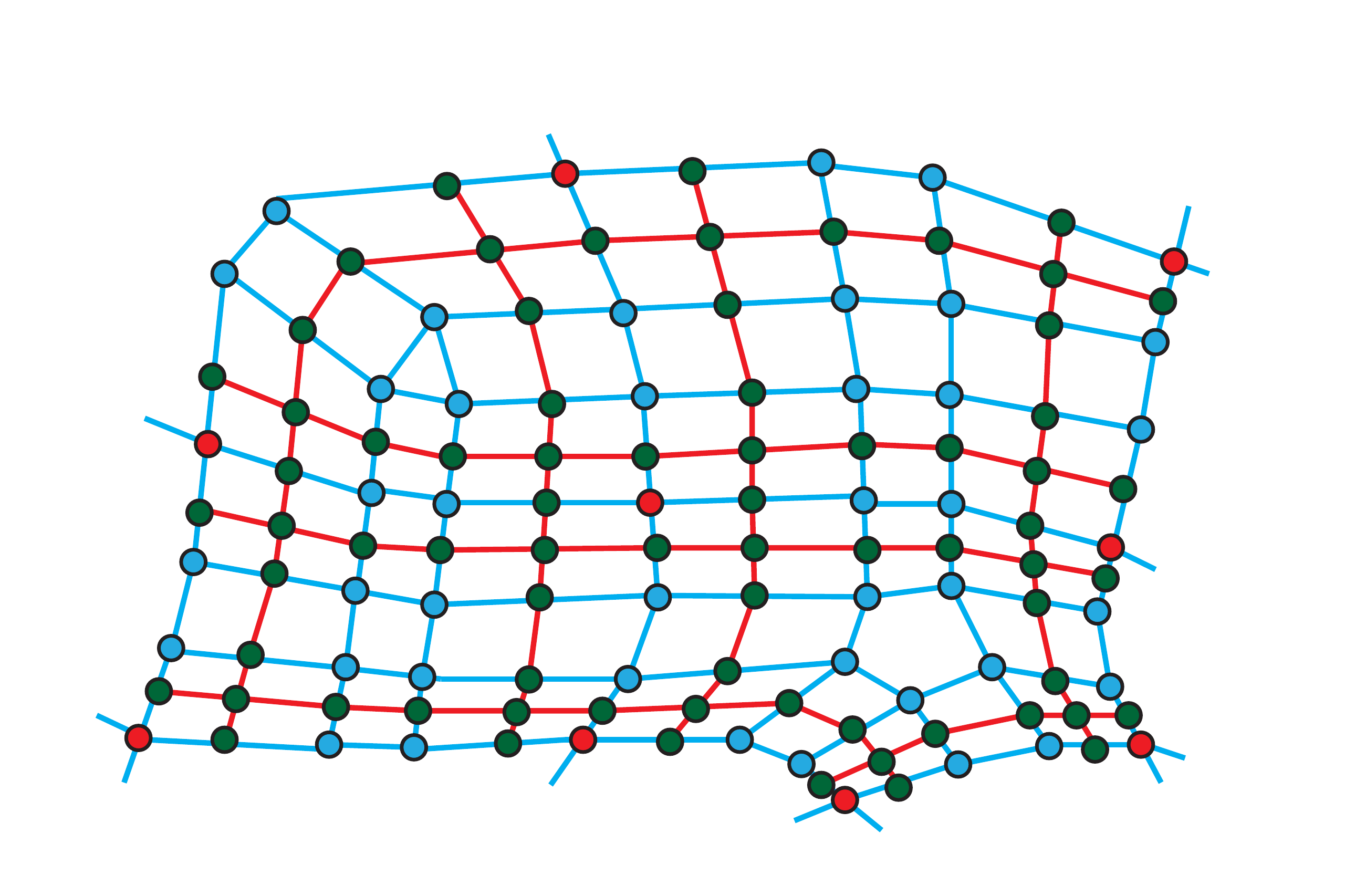}\\
(a) The first level & (b) The subsequent levels
\end{tabular}
\end{center}
\caption{The topological rules of tHNUS. (a) Converting the input quadrilateral mesh (light grey dots and lines) to a corresponding hybrid mesh (blue and red dots), and (b) refinement of the hybrid mesh in (a).}
\label{fig:topo}
\end{figure}

The topological rule for the subsequent levels is about how to split the initial hybrid mesh as in Figure~\ref{fig:topo}(a). Overall, every edge with a nonzero knot interval is split equally into two. As a result, a quadrilateral face is split into four or two subfaces, depending on the number of nonzero-knot-interval edges it has. All the polygonal faces\footnote{We refer to non-quadrilaterals as polygonal faces.} stay unchanged (topologically).

We next introduce the geometric rules of tHNUS, which again are divided into the first level and the subsequent levels. At the first level, the rule to update regular vertices is the same as NURBS refinement, whereas the rule to compute polygon vertices is derived such that the limit surface of tHNUS has the same limit point and tangent plane as that of the non-uniform subdivision via eigen-polyhedron~\cite{xin16}. We take the eigen-polyhedron-based subdivision as the reference because it shows demonstrated shape quality under non-uniform parameterization. However, the computation is rather complicated and there is no explicit formula available. Alternatively, a simple explicit rule was provided in~\cite{xin18}, where each polygon vertex is computed as a convex combination of neighboring vertices. However, this explicit rule does not guarantee shape quality. Note that the geometric rule for the first level plays a crucial role in determining shape quality, but it has nothing to do with the proof of surface $G^1$ continuity or the convergence performance in IGA.

\begin{figure}[htbp]
\begin{center}
\begin{tabular}{cc}
\includegraphics[width=0.45\textwidth]{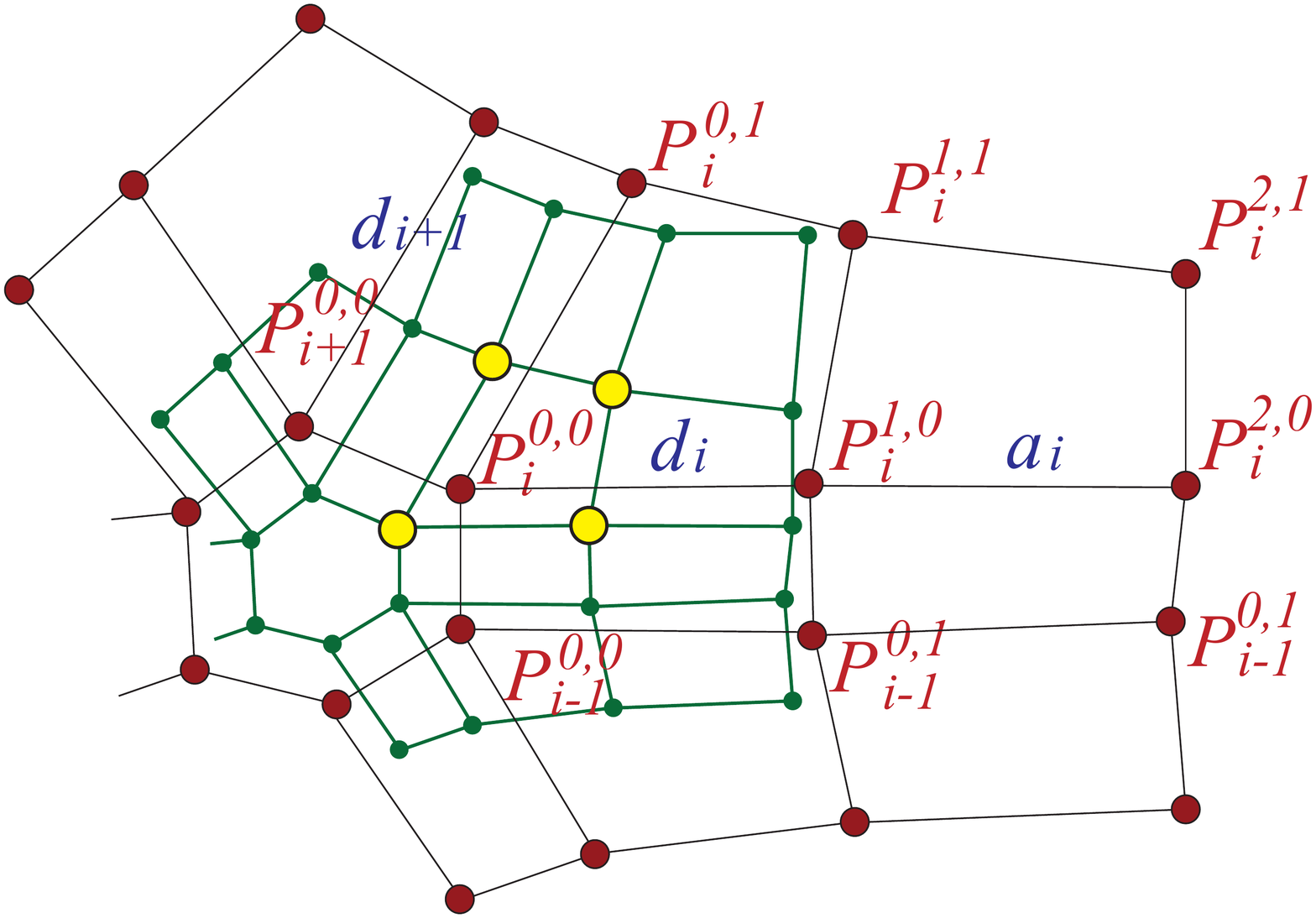}&
\includegraphics[width=0.45\textwidth]{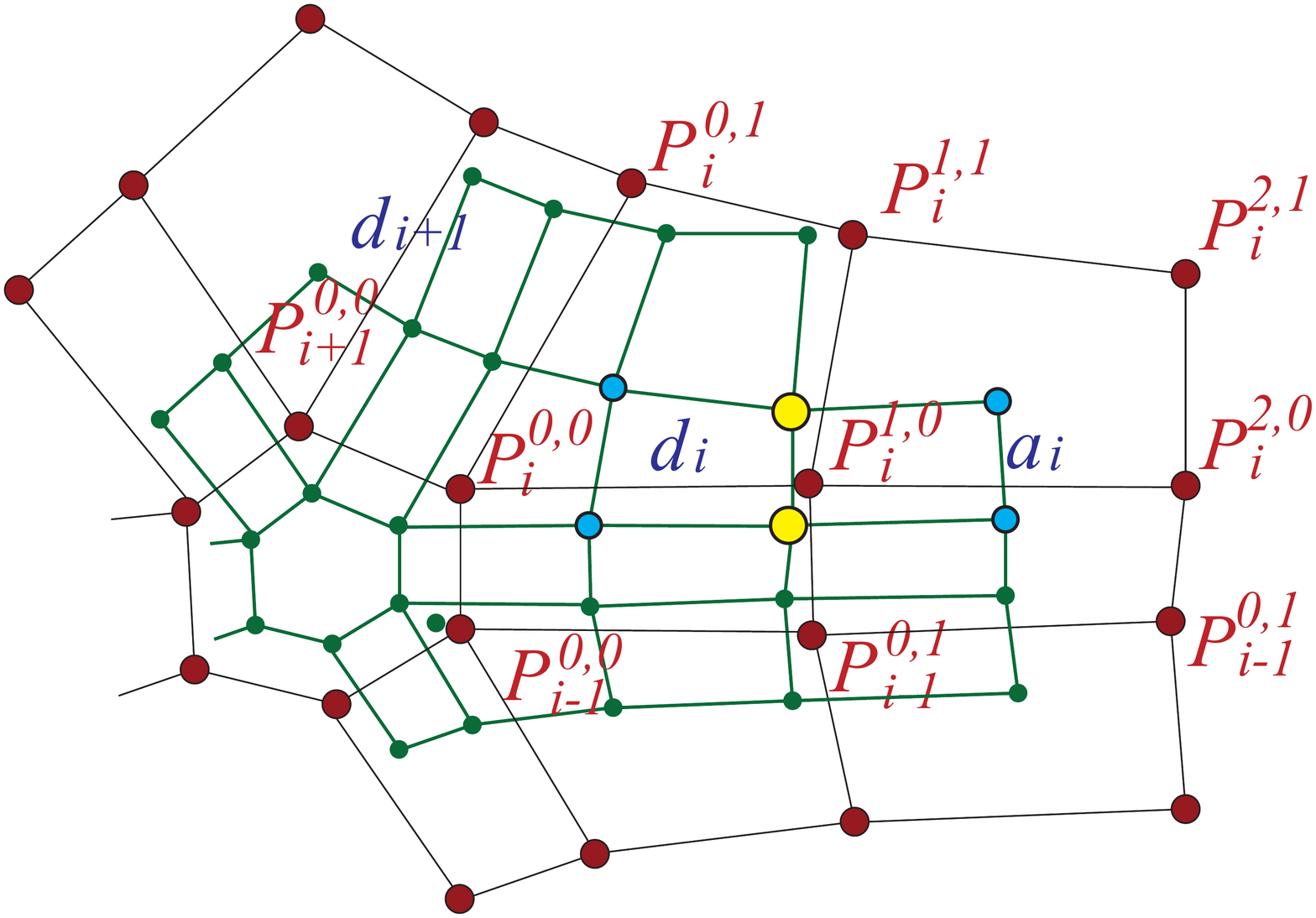}\\
(a) The first two-ring vertices & (b) The third-ring vertices
\end{tabular}
\end{center}
\caption{The geometric rule of tHNUS for the subsequent levels. (a) is the rule for the first two-ring
vertices, and (b) is the rule for the third-ring vertices.}
\label{fig:geom}
\end{figure}

Now we provide the geometric rule of tHNUS for the subsequent levels, which differs from HNUS in that there is an additional tuning parameter $\lambda$ in the formula to update polygon vertices. Referring to Figure~\ref{fig:geom} and given knot intervals $a_i$, $d_i$, the points $\overline{P}_{i}^{0, 0}$, $\overline{P}_{i}^{1, 0}$, $\overline{P}_{i}^{1, 1}$ and $\overline{P}_{i}^{0, 1}$ in the refined mesh are defined as
\begin{small}
\begin{equation}\label{eq:rule_irr}
\left\{
\begin{aligned}
\overline{P}_{i}^{0, 0} &= (1 - \lambda)C + \lambda P_{i}^{0, 0} + 2\lambda\alpha_{i}( -n P_{i}^{0, 0} + \sum_{j = 0}^{n-1}( 1 + 2\cos(\frac{2(j - i)\pi}{n}))P_{j}^{0, 0}), \\
\overline{P}_{i}^{1, 1} &= \frac{d_{i}d_{i+1}P_{i}^{1, 1} + d_{i}(d_{i+1}+2a_{i+1})P_{i}^{1, 0} +
d_{i+1}(d_{i}+2a_{i})P_{i}^{0, 1} + (d_{i}+2a_{i})(d_{i+1}+2a_{i+1})P_{i}^{0, 0}}{4(d_{i}+a_{i})(d_{i+1}+a_{i+1})}, \\
\overline{P}_{i}^{1, 0} &= \frac{d_{i}d_{i+1}P_{i-1}^{0, 1} + d_{i}(d_{i+1}+2d_{i-1})P_{i}^{1, 0} +
d_{i+1}(d_{i}+2a_{i})P_{i-1}^{0, 0} + (d_{i}+2a_{i})(d_{i+1}+2d_{i-1})P_{i}^{0, 0}}{4(d_{i-1}+d_{i+1})(d_{i}+a_{i})}, \\
\overline{P}_{i-1}^{0, 1} &= \frac{d_{i}d_{i-1}P_{i}^{1, 0} + d_{i}(2d_{i+1}+d_{i-1})P_{i-1}^{0, 1} +
d_{i-1}(d_{i}+2a_{i})P_{i}^{0, 0} + (d_{i}+2a_{i})(2d_{i+1}+d_{i-1})P_{i-1}^{0, 0}}{4(d_{i-1}+d_{i+1})(d_{i}+a_{i})},
\end{aligned}
\right.
\end{equation}
\end{small}
where $\alpha_{j} = \frac{1}{n}\frac{d_{j-1}d_{j+2}}{(d_{j-1}+d_{j+1})(d_{j} + d_{j+2})}$, then
\begin{equation}
C = \frac{\sum_{i=0}^{n-1}(d_{i}P_{i+1}^{0, 0} + d_{i+2}P_{i}^{0, 0})(d_{i-1}+d_{i+3})}{\sum_{j=0}^{n-1}(d_{j}+d_{j+2})(d_{j-1}+d_{j+3})} \doteq
\sum_{i=0}^{n-1} \beta_{i} P_{i}^{0, 0}.
\label{eq:rule_c}
\end{equation}
The remaining points are computed by the NURBS mid-knot insertion. For example,
\begin{align}
\overline{P}_{i}^{2, 0} &= \frac{a_{i}\overline{P}_{i}^{1, 0}}{2(d_{i} + a_{i})} + \frac{1}{4}\frac{(d_{i+1} + 2d_{i-1})P_{i}^{1, 0} + d_{i+1}P_{i-1}^{0, 1}}{d_{i+1}+d_{i-1}} + \frac{d_{i}\overline{P}_{i}^{3, 0}}{2(d_{i} + a_{i})}, \nonumber\\
\overline{P}_{i}^{2, 1} &= \frac{a_{i}\overline{P}_{i}^{1, 1}}{2(d_{i} + a_{i})} + \frac{1}{4}\frac{d_{i+1} P_{i}^{1, 1} + (d_{i+1} + 2a_{i+1})P_{i}^{1, 0}}{d_{i+1} + a_{i+1}} + \frac{d_{i}\overline{P}_{i}^{3, 1}}{2(d_{i} + a_{i})}, \nonumber\\
\overline{P}_{i}^{2, 2} &=  \frac{P_{i}^{1, 1}}{4} + \frac{a_{i}a_{i+1}\overline{P}_{i}^{1, 1} + d_{i}a_{i+1}\overline{P}_{i}^{3, 1} + a_{i}d_{i+1}\overline{P}_{i}^{1, 3} + d_{i}d_{i+1}\overline{P}_{i}^{3, 3} }{4(d_{i}+a_{i})(d_{i+1}+a_{i+1})} \nonumber \\
&+ \frac{a_{i+1}}{4(d_{i+1}+a_{i+1})}M_{1} + \frac{d_{i+1}}{4(d_{i+1}+a_{i+1})}M_{3} + \frac{d_{i}}{4(d_{i}+a_{i})}M_{2} + \frac{a_{i}}{4(d_{i}+a_{i})}M_{4},
\label{eq:rule_reg}
\end{align}
where
\begin{align}
M_{1} &= \frac{d_{i+1} P_{i}^{1, 1} + (d_{i+1} + 2a_{i+1})P_{i}^{1, 0}}{2(d_{i+1} + a_{i+1})},
&M_{2} &= \frac{P_{i}^{1, 1} + P_{i}^{2, 1}}{2}, \nonumber\\
M_{4} &= \frac{P_{i}^{1, 1} + (d_{i} + 2a_{i})P_{i}^{0, 1}}{2(d_{i} + a_{i})},
&M_{3} &= \frac{P_{i}^{1, 1} + P_{i}^{1, 2}}{2}.
\label{eq:rule_regm}
\end{align}

\begin{myremark}
All the computations of these points are the same as those in HNUS except for polygon vertices $\overline{P}_{i}^{0, 0}$, where the tuning parameter $\lambda$ is introduced to control the size of shrinkage in the updated polygon: the smaller $\lambda$ is, the more the polygon shrinks. As a result, isoparametric lines become more concentrated around extraordinary vertices. We will see examples in Section \ref{sec:result}. When $\lambda=\frac{1}{2}$, the computation of $\overline{P}_{i}^{0, 0}$ coincides with that in HNUS, and thus tHNUS is equivalent to HNUS in this particular case. However, the generalization via introducing $\lambda$ is not as straightforward as it appears. The key insight is that $\lambda$ turns out to be the subdominant eigenvalues (i.e., the 2nd and 3rd eigenvalues) of the subdivision matrix in tHNUS. As has been reported in \cite{ref:ma19}, convergence behavior in a subdivision scheme is mostly influenced by the subdominant eigenvalues. Therefore, tuning $\lambda$ is equivalent to ``controlling" convergence. We will have more detailed discussion about how $\lambda$ improves convergence with specific examples in Section \ref{sec:result}.
\end{myremark}

Tuned subdivision is a well studied subject aiming to optimized subdivision stencils (i.e., coefficients in the subdivision matrix) to improve certain properties of a subdivision scheme, for example, to minimize curvature variations to achieve a better surface fairness \cite{ref:halstead93, ref:kobbelt96, ref:augsdorfer06}. Recently, it has been explored in the context of IGA to improve accuracy \cite{ref:qzhang18} as well as convergence \cite{ref:ma19}. In particular, the tuned Catmull-Clark subdivision \cite{ref:ma19} is the first work in IGA that is able to use a subdivision scheme to achieve optimal convergence rates (in the $L^2$-norm error by solving the Poisson's equation). However, the optimization framework proposed in \cite{ref:ma19} only works for uniform parameterization and cannot be extended to non-uniform subdivision schemes because the subdivision stencils in a uniform subdivision scheme like Catmull-Clark only depend on the valence of a given extraordinary vertex, and the optimization can be focused on a finite number of stencils of interest. Thus, optimization only needs to be done once and the optimized stencils can be stored for future use. On the other hand, the subdivision stencils in a non-uniform subdivision depend on not only the valence of an extraordinary vertex, but also the surrounding knot intervals, leading to infinite possible cases of stencils. Therefore, it is not feasible to apply optimization to non-uniform subdivision because otherwise it would be very time-consuming and also problem specific.

\begin{myremark}
We introduce $\lambda$ \emph{explicitly} to the formula of $\overline{P}_{i}^{0, 0}$. Note that $\lambda$ is a single parameter for all situations. It is independent of the valence of extraordinary vertices and the choice of knot intervals. However, this is only possible when bounded curvature is not of primary interest. tHNUS generally does not have bounded curvature. Nonetheless, bounded curvature under non-uniform parameterization remains an open problem and may not be available at all.
\end{myremark}

\section{Proof of continuity}
\label{sec:proof}

In order to prove tHNUS surfaces to be $G^1$ continuous, we need to prove that the spectrum of the subdivision matrix satisfies certain constraints and the associated characteristic map is regular and injective. Note that introducing $\lambda$ to HNUS indeed significantly complicates the proof of $G^1$ continuity. Referring to Figure~\ref{fig:subm} for the notations, the subdivision rule can be written into the following equations since the neighbor knot intervals $a_{i}$ equals to $d_{i}$ with enough subdivision levels.

\begin{figure}[htbp]
\begin{center}
\includegraphics[width=0.45\textwidth]{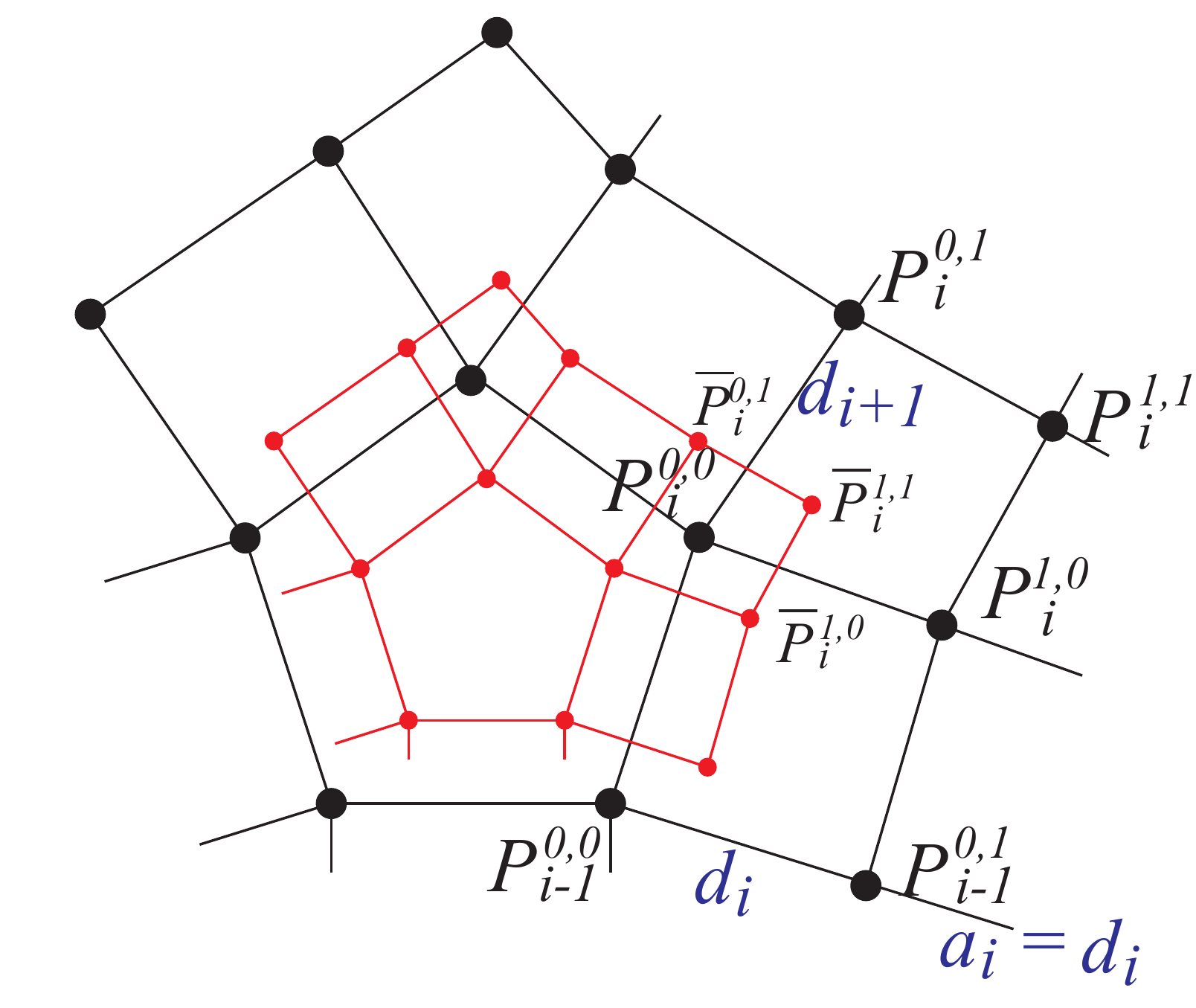}
\end{center}
\caption{The notations to define the subdivision matrix around a polygonal face. \label{fig:subm}}
\end{figure}
\begin{small}
\begin{equation}
\left\{
\begin{aligned}
\overline{P}_{j}^{0, 0} &= (1 - \lambda)C + \lambda P_{j}^{0, 0} + 2\lambda\alpha_{j} \left[ -n P_{j}^{0, 0} + \sum_{i = 0}^{n-1} ( 1 + 2\cos(\frac{2(j-i)\pi}{n}))P_{i}^{0, 0} \right], \\
\overline{P}_{j}^{1, 0} &= \frac{3(2d_{j-1}+d_{j+1})}{8(d_{j-1}+d_{j+1})}P_{j}^{0, 0} + \frac{3d_{j+1}}{8(d_{j-1}+d_{j+1})}P_{j-1}^{0, 0} +
\frac{(2d_{j-1}+ d_{j+1})}{8(d_{j-1}+d_{j+1})}P_{j}^{1, 0} + \frac{d_{j+1}}{8(d_{j-1}+d_{j+1})}P_{j-1}^{0, 1} , \\
\overline{P}_{j-1}^{0, 1} &= \frac{3d_{j-1}}{8(d_{j-1}+d_{j+1})}P_{j}^{0, 0} + \frac{3(d_{j-1}+2d_{j+1})}{8(d_{j-1}+d_{j+1})}P_{j-1}^{0, 0} +
\frac{d_{j-1}}{8(d_{j-1}+d_{j+1})}P_{j}^{1, 0} + \frac{(d_{j-1}+2d_{j+1})}{8(d_{j-1}+d_{j+1})}P_{j-1}^{0, 1} , \\
\overline{P}_{j}^{1, 1} &= \frac{9}{16}P_{j}^{0, 0} + \frac{3}{16}P_{j}^{0, 0} + \frac{3}{16}P_{j}^{0, 0} + \frac{1}{16}P_{j}^{1, 1}.\\
\end{aligned}
\right.
\end{equation}
\end{small}
We arrange them in a matrix form $\overline{M} = S_{n} M$, i.e.,
\begin{equation}
\left(
\begin{array}{c}
  \overline{P}_{0}^{0, 0} \\
  \vdots \\
  \overline{P}_{n-1}^{0, 0} \\
  \overline{P}_{0}^{1, 0} \\
  \vdots \\
  \overline{P}_{n-1}^{0, 1} \\
  \overline{P}_{0}^{1, 1} \\
  \vdots \\
  \overline{P}_{n-1}^{1, 1}
\end{array}
\right)
= \left(
    \begin{array}{ccccccccc}
       &  &  &  &  &  &  &  &  \\
       & Q_{n} &  &  & 0 &  &  & 0 &  \\
       &   &  &  &  &  &  &  &  \\
       &   &  & E_{0} & \ldots & 0 &  &  &  \\
       & \divideontimes &  & \vdots & \ddots & 0 &  & 0 &  \\
       &   &  & 0 & 0 & E_{n-1} &  &  &  \\
       &   &  &  &  &  & \frac{1}{16} & 0 & 0 \\
       & \divideontimes &  &  & \divideontimes &  & \vdots & \ddots &  \\
       &   &  &  &  &  & 0 & 0 & \frac{1}{16} \\
    \end{array}
  \right)
\left(
  \begin{array}{c}
  P_{0}^{0, 0} \\
  \vdots \\
  P_{n-1}^{0, 0} \\
  P_{0}^{01, 0} \\
  \vdots \\
  P_{n-1}^{0, 1} \\
  P_{0}^{1, 1} \\
  \vdots \\
  P_{n-1}^{1, 1}
\end{array}
\right) .
\end{equation}
Denote $Q_{n} = (Q_{i, j})$, where $i, j \in \{0,1,\dots, n-1\}$, and then we have
$$Q_{i, j} = \left\{
               \begin{array}{ll}
                 (1-\lambda)\beta_{j} + 2\left( 1 + 2\cos\left(\frac{2(j-i)\pi}{n}\right)\right)\lambda\alpha_{i}, & \hbox{$j \neq i$} \\
                 \lambda + (1-\lambda)\beta_{i}- 2(n-3)\lambda\alpha_{i}, & \hbox{$j = i$}
               \end{array}
             \right.
$$
and
$$E_{j} = \left(
             \begin{array}{cc}
               \frac{2d_{j-1}+ d_{j+1}}{8(d_{j-1}+d_{j+1})} & \frac{d_{j+1}}{8(d_{j-1}+d_{j+1})} \\
               \frac{d_{j-1}}{8(d_{j-1}+d_{j+1})} & \frac{d_{j-1}+2d_{j+1}}{8(d_{j-1}+d_{j+1})} \\
             \end{array}
           \right).
$$

\begin{lemma}
\label{lemma:eigen}
Given an extraordinary vertex of any valence and an arbitrary choice of positive knot intervals, the eigenvalues of $Q_{n}$ satisfy
\begin{equation}
\lambda_{1} = 1 > \lambda_{2} = \lambda_{3} = \lambda > |\lambda_{k}|, k = 4, 5, \dots, n.
\end{equation}
\end{lemma}
\begin{proof}We use the discrete Fourier transform to compute the eigenvalues of $Q_{n}$. Let $p_{k}$ and $\overline{p}_{k}$ ($k = 0, \dots, n-1$) be the Fourier vectors corresponding to $P_{j}$ and $\overline{P}_{j}$, respectively, i.e.,
\begin{eqnarray}
&p_{k} = \frac{1}{n}\sum_{j = 0}^{n-1}P_{j}^{0, 0}\overline{\omega}^{jk}, &\overline{p}_{k} = \frac{1}{n}\sum_{j = 0}^{n-1}\overline{P}_{j}^{0, 0}\overline{\omega}^{jk}, \\
&P_{k}^{0, 0} = \sum_{j = 0}^{n-1}p_{j}\omega^{jk}, &\overline{P}_{k}^{0, 0} = \sum_{j = 0}^{n-1}\overline{p}_{j}\omega^{jk},
\end{eqnarray}
where $\omega = e^{\frac{2\pi}{n}}$ and $\overline{\omega} = e^{-\frac{2\pi}{n}}$. Now the subdivision rule can be formulated in terms of the Fourier vectors,
\begin{equation}
\sum_{k=0}^{n-1}\overline{p}_{k}\omega^{jk} = \sum_{k=0}^{n-1}\left(\sum_{j=0}^{n-1}(1-\lambda)\beta_{j} \omega^{jk}\right)p_{k}
+ p_{0} + \lambda\omega^{j}p_{1} + \lambda\omega^{j(n-1)}p_{n-1} + 2\lambda\left(\frac{1}{2} - n\alpha_{j}\right)\sum_{k=2}^{n-2}p_{k}\omega^{jk}.
\end{equation}
Using the inverse discrete Fourier transform, we obtain
\begin{equation}
\left(\begin{array}{c}
  \overline{p}_{0}  \\
  \overline{p}_{1}  \\
  \vdots \\
  \overline{p}_{n-1}
\end{array}
\right) = \left(
            \begin{array}{cccc}
              1 & (1-\lambda)\beta_{1} & \cdots & (1-\lambda)\beta_{n-1}\\
              0 & \lambda & \divideontimes & 0 \\
              0 & 0 & B_{n-3} & 0 \\
              0 & 0 & \divideontimes & \lambda \\
            \end{array}
          \right)
\left(\begin{array}{c}
  p_{0}  \\
  p_{1}  \\
  \vdots \\
  p_{n-1}
\end{array}
\right) ,
\end{equation}
where
\begin{equation}
B_{n - 3} = \lambda I - 2\lambda\left(
                             \begin{array}{cccc}
                               \sum_{j=0}^{n-1}\alpha_{j} & \sum_{j=0}^{n-1}\alpha_{j}\omega^{j} & \ldots & \sum_{j=0}^{n-1}\alpha_{j}\omega^{(n-4)j}\\
                               \sum_{j=0}^{n-1}\alpha_{j}\omega^{(n-1)j} & \sum_{j=0}^{n-1}\alpha_{j} & \cdots & \sum_{j=0}^{n-1}\alpha_{j}\omega^{(n-5)j} \\
                               \vdots & \vdots & \ddots & \vdots \\
                               \sum_{j=0}^{n-1}\alpha_{j}\omega^{4j} & \sum_{j=0}^{n-1}\alpha_{j}\omega^{5j} & \cdots & \sum_{j=0}^{n-1}\alpha_{j} \\
                             \end{array}
                           \right) =: \lambda I - 2\lambda G_n.
\end{equation}
Similar to~\cite{xin18}, both $G_{n}$ and $I - G_{n}$ are positive definite. Denote $\lambda_{B, i}$ the eigenvalues of $B_{n-3}$. As $G_{n}$ is positive definite, $\frac{1}{2}I - \frac{1}{2\lambda} B_{n-3}$ ($=G_n$) is also a positive definite matrix, which means that $\lambda_{B, i} < \lambda$. On the other hand, $I - G_{n}$ is a positive definite matrix as well because $\mu_{k} < 1$, and equivalently, $I - (\frac{1}{2}I - \frac{1}{2\lambda}B_{n-3})$ is positive definite, which means that $\lambda_{B, i} > -\lambda$. Therefore, we complete the proof.
\end{proof}

\begin{lemma}
\label{lemma:eigen2}
Given an extraordinary vertex of any valence and an arbitrary choice of positive knot intervals, if $\lambda > \frac{1}{4}$, then the eigenvalues of $S_{n}$ satisfy
\begin{equation}
\lambda_{1} = 1 > \lambda_{2} = \lambda_{3} = \lambda > |\lambda_{k}|, \mbox{where }k = 4, 5, \dots, 4n.
\end{equation}
\end{lemma}
\begin{proof}
The eigenvalues of $S_n$ consist of those of $Q_{n}$, $E_{i}$ and $\frac{1}{16}I_{n}$, where $I_n$ is an $n\times n$ identity matrix. As proved in Lemma~\ref{lemma:eigen}, the first three eigenvalues of $Q_{n}$ are $1, \lambda, \lambda$, and the remaining ones are less than $\lambda$. $\frac{1}{16}I_{n}$ has $n$ equal eigenvalues $\frac{1}{16}$ ($< \lambda$). It is also straightforward to verify that the eigenvalues of the $2\times2$ matrix $E_{i}$ are $\frac{1}{4}$ ($< \lambda$) and $\frac{1}{8}$ ($< \lambda$). Therefore, we conclude that the eigenvalues of $S_{n}$ are
\begin{equation}
\lambda_{1} = 1 > \lambda_{2} = \lambda_{3} = \lambda > |\lambda_{k}|, \mbox{where }k = 4, 5, \dots, 4n.
\end{equation}
\end{proof}

The next step is to compute the characteristic map and prove that it is regular and injective. We first prove the following lemma.
\begin{lemma}
\label{lemma:char0}
Let $P_{i} = \left(\cos(\frac{2i\pi}{n}), \sin(\frac{2i\pi}{n})\right) \in \mathbb{R}^2$ ($i = 0, \dots, n-1$), $C_{P} = \sum_{i=0}^{n-1}\beta_{i}P_{i}$,
and ${P}$ be an $n\times2$ vector containing all $P_{i}$, i.e., $P = [P_{0}, P_{1}, \dots, P_{n-1}]^{T}$. Then we have
\begin{equation}
S_{n}({P} - C_{P}) = \lambda({P} - C_{P}).
\end{equation}
\end{lemma}
\begin{proof}
Denote $\overline{{P}} = S_{n}{P}$, and we can obtain
\begin{small}
\begin{align*}
\overline{P}_{j} - C_{P}=& \lambda( P_{j} - C_{P}) + 2\lambda\alpha_{j}\left[ -n \left(\cos(\frac{2j\pi}{n}), \sin(\frac{2j\pi}{n})\right) + \sum_{i = 0}^{n-1} \left( 1 + 2\cos(\frac{2(j-i)\pi}{n})\right)\left(\cos(\frac{2i\pi}{n}), \sin(\frac{2i\pi}{n})\right)\right] \nonumber\\
=& \lambda( P_{j} - C_{P}) + 2\lambda\alpha_{j}\left[ -n \left(\cos(\frac{2j\pi}{n}), \sin(\frac{2j\pi}{n})\right) + \sum_{i = 0}^{n-1}2\cos\left(\frac{2(j-i)\pi}{n}\right)\left(\cos(\frac{2i\pi}{n}), \sin(\frac{2i\pi}{n})\right)\right] \nonumber\\
=& \lambda( P_{j} - C_{P}) + 2\lambda\alpha_{j}[ -n (\cos(\frac{2j\pi}{n}), \sin(\frac{2j\pi}{n})) + \nonumber\\
&\sum_{i = 0}^{n-1}(\cos(\frac{2j\pi}{n}) +
\cos(\frac{2(j-2i)\pi}{n}), \sin(\frac{2j\pi}{n}) - \sin(\frac{2(j-2i)\pi}{n}))] \nonumber\\
=& \lambda( P_{j} - C_{P}) .
\end{align*}
\end{small}
Since the above equation holds for any $0 \leq j \leq n-1$, we conclude
\begin{equation}
S_{n}({P} - C_{P}) = \lambda({P} - C_{P}) .
\end{equation}
\end{proof}
\begin{figure}[htbp]
\begin{center}
\begin{tabular}{cc}
\includegraphics[width=0.45\textwidth]{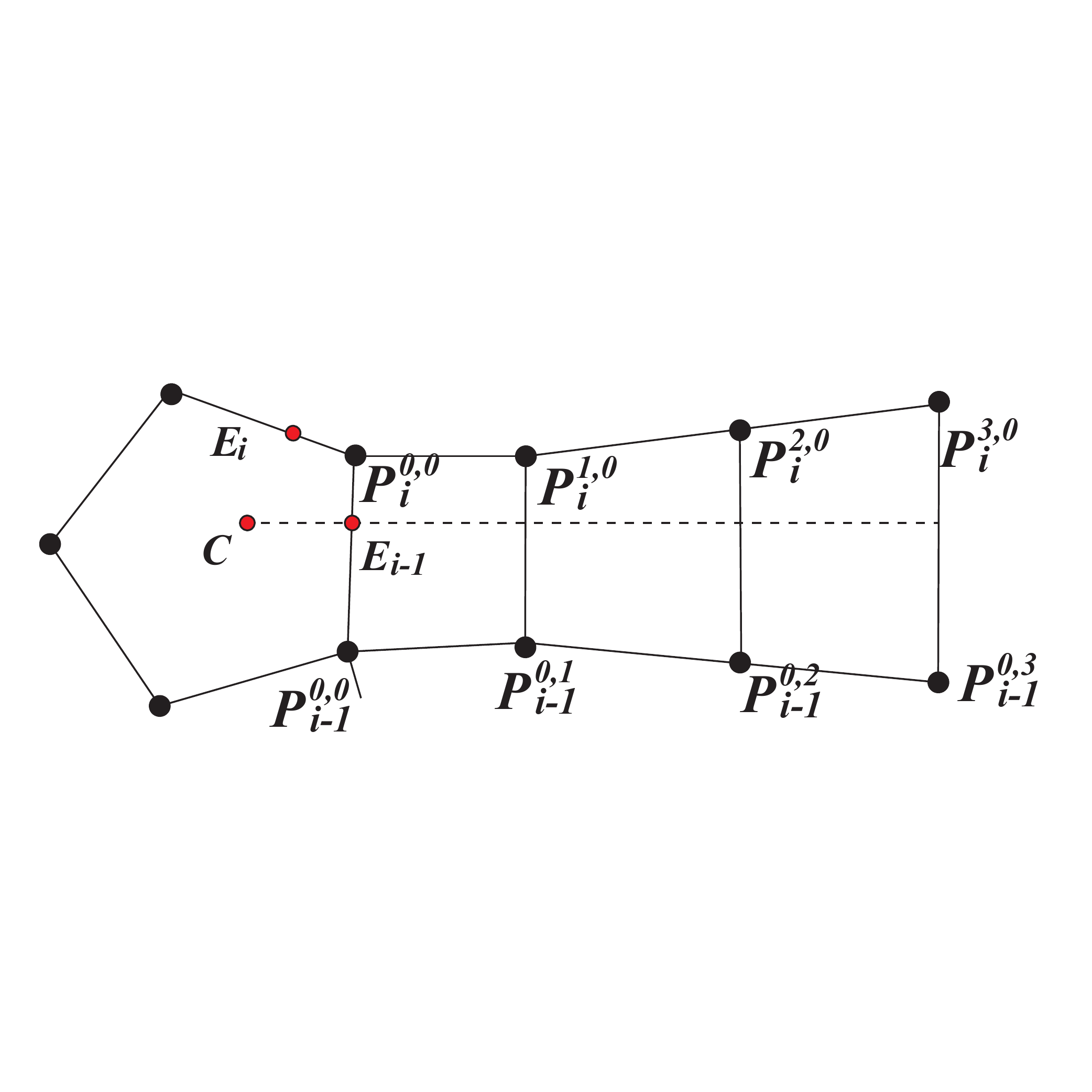}&
\includegraphics[width=0.35\textwidth]{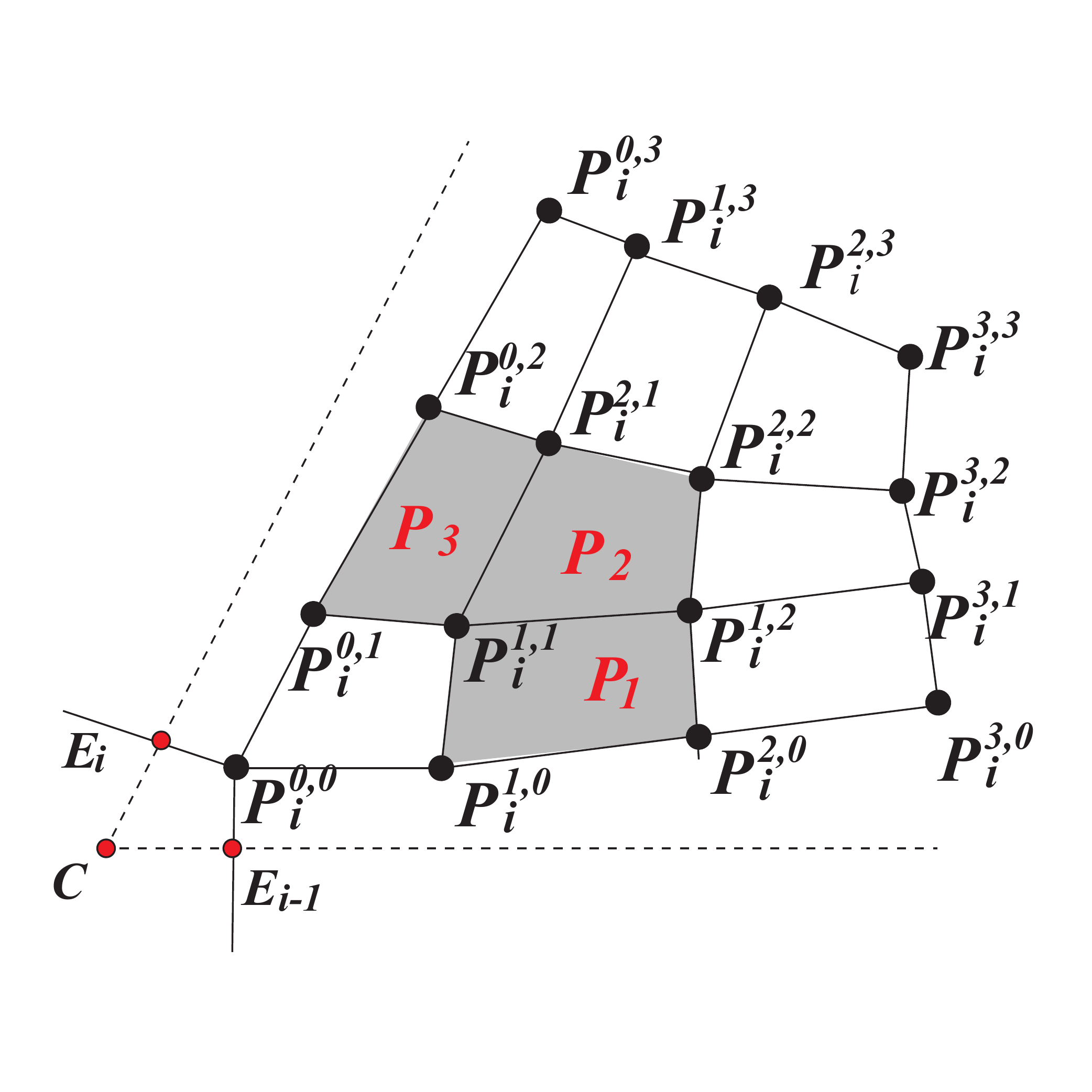}\\
(a) Control points $P_{i}^{j, k}$, $j = 0$ or $k = 0$ & (b) The other control points
\end{tabular}
\end{center}
\caption{The control points of the characteristic map of tHNUS. (a) shows the control points $P_{i}^{0j}$ and $P_{i}^{j0}$
while (b) shows the rest of the control points of the characteristic map.\label{fig:character}}
\end{figure}
\begin{lemma}
\label{lemma:char}
The characteristic map of tHNUS is regular and injective for any valence extraordinary vertices and any positive knot intervals if $\lambda\in (\frac{1}{4}, 1)$.
\end{lemma}
\begin{proof}
To prove that the characteristic map is regular and injective, we need a $4\times4$ grid of control points. We first compute the coordinates of this control grid that is used to define the characteristic map. The key idea is based on the fact that applying subdivision to the control grid of a characteristic map is equivalent to scaling the control grid by $\lambda$.

Referring to Figure~\ref{fig:character}, we have control points $P_{i}^{j, k}$, where $0 \leq j, k \leq 3$ ($ 0 \leq i \leq n-1$). According to Lemma~\ref{lemma:char0}, if
we let $P_{i}^{0, 0} = (\cos(\frac{2i\pi}{n}), \sin(\frac{2i\pi}{n})) \in \mathbb{R}^2$, $C = \sum_{i=0}^{n-1}\beta_{i}P_{i}^{0, 0}$, then we have
$S_{n}[P_{0}^{0, 0}-C, \dots, P_{n-1}^{0, 0}-C]^{T} = \lambda[P_{0}^{0, 0}-C, \dots, P_{n-1}^{0, 0}-C]^{T}$.

Further let $E_{i} = \frac{d_{i}}{d_{i}+d_{i+2}}P_{i+1}^{0,0}+\frac{d_{i+2}}{d_{i}+d_{i+2}}P_{i}^{0,0}$, $p = P_{i}^{0,0} - C$, $v = E_{i-1} - C$ and $w = E_{i} - C$.
By definition, we have
\begin{align*}
\frac{1}{4}(\frac{d_{i+1}+2d_{i-1}}{2d_{i+1}+2d_{i-1}}(P_{i}^{1,0} - P_{i}^{0,0})+
\frac{d_{i+1}}{2d_{i+1}+2d_{i-1}}(P_{i-1}^{0,1} - P_{i-1}^{0,0})) + \frac{1}{2}(E_{i-1}-C) &= \lambda(P_{i}^{1,0} - P_{i}^{0,0}), \\
\frac{1}{4}(\frac{d_{i-1}}{2d_{i+1}+2d_{i-1}}(P_{i}^{1,0} - P_{i}^{0,0})+
\frac{d_{i-1}+2d_{i+1}}{2d_{i+1}+2d_{i-1}}(P_{i-1}^{0,1} - P_{i-1}^{0,0})) + \frac{1}{2}(E_{i-1}-C) &= \lambda(P_{i-1}^{0,1} - P_{i-1}^{0,0}).
\end{align*}
Solving the linear systems, we obtain
\begin{equation}
P_{i}^{1, 0} - P_{i}^{0, 0} = P_{i-1}^{0, 1} - P_{i-1}^{0, 0} = \frac{4(1-\lambda)}{4\lambda-1}v + \frac{4(1-2\lambda)}{8\lambda-1}(p - v) .
\end{equation}
Similarly, we compute $P_{i}^{2, 0}$, $P_{i}^{3, 0}$, $P_{i-1}^{0, 3}$, $P_{i-1}^{0, 3}$ as follows,
\begin{align*}
P_{i}^{2, 0} - P_{i}^{1, 0} &= \frac{18(1-\lambda)}{(8\lambda-1)(4\lambda-1)}v + \frac{18(1-2\lambda)}{(16\lambda-1)(8\lambda-1)}(p-v) ,\\
P_{i}^{3, 0} - P_{i}^{2, 0} &= \frac{6(1-\lambda)(1+\lambda)}{(8\lambda-1)\lambda(4\lambda-1)}v + \frac{3(1-4\lambda^{2})}{(16\lambda-1)\lambda(8\lambda-1)}(p-v) .
\end{align*}
We can also compute the remaining control points $P_{i}^{j, k}$ ($1 \leq j, k \leq 3$), whose coefficients are complex expressions in $\lambda$. The detailed expressions are given in the Appendix.

With all these control points, we can now extract the B\'{e}zier control points for patches $P_{1}$, $P_{2}$ and $P_{3}$; see Figure~\ref{fig:character}(b). For example, in the patch $P_{2}$, let $B_{2}^{j, k}$ ($j, k = 0, \dots, 3$) be the $4\times4$ B\'{e}zier control points. We denote $S_{2}^{j, k} = B_{2}^{j+1, k} - B_{2}^{j, k}$ and $T_{2}^{j, k} = B_{2}^{j, k+1} - B_{2}^{j, k}$. All $S_{2}^{j, k}$ and $T_{2}^{j, k}$ can be written as linear combinations of $p$, $v$ and $w$, where the coefficients are again complex expressions in $\lambda$; see Appendix. We further plot some of these coefficients as functions of $\lambda\in (\frac{1}{4}, 1)$; see Figures~\ref{fig:coe} and~\ref{fig:coe1}. We observe that $S_{2}^{j, k}$ are convex combinations of vectors $p$, $v$ and $-w$, while $T_{2}^{j, k}$ are convex combinations of $p$, $-v$ and $w$. Moreover, $C$ is a convex combination of the points $E_{i}$ from Equation~\eqref{eq:rule_c}, so the patch $P_{2}$ is regular and injective. As a result, all the control points $P_{i}^{j, k}$ ($0 \leq j, k \leq 3$) lie in the region bounded by two rays $CE_{i-1}$ and $CE_{i}$, which means that any two different patches must not intersect with one another. Similar results can also be achieved for patches $P_1$ and $P_3$. Therefore, the characteristic map of tHNUS is regular and injective for any $\lambda\in (\frac{1}{4}, 1)$, any valence extraordinary vertices and any positive knot intervals.

\begin{figure}[htbp]
\begin{center}
\begin{tabular}{cccc}
\includegraphics[width=0.22\textwidth]{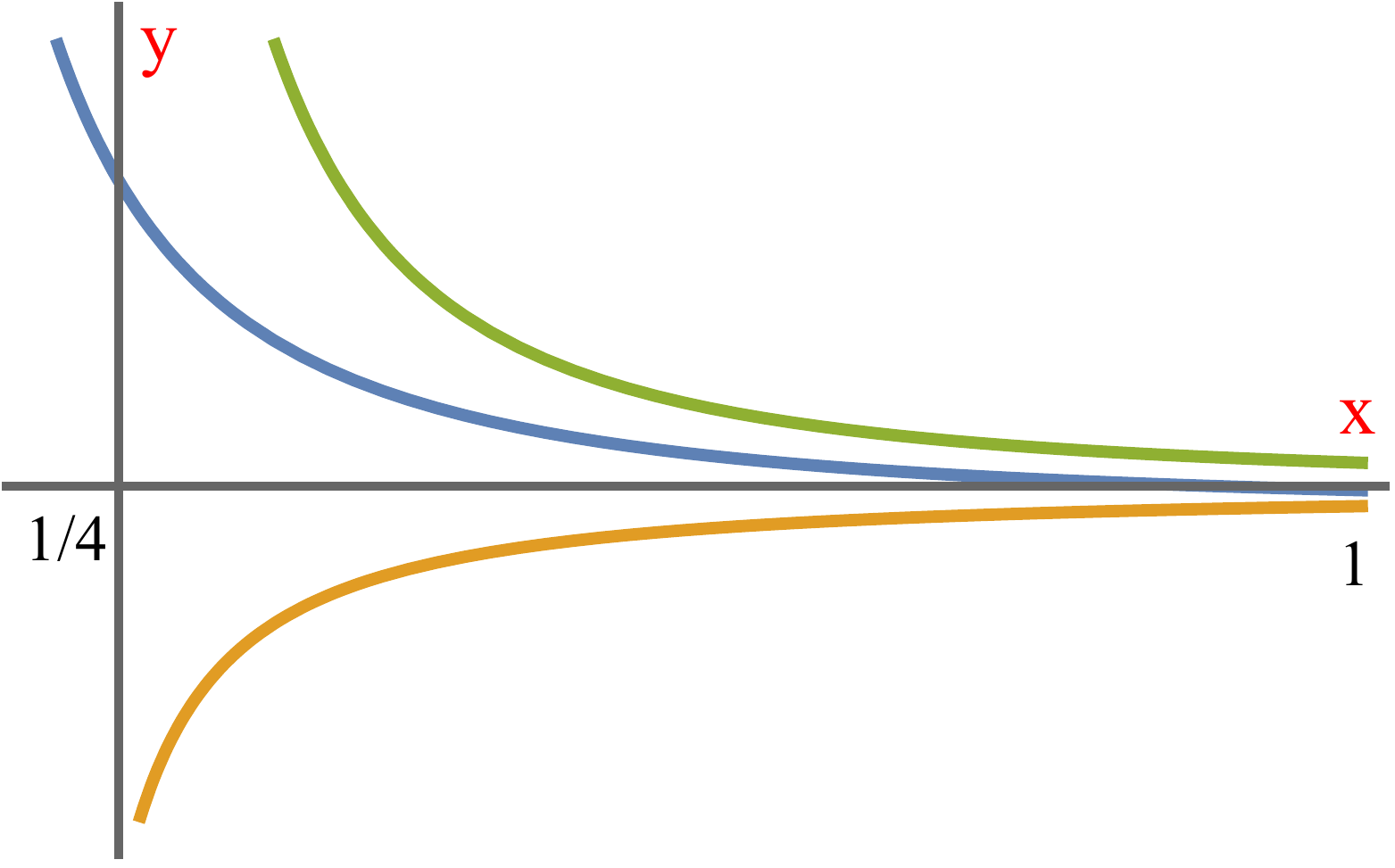}&
\includegraphics[width=0.22\textwidth]{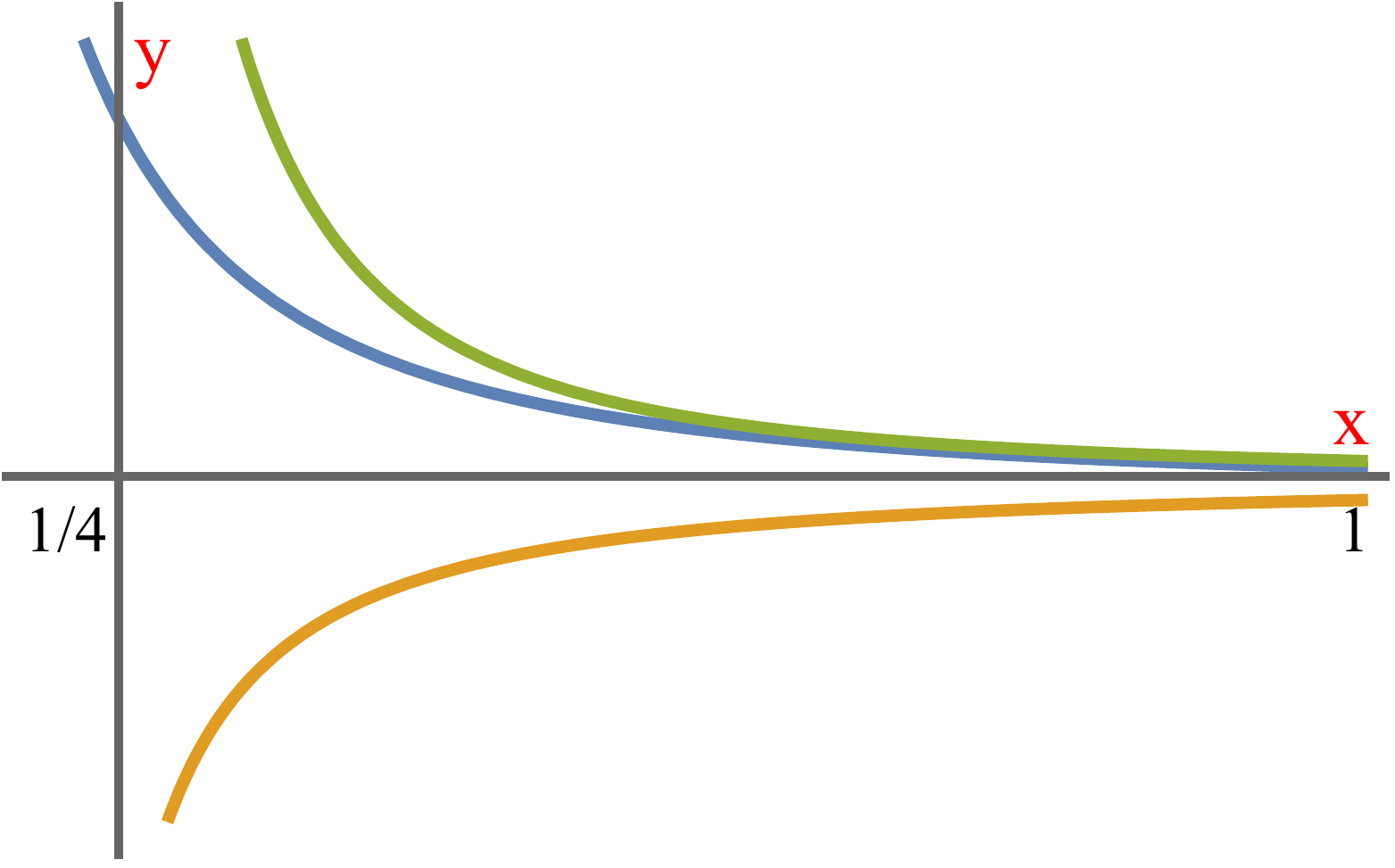}&
\includegraphics[width=0.22\textwidth]{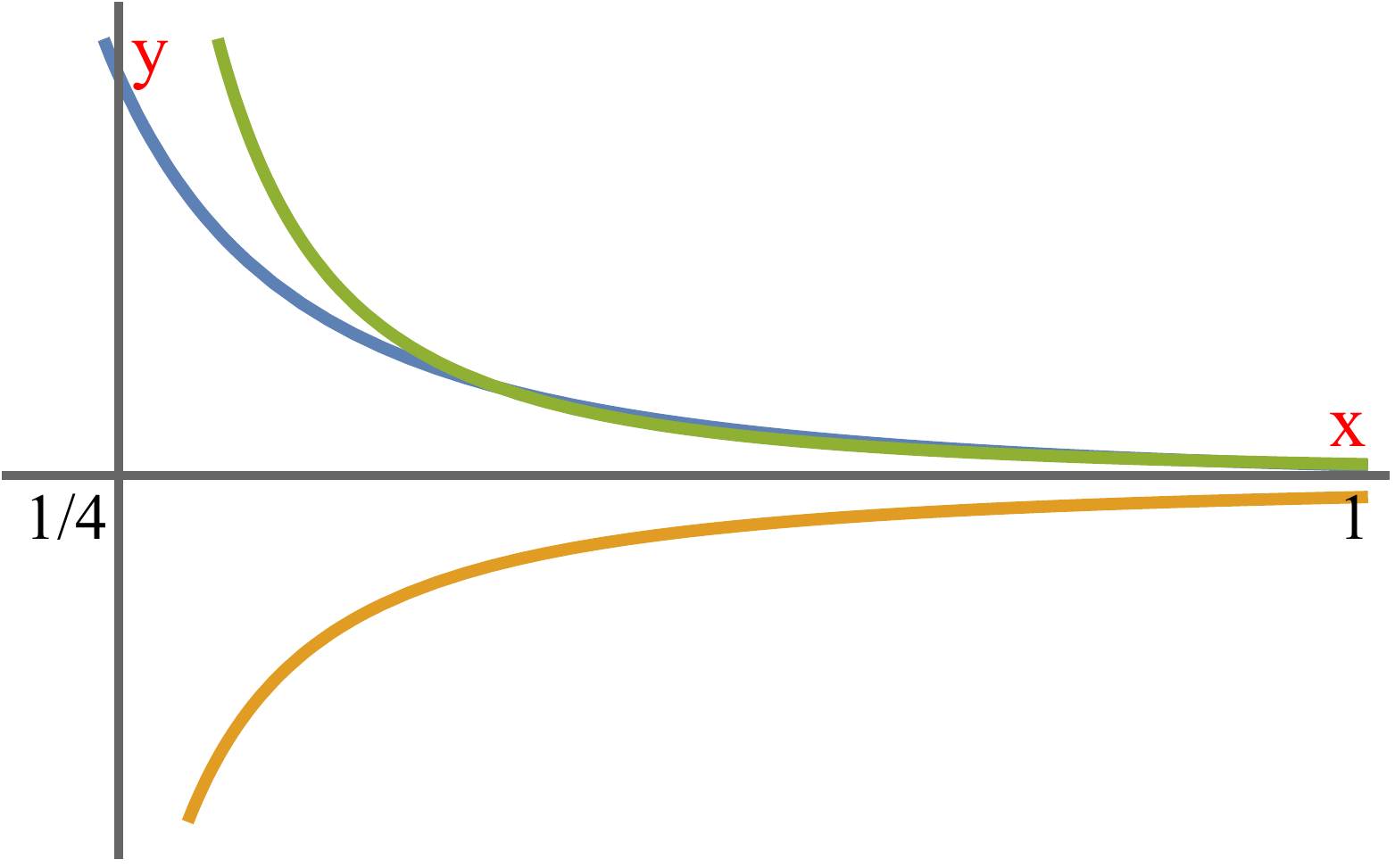}&
\includegraphics[width=0.22\textwidth]{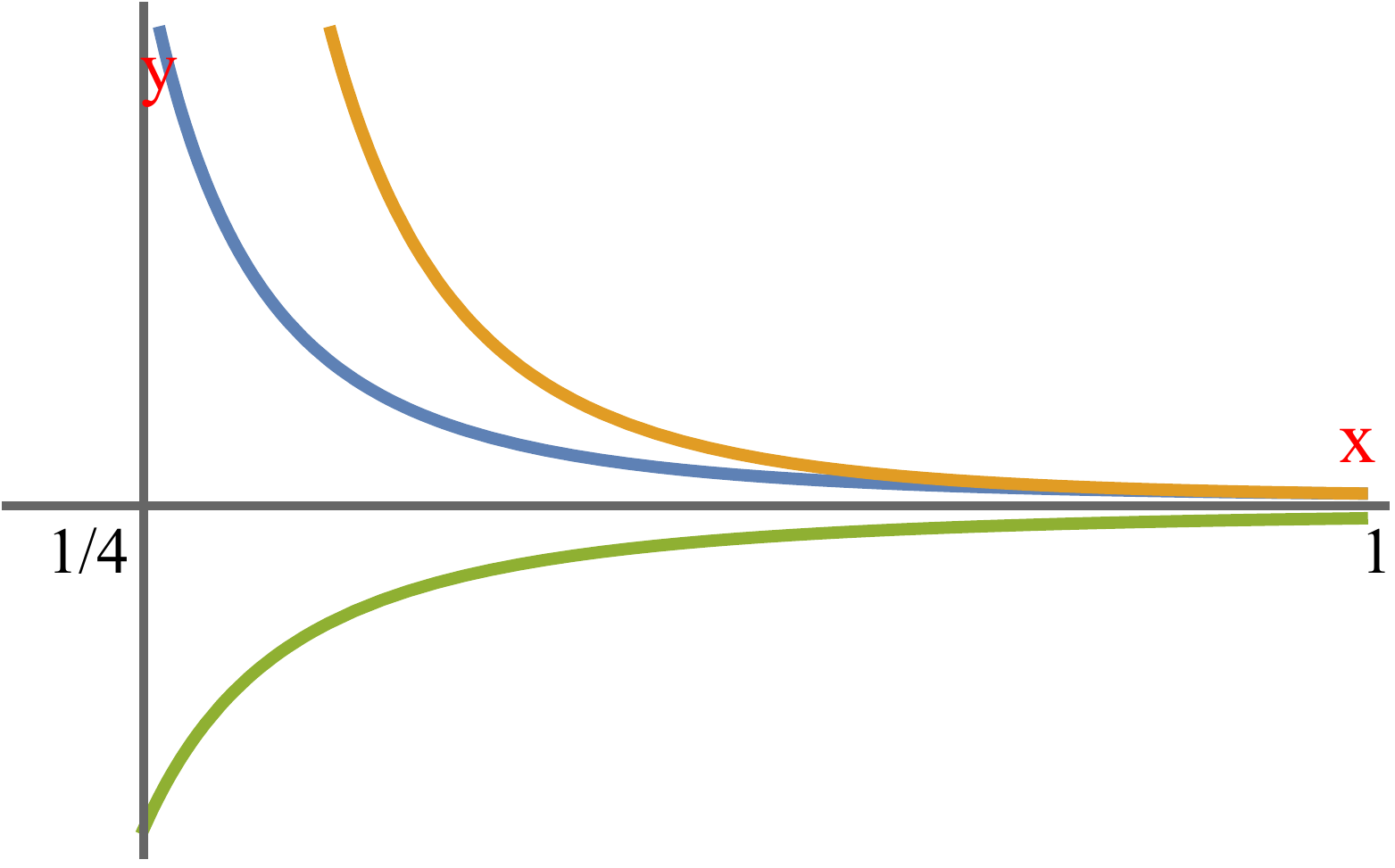}\\
$S_{2}^{0,0}$
&$S_{2}^{1,1}$
&$S_{2}^{2,2}$
&$S_{2}^{2,3}$
\end{tabular}
\end{center}
\caption{The plots of the coefficients of $S_{2}^{j, k}$ in terms of $\lambda\in (\frac{1}{4}, 1)$, where the x-axis represents $\lambda$ and y-axis represents
the value of the coefficients. The green, blue and orange lines represent coefficients corresponding to $v$, $p$ and $w$, respectively. Each $S_{2}^{j, k}$ is a convex combination of $p$, $v$ and $-w$. \label{fig:coe}}
\end{figure}

\begin{figure}[htbp]
\begin{center}
\begin{tabular}{cccc}
\includegraphics[width=0.22\textwidth]{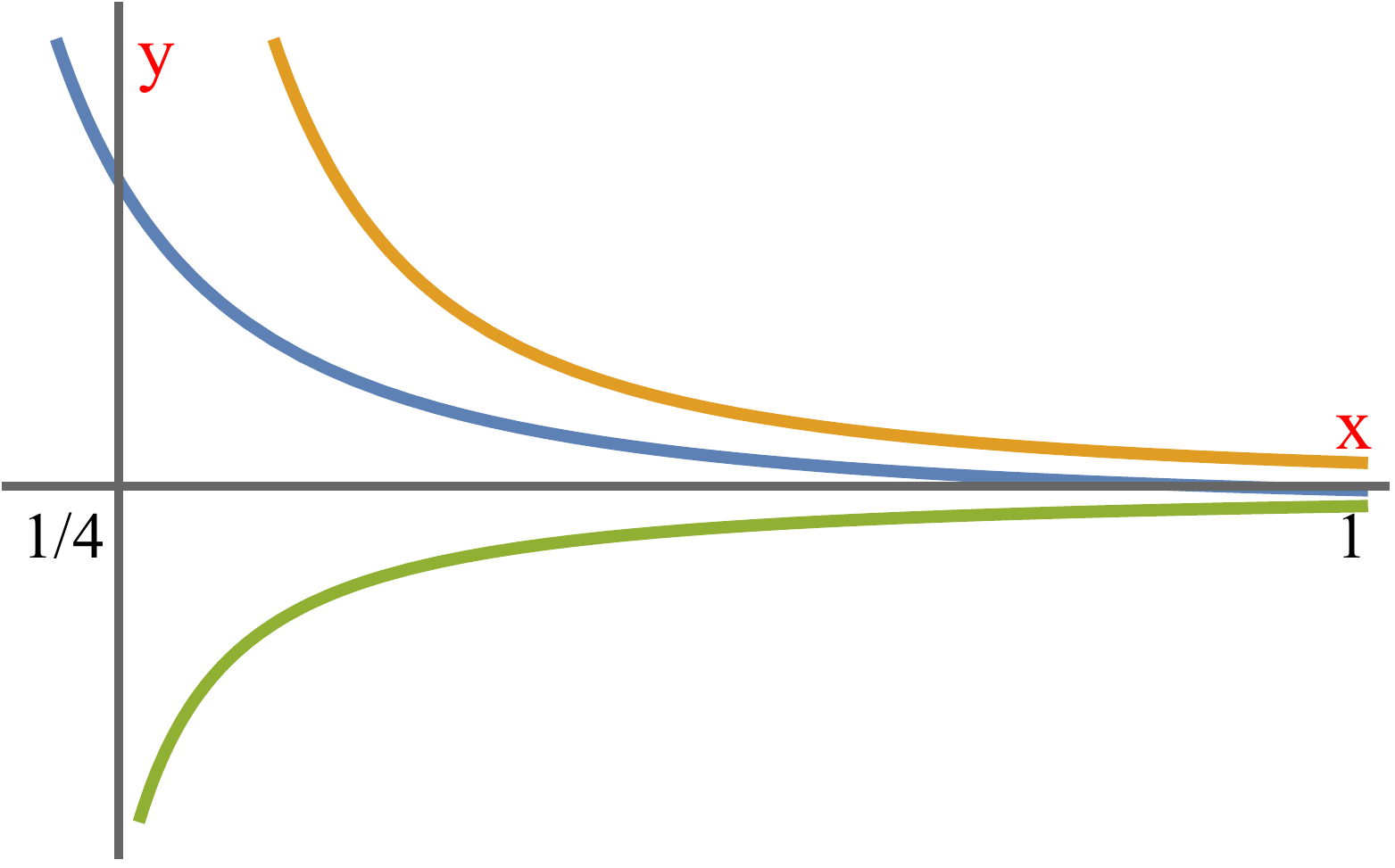}&
\includegraphics[width=0.22\textwidth]{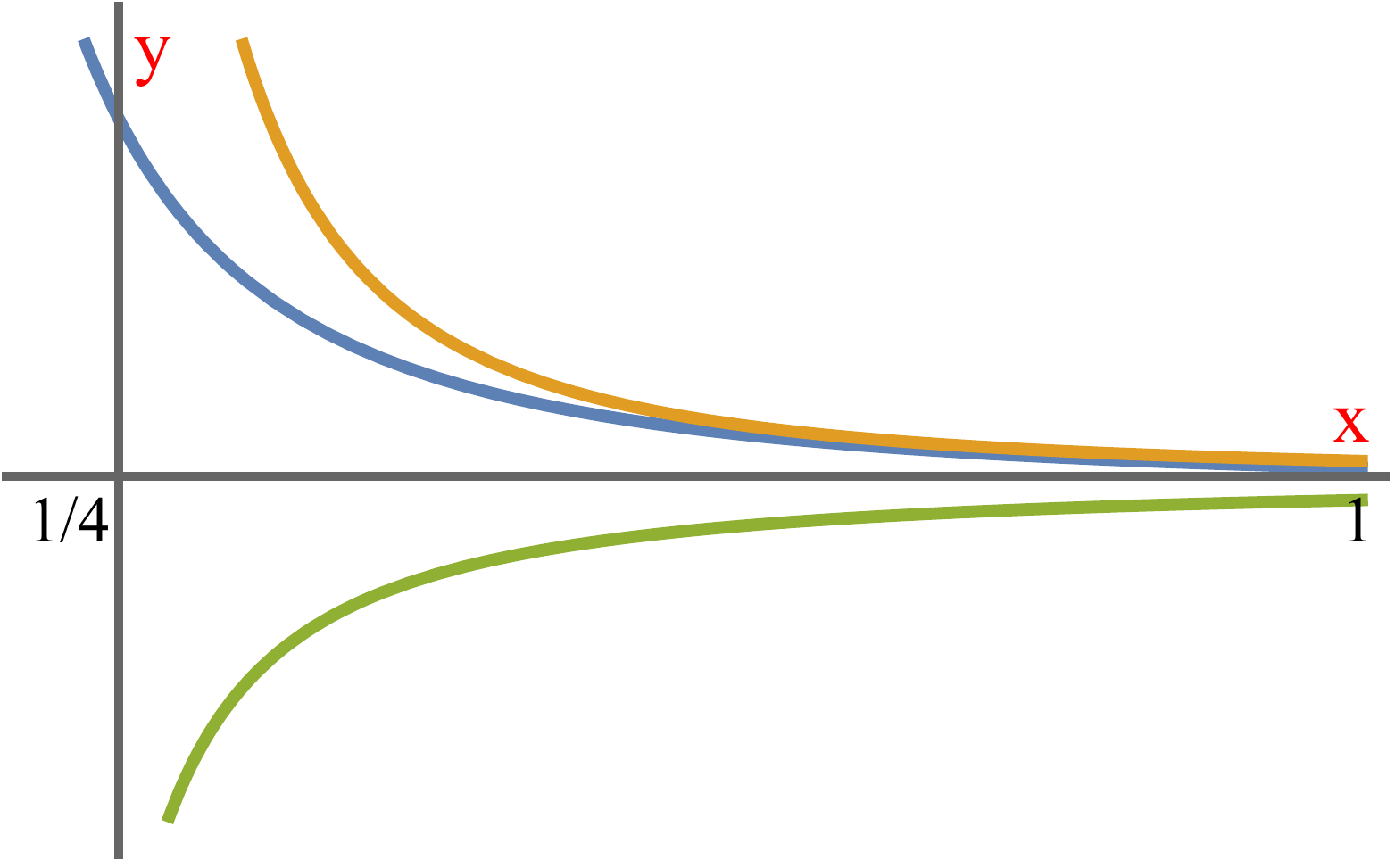}&
\includegraphics[width=0.22\textwidth]{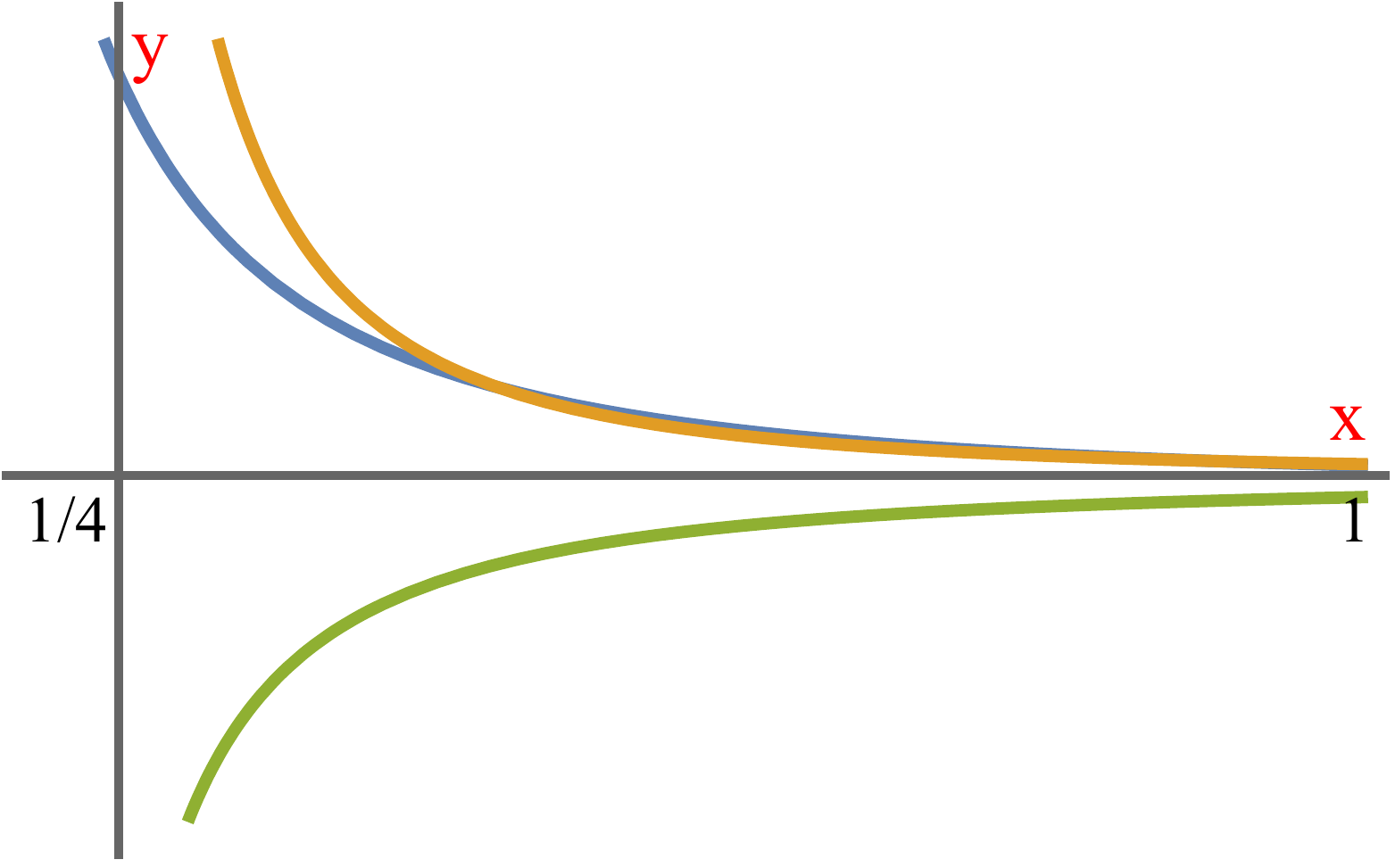}&
\includegraphics[width=0.22\textwidth]{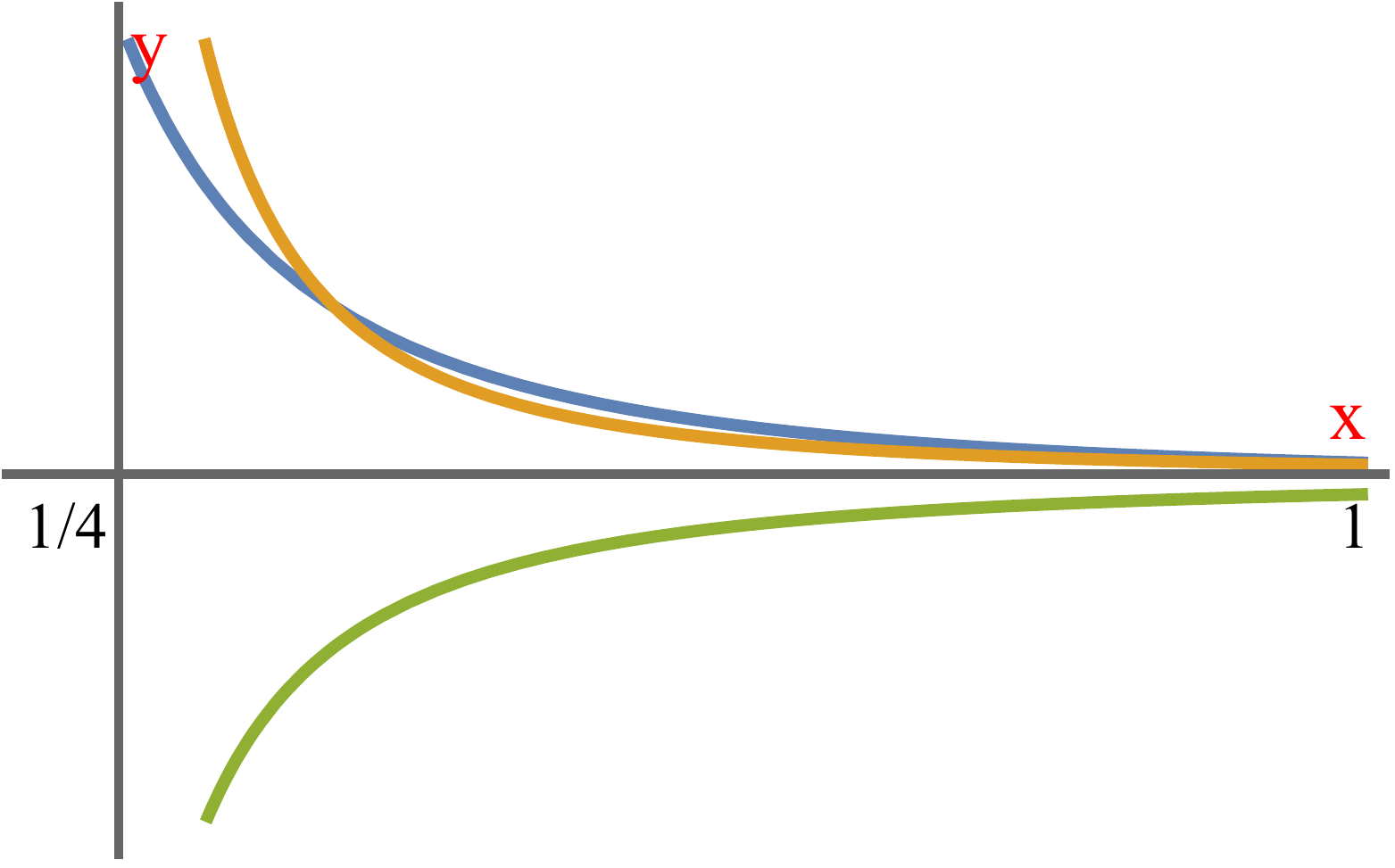}\\
$T_{2}^{0,0}$
& $T_{2}^{1,0}$
&$T_{2}^{2,0}$
&$T_{2}^{3,0}$
\end{tabular}
\end{center}
\caption{The plots of the coefficients of $T_{2}^{j, k}$ in terms of $\lambda\in (\frac{1}{4}, 1)$, where the x-axis represents $\lambda$ and y-axis represents
the value of the coefficients. The green, blue and orange lines represent coefficients corresponding to $v$, $p$ and $w$, respectively. Each $T_{2}^{j, k}$ is a convex combination of $p$, $-v$ and $w$. \label{fig:coe1}}
\end{figure}
\end{proof}

\begin{theorem}
\label{the:char}
Given an arbitrary 2-manifold control mesh with any choice of positive knot intervals and any $\lambda\in (\frac{1}{4}, 1)$, the corresponding tHNUS limit surface is globally $G^1$-continuous.
\end{theorem}
\begin{proof}
The theorem is a direct result of Lemma~\ref{lemma:eigen}, Lemma~\ref{lemma:eigen2} and Lemma~\ref{lemma:char}.
\end{proof}

\section{Hybrid subdivision basis functions}
\label{sec:basis}

In this section, we introduce basis functions of hybrid non-uniform subdivision. The derivation of such subdivision functions essentially follows Stam's method for Catmull-Clark subdivision \cite{ref:stam98}. However, there are two major differences. First, Catmull-Clark basis functions are associated with the input quadrilateral control mesh, whereas tHNUS basis functions are associated with the hybrid control mesh; see Figure \ref{fig:topo}(a). Second, Catmull-Clark subdivision features uniform knot intervals everywhere, leading to a subdivision matrix that only depends on the valence of a particular extraordinary vertex. In contrast, tHNUS (or HNUS) supports general non-uniform knot intervals, so the subdivision matrix depends not only on the valence of the extraordinary vertex, but also on the surrounding knot intervals.

\subsection{Definition of basis functions}
\label{sec:defbf}

We now introduce how tHNUS basis functions are defined on a hybrid control mesh. We start with distinguishing different types of faces. Recall that there exists both quadrilateral and polygonal faces in the hybrid mesh, and each edge in a polygonal face is assigned with a zero knot interval by construction. The knot intervals of other edges inherit from the input quadrilateral mesh and are constrained by the assumption that opposite edges in a quadrilateral face have the same knot interval. Moreover, note that edges perpendicular to the boundary also have zero knot intervals to make use of open knot vectors. An example of the knot interval configuration is shown in Figure \ref{fig:meshterm}(b), where the hybrid mesh is obtained from the input mesh in Figure \ref{fig:meshterm}(a).

\begin{figure}[htb]
\centering
\begin{tabular}{ccc}
\includegraphics[width=0.3\linewidth]{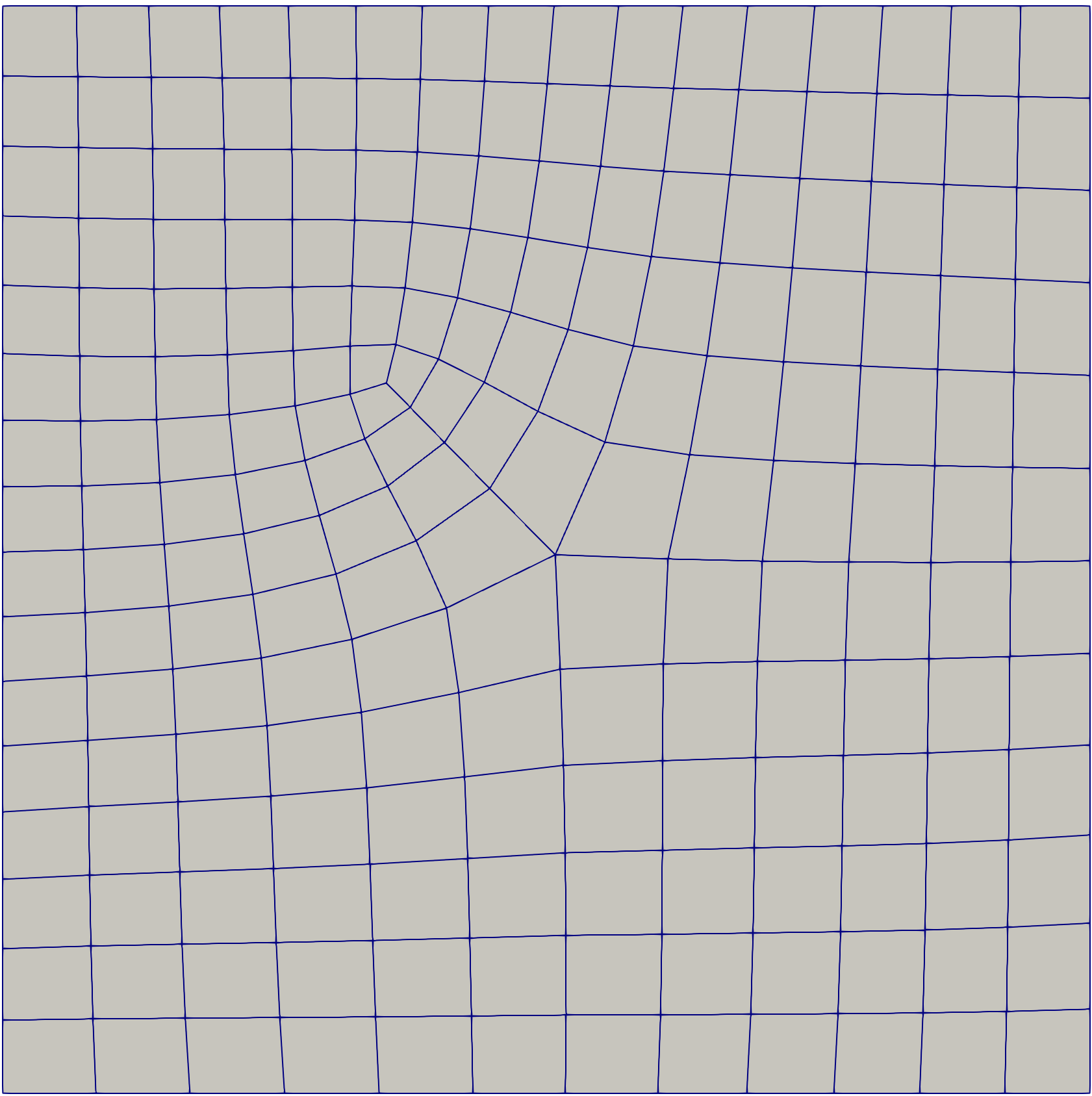}&
\includegraphics[width=0.3\linewidth]{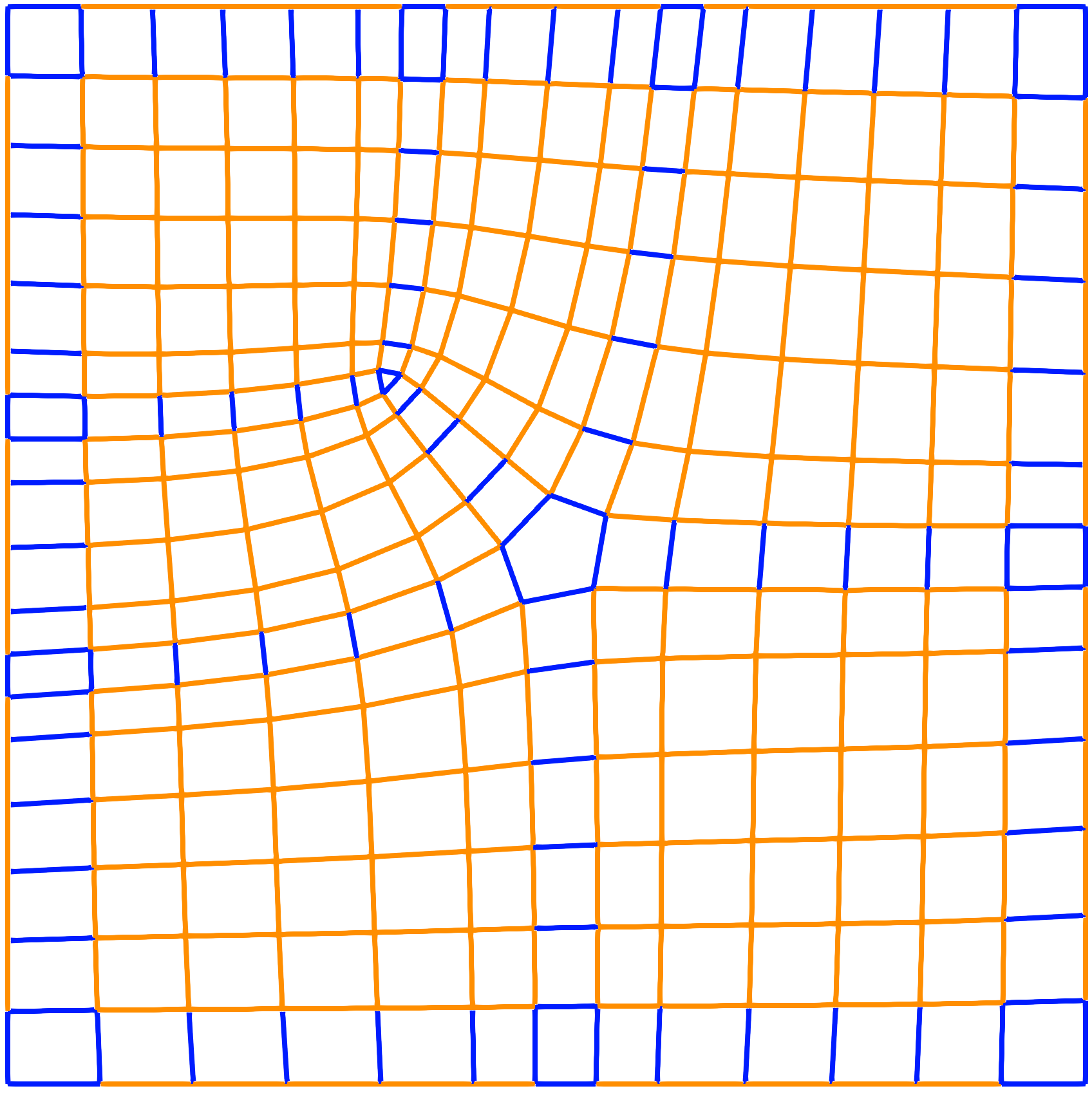}&
\includegraphics[width=0.3\linewidth]{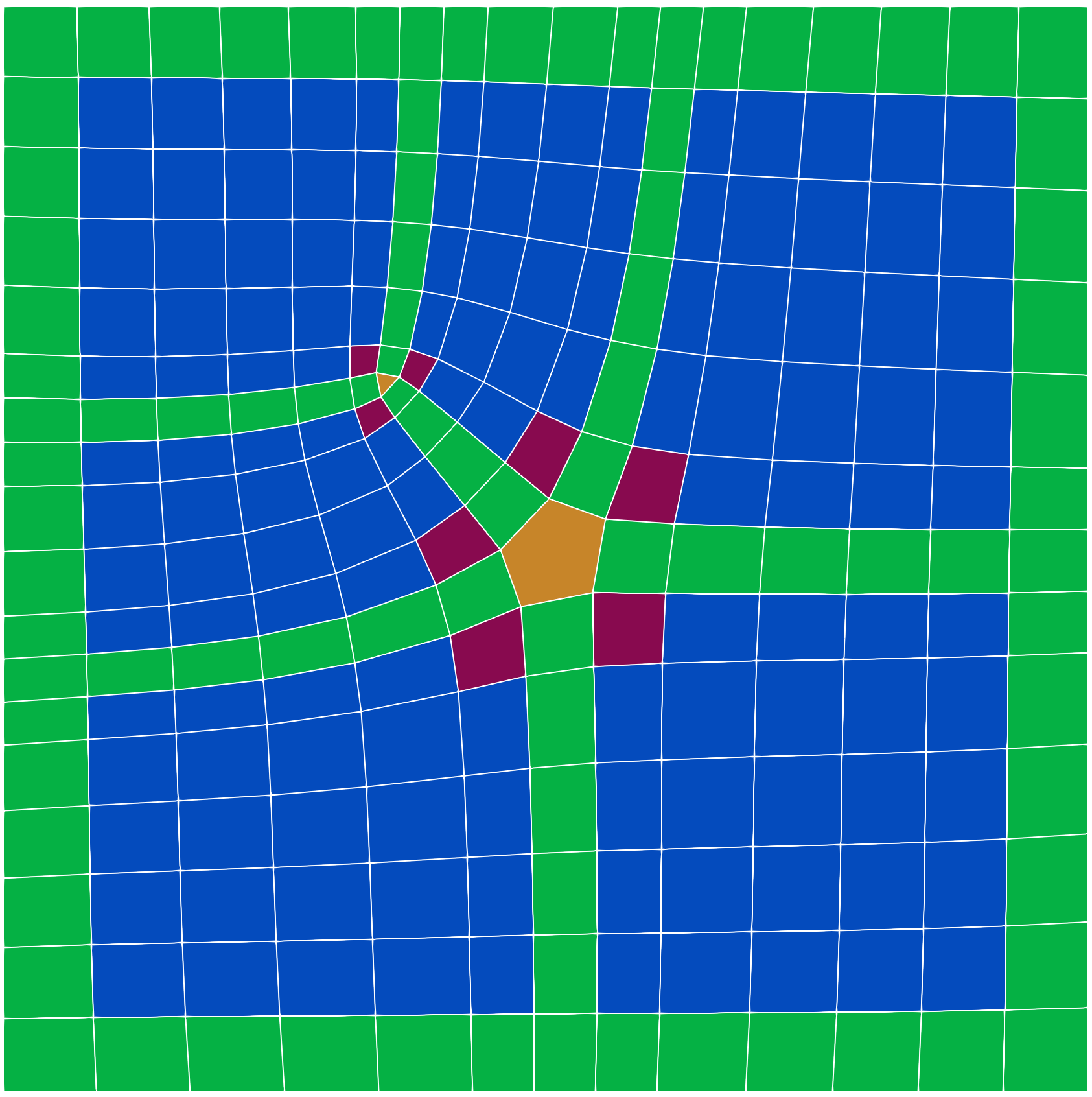}\\
(a) & (b) & (c) \\
\end{tabular}
\caption{Knot intervals and mesh terminologies. (a) The input quadrilateral mesh, (b) edges with zero knot intervals (blue) and nonzero intervals (orange), and (c) different types of faces: quadrilateral faces with zero-measure (green), polygonal faces (orange), regular faces (blue), and irregular faces (red).}
\label{fig:meshterm}
\end{figure}

We identify faces of zero-measure and nonzero-measure according to their parametric areas, which are computed using knot intervals. Zero-measure faces are not used in geometric representation and have no contribution to analysis. Note that all the polygonal faces and boundary faces have a zero-measure. The nonzero-measure faces, on the other hand, are divided into \emph{regular} and \emph{irregular} faces. An irregular face is a nonzero-measure face that shares a vertex with a certain polygonal face; all the other nonzero-measure faces are regular; see Figure \ref{fig:meshterm}(c). The tHNUS basis functions defined on a regular element\footnote{We use \emph{face} and \emph{element} interchangeably, but ``face" emphasizes mesh topology whereas ``element" is IGA-oriented.} are simply B-spline basis functions. In what follows, we restrict our attention to those defined on an irregular element. For simplicity of explanation, we assume that there is only one polygonal face in the 1-ring neighborhood of an irregular element. The $1$-ring neighborhood of a face is a collection of faces that share vertices with this face, and recursively, the $n$-ring ($n\geq 2$) neighborhood consists of faces in the $(n-1)$-ring neighborhood as well as the faces sharing vertices with the $(n-1)$-ring neighborhood.

\begin{myremark}
In a hybrid control mesh, each interior vertex is shared by four faces (or edges) and thus it has a regular valence of four. However, it does not mean that mesh irregularities are removed by converting the input quadrilateral mesh to its hybrid counterpart. In fact, irregularities are now manifested in the polygonal faces, which will be detailed in the following.
\end{myremark}

%Recursively, the ($n+1$)-ring ($n\geq 1$) neighborhood of a face is formed by its $n$-ring neighborhood as well as the faces sharing vertices with the $n$-ring neighborhood. This assumption allows us to focus on one polygonal face at a time.

Given an irregular element $\Omega$, let $N$ denote the number of vertices in its adjacent polygonal face which is equivalent to the valence of the corresponding extraordinary vertex in the input quadrilateral mesh. There are $K:=N+12$ basis functions defined on $\Omega$, associated with a local mesh around the polygonal face; see Figure \ref{fig:localmesh}(a). We denote
$$
\mathbf{B}_0(u,v) = [B_{0,1}(u,v), B_{0,2}(u,v), \ldots, B_{0,K}(u,v)]^T
$$
and
$$
\mathbf{P}_0 = [P_{0,1}, P_{0,2}, \ldots, P_{0,K}]^T
$$
the basis functions and the corresponding control vertices, respectively. Their indices are ordered according to Figure \ref{fig:localmesh}(a). The surface patch, i.e., the geometric mapping restricted to $\Omega$ is then represented by
\begin{equation}
s(u,v) = \mathbf{P}_0^{T} \mathbf{B}_0(u,v), \quad (u,v)\in\Omega.
\label{eq:s0}
\end{equation}
Note that $\Omega$ naturally has a parametric domain $[0,d_1]\times [0,d_2]$ that is determined by the corresponding knot intervals. We rescale it to $\Omega=[0,1]^2$ to unify the treatment of irregular elements. The influence of the rescaling will be discussed in Remark \ref{rmk:rescale}.

%To ensure that the rescaling does not affect the subdivision matrix, we also rescale the surrounding knot intervals as the subdivision matrix is determined by the ratios of these knot intervals; see Figure \ref{fig:localmesh}(b).

\begin{figure}[htb]
\centering
\begin{tabular}{cc}
\includegraphics[width=0.3\linewidth]{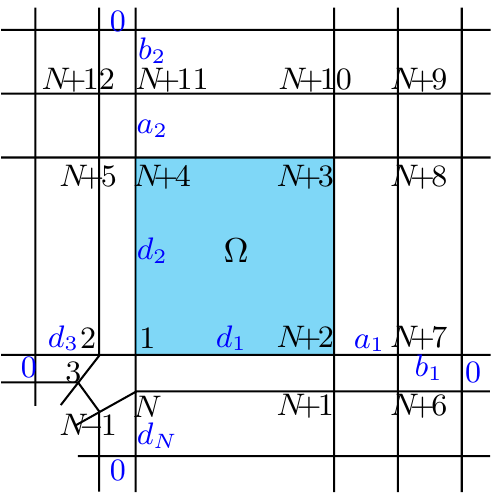}&
\includegraphics[width=0.3\linewidth]{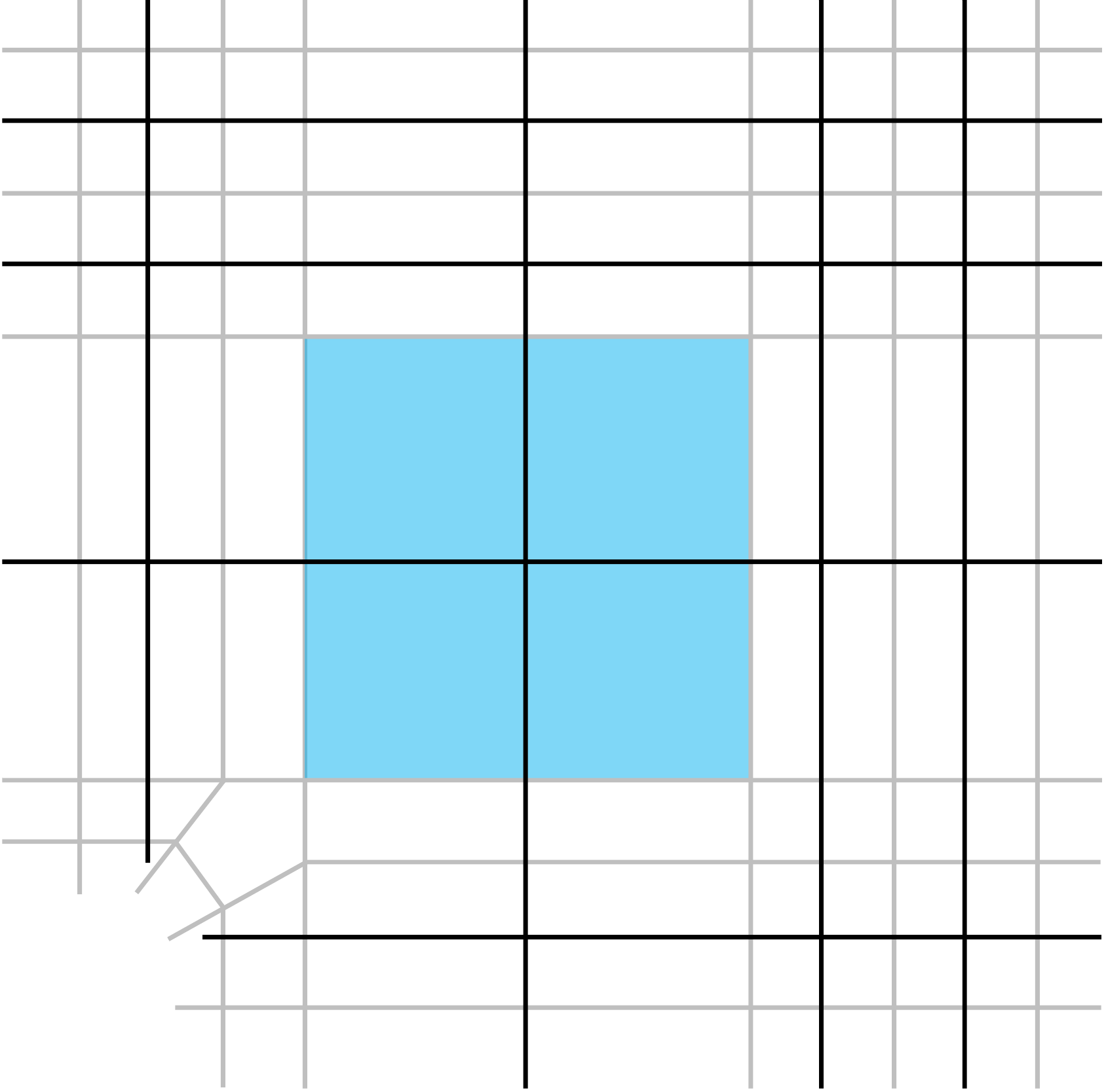}\\
(a) & (b)\\
\end{tabular}
\begin{tabular}{ccc}
\includegraphics[width=0.3\linewidth]{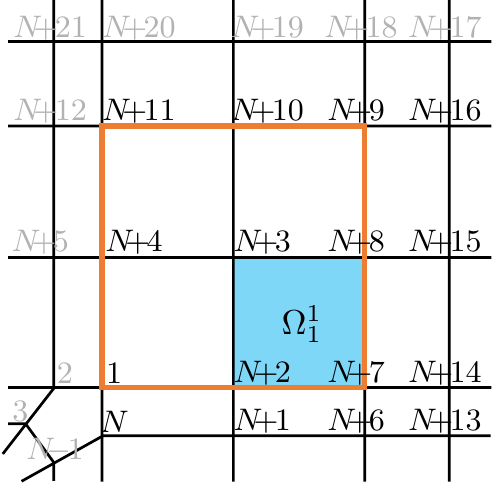}&
\includegraphics[width=0.3\linewidth]{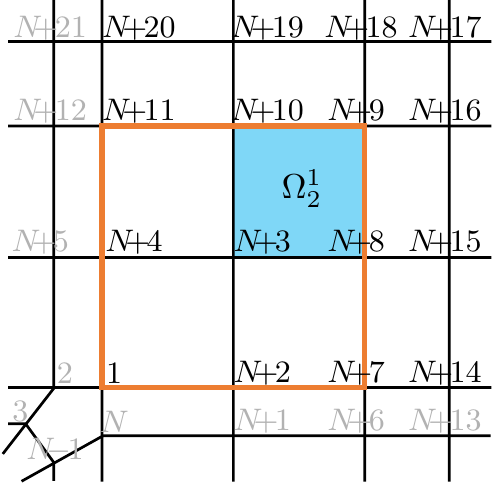}&
\includegraphics[width=0.3\linewidth]{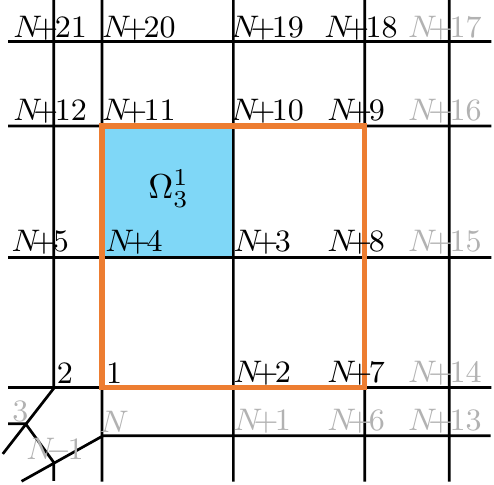}\\
(c) & (d) & (e)\\
\end{tabular}
\caption{Local meshes of an irregular element $\Omega$ and its refined subelements $\Omega_k^1$ ($k=1,2,3$). (a) The local mesh and surrounding knot intervals of $\Omega$, (b) the globally refined mesh, and (c--e) the local meshes of $\Omega_k^1$, where the orange lines indicate the boundary of $\Omega$, and indices in light gray imply that the corresponding basis functions have no support on the highlighted subelement.}
\label{fig:localmesh}
\end{figure}

Our focus is to derive $\mathbf{B}_0$, which relies on subdivision of the corresponding control mesh $\mathbf{P}_0$. Applying subdivision once yields Level-1 control vertices, denoted by
\begin{align}
\mathbf{P}_1 &= [P_{1,1}, P_{1,2}, \ldots, P_{1,K}]^T = \mathbf{S}_1 \mathbf{P}_0, \nonumber \\
\bar{\mathbf{P}}_1 &= [P_{1,1}, P_{1,2},  \ldots, P_{1,K}, P_{1,K+1},\ldots, P_{1,M}]^T = \bar{\mathbf{S}}_1 \mathbf{P}_0,
\label{eq:sub0}
\end{align}
where $M:=K+9=N+21$. The subdivision matrices $\mathbf{S}_1$ and $\bar{\mathbf{S}}_1$ have the dimension of $K\times K$ and $M\times K$, respectively. Clearly, $\bar{\mathbf{S}}_1$ yields additional 9 vertices compared to $\mathbf{S}_1$. The entries of $\mathbf{S}_1$ and $\bar{\mathbf{S}}_1$ come from the tHNUS geometric rules as well as the mid-knot insertion of B-splines, see Equations (\ref{eq:rule_irr}, \ref{eq:rule_c}) and (\ref{eq:rule_reg}, \ref{eq:rule_regm}), respectively. Among the four subelements at Level 1, three of them $\Omega_k^1$ ($k=1,2,3$) are regular and correspond to regular $C^1$ B-spline patches. In other words, the surface patch restricted to $\Omega_k^1$ is given by
\begin{equation}
s(u,v) = \mathbf{P}_{1,k}^{T} \mathbf{N}_{1,k}(u,v), \quad (u,v)\in \Omega_k^1 \subset \Omega,
\label{eq:subsurf}
\end{equation}
where $\mathbf{P}_{1,k}$ is a subvector of $\bar{\mathbf{P}}_1$ and $\mathbf{N}_{1,k}(u,v)$ is the vector of B-splines defined on $\Omega_k^1$ ($k=1,2,3$); see Figure \ref{fig:localmesh}(c--e). Both $\mathbf{P}_{1,k}$ and $\mathbf{N}_{1,k}$ have a dimension of 16 due to the bicubic degree setting. $\mathbf{P}_{1,k}$ can be obtained with the help of a permutation matrix $\mathbf{T}_k$, i.e., $\mathbf{P}_{1,k}=\mathbf{T}_k \bar{\mathbf{P}}_1 = \mathbf{T}_k \bar{\mathbf{S}}_1 \mathbf{P}_0$. Therefore, Equation \eqref{eq:subsurf} becomes
\begin{equation}
s(u,v) = (\mathbf{T}_k \bar{\mathbf{S}}_1 \mathbf{P}_0)^T \mathbf{N}_{1,k}(u,v) = \mathbf{P}_{0}^T (\mathbf{T}_k \bar{\mathbf{S}}_1)^{T} \mathbf{N}_{1,k}(u,v), \quad (u,v)\in\Omega_k^1.
\label{eq:s1}
\end{equation}
For Equations (\ref{eq:s0}, \ref{eq:s1}) to be equivalent under arbitrary choice of $\mathbf{P}^0$, we need
\begin{equation}
\mathbf{B}_0(u,v) = (\mathbf{T}_k \bar{\mathbf{S}}_1)^{T} \mathbf{N}_{1,k}(u,v), \quad (u,v)\in \Omega_k^1 .
\label{eq:B01}
\end{equation}
In other words, we have found the definition of $\mathbf{B}_0$ on three quarters ($\Omega_1^1$, $\Omega_2^1$ and $\Omega_3^1$) of $\Omega$.

%computed from $\mathbf{P}_0$ through a submatrix $\bar{\mathbf{S}}_{1,k}$ of $\bar{\mathbf{S}}_1$, which is formed by picking corresponding rows of $\bar{\mathbf{S}}_1$.

Now we are left to find the definition of $\mathbf{B}_0$ on the remaining quarter $[0,\frac{1}{2}]^2$, and we proceed with the same idea explained above. As a result, the domain $\Omega$ is partitioned into an infinite series of tiles,
$$
\Omega = \bigcup_{n=1}^{\infty} \bigcup_{k=1}^3 \Omega_k^n,
$$
where
\begin{align*}
\Omega_1^n &= \left[\frac{1}{2^n}, \frac{1}{2^{n-1}} \right] \times \left[0,\frac{1}{2^n} \right], \\
\Omega_2^n &= \left[\frac{1}{2^n}, \frac{1}{2^{n-1}} \right] \times \left[\frac{1}{2^n},\frac{1}{2^{n-1}} \right], \\
\Omega_3^n &= \left[0,\frac{1}{2^n} \right] \times \left[\frac{1}{2^n},\frac{1}{2^{n-1}}\right].
\end{align*}
Analogous to Equation \eqref{eq:sub0}, we have
\begin{align*}
\mathbf{P}_n &=\mathbf{S}_n \mathbf{P}_{n-1} = \mathbf{S}_n \mathbf{S}_{n-1} \cdots \mathbf{S}_1 \mathbf{P}_0, \\
\bar{\mathbf{P}}_n &=\bar{\mathbf{S}}_n \mathbf{P}_{n-1} = \bar{\mathbf{S}}_n \mathbf{S}_{n-1} \cdots \mathbf{S}_1 \mathbf{P}_0.
\end{align*}
Again, $\mathbf{S}_n$ and $\bar{\mathbf{S}}_n$ ($n\geq 1$) have the dimension of $K\times K$ and $M\times K$, respectively. Note that $\mathbf{S}_2=\mathbf{S}_3=\cdots =\mathbf{S}_n$ ($n\geq 2$) and $\bar{\mathbf{S}}_3 = \bar{\mathbf{S}}_4 = \cdots = \bar{\mathbf{S}}_n$ ($n\geq 3$) because the ratios of knot intervals around an irregular subelement become fixed as the subdivision level increases. Therefore, following the same argument in deriving Equation \eqref{eq:B01}, we have the general expressions for $\mathbf{B}_0$,
\begin{equation*}
\mathbf{B}_0(u,v) =
\begin{cases}
(\mathbf{T}_k \bar{\mathbf{S}}_{1})^{T} \mathbf{N}_{1,k}(u,v) & n=1 \\
\left( \mathbf{T}_k \bar{\mathbf{S}}_{2} \mathbf{S}_1 \right)^{T} \mathbf{N}_{2,k}(u,v) & n=2 \\
\left( \mathbf{T}_k \bar{\mathbf{S}}_{3} \left(\mathbf{S}_2\right)^{n-2} \mathbf{S}_1 \right)^{T} \mathbf{N}_{n,k}(u,v) & n\geq 3
\end{cases},
\end{equation*}
where $(u,v)\in\Omega_k^n$, and $\mathbf{N}_{n,k}(u,v)$ ($n\geq 1$) is the vector of B-splines defined on $\Omega_k^n$.

\begin{myremark}
Rescaling $\Omega$ to $[0,1]^2$ only affects the knot vectors (or equivalently, the vectors of knot intervals) of B-splines $\mathbf{N}_{n,k}(u,v)$ ($n\geq 1$). For example, when $n=1$ without scaling, the vector of knot intervals in the $u$ direction is
$$\left\{\frac{d_3}{2},\frac{d_3}{2},0,\frac{d_1}{2},\frac{d_1}{2},\frac{a_1}{2},\frac{a_1}{2},\frac{b_1}{2}\right\},$$
whereas after scaling with respect to $d_1$, it becomes
$$U_1 = \left\{\frac{d_3}{2 d_1},\frac{d_3}{2 d_1},0,\frac{1}{2},\frac{1}{2},\frac{a_1}{2 d_1},\frac{a_1}{2 d_1},\frac{b_1}{2 d_1}\right\},$$
and similarly, we have the vector of knot intervals in the $v$ direction,
$$V_1 = \left\{\frac{d_{N}}{2 d_2},\frac{d_{N}}{2 d_2},0,\frac{1}{2},\frac{1}{2},\frac{a_2}{2 d_2},\frac{a_2}{2 d_2},\frac{b_2}{2 d_2}\right\}. $$
Moreover, when $n=2$, we have
$$
U_2 = \left\{\frac{d_3}{2^2 d_1},\frac{d_3}{2^2 d_1},0,\frac{1}{2^2},\frac{1}{2^2},\frac{1}{2^2},\frac{1}{2^2},\frac{a_1}{2^2 d_1}\right\},\
V_2  = \left\{\frac{d_{N}}{2^2 d_2},\frac{d_{N}}{2^2 d_2},0,\frac{1}{2^2},\frac{1}{2^2},\frac{1}{2^2},\frac{1}{2^2},\frac{a_2}{2^2 d_2}\right\},
$$
and when $n\geq 3$,
$$
U_n = \left\{\frac{d_3}{2^n d_1},\frac{d_3}{2^n d_1},0,\frac{1}{2^n},\frac{1}{2^n},\frac{1}{2^n},\frac{1}{2^n},\frac{1}{2^n}\right\}, \
V_n  = \left\{\frac{d_{N}}{2^n d_2},\frac{d_{N}}{2^n d_2},0,\frac{1}{2^n},\frac{1}{2^n},\frac{1}{2^n},\frac{1}{2^n},\frac{1}{2^n}\right\}.
$$
$\mathbf{N}_{n,k}(u,v)$ are defined using $U_n$ and $V_n$.
\label{rmk:rescale}
\end{myremark}

In fact, it is practically useful to rescale each tile $\Omega_k^n$ to $[0,1]^2$ by
$$
\begin{array}{lll}
k=1 \qquad & \xi = 2^n u-1, & \eta = 2^n v, \\
k=2 \qquad & \xi = 2^n u-1, & \eta = 2^n v -1, \\
k=3 \qquad & \xi = 2^n u,     & \eta = 2^n v -1. \\
\end{array}
$$
Correspondingly, the vectors of knot intervals are rescaled to
$$
\begin{array}{lll}
n=1 & \Xi_1 = \left\{ \frac{d_3}{d_1}, \frac{d_3}{d_1}, 0,1,1, \frac{a_1}{d_1},\frac{a_1}{d_1},\frac{b_1}{d_1}  \right\}, & \Theta_1 = \left\{ \frac{d_N}{d_2}, \frac{d_N}{d_2}, 0,1,1, \frac{a_2}{d_2},\frac{a_2}{d_2},\frac{b_2}{d_2} \right\}, \\
n=2 & \Xi_2 = \left\{ \frac{d_3}{d_1}, \frac{d_3}{d_1}, 0,1,1, 1, 1,\frac{a_1}{d_1} \right\}, & \Theta_2 = \left\{ \frac{d_N}{d_2}, \frac{d_N}{d_2}, 0,1,1, 1,1,\frac{a_2}{d_2} \right\}, \\
n\geq 3 & \Xi_3 = \left\{ \frac{d_3}{d_1}, \frac{d_3}{d_1}, 0,1,1, 1, 1,1  \right\}, & \Theta_3 = \left\{ \frac{d_N}{d_2}, \frac{d_N}{d_2}, 0,1,1, 1,1,1 \right\}, \\
\end{array}
$$
where the rescaled knot intervals are now independent of the subdivision level $n$ when $n \geq 3$. In summary, the basis functions of interest are defined as
\begin{equation}
\mathbf{B}_0(u,v) =
\begin{cases}
(\mathbf{T}_k \bar{\mathbf{S}}_{1})^{T} \mathbf{b}_{1,k}(\xi(u),\eta(v)) & n=1 \\
\left( \mathbf{T}_k \bar{\mathbf{S}}_{2} \mathbf{S}_1 \right)^{T} \mathbf{b}_{2,k}(\xi(u),\eta(v)) & n=2 \\
\left( \mathbf{T}_k \bar{\mathbf{S}}_{3} \left(\mathbf{S}_2\right)^{n-2} \mathbf{S}_1 \right)^{T} \mathbf{b}_{3,k}(\xi(u),\eta(v)) & n\geq 3
\end{cases},
\end{equation}
where $(\xi,\eta)\in [0,1]^2$ and $\mathbf{b}_{l,k}(\xi,\eta)$ are B-splines defined using $\Xi_l$ and $\Theta_l$ ($l= 1,2,3$). Note that when $n=1,2$, we have $\mathbf{N}_{n,k}(u,v) = \mathbf{b}_{n,k}(\xi(u),\eta(v))$, and when $n\geq 3$, $\mathbf{N}_{n,k}(u,v) = \mathbf{b}_{3,k}(\xi(u),\eta(v))$.

%\begin{myremark}
%We keep in mind that basis functions are associated with the initial control mesh. Subsequent meshes are used only for explanation purpose.
%\end{myremark}

%explain the procedure: given $(u,v)$, find the tile $\Omega_k^n$, i.e., find $n$ and $k$, find the corresponding B-splines, evaluate these B-splines (values and/or derivatives) on a unit square. choose the right subdivision matrix and matrix multiplications, get the .

\begin{myremark}
Evaluation of $\mathbf{B}_0(u,v)$ at $(0,0)$ in an irregular element needs the computation of $\lim_{n\to\infty} (\mathbf{S}_2)^{n-2}$. Following \cite{ref:stam98}, we need to eigen-decompose $\mathbf{S}_2$ such that $\mathbf{S}_2 = \mathbf{V} \mathbf{\Lambda} \mathbf{V}^{-1}$, where $\mathbf{\Lambda}$ is a diagonal matrix containing the eigenvalues of $\mathbf{S}_2$ and $\mathbf{V}$ is an invertible matrix with columns being the corresponding eigenvectors. Accordingly, we have $\lim_{n\to\infty} (\mathbf{S}_2)^{n-2} =\lim_{n\to\infty} \mathbf{V}\mathbf{\Lambda}^{n-2}\mathbf{V}^{-1}$. Recall that all the eigenvalues are smaller than 1 (and greater than 0) except the first one $\lambda_1=1$, so $\lim_{n\to\infty} \mathbf{\Lambda}^{n-2}$ is a matrix whose entries are all zero except the first-row-first-column entry, which is one. On the other hand, the derivatives of $\mathbf{B}_0(u,v)$ are not bounded around $(0,0)$. We can see this by applying the chain rule, for example,
$$
\pd{\mathbf{B}_0(u,v)}{u} = 2^n \left( \mathbf{T}_k \bar{\mathbf{S}}_{3} \left(\mathbf{S}_2\right)^{n-2} \mathbf{S}_1 \right)^{T} \pd{\mathbf{b}_{3,k}(\xi,\eta)}{\xi},
$$
where the factor comes from $d\xi/d u = 2^n$. A differentiable version of $\mathbf{B}_0$ (with respect to certain parameters) can be obtained via characteristic-map-based reparameterization \cite{ref:boiermartin04}. However, in this paper, we are interested in applying tHNUS basis functions in the context of IGA, so we only need derivatives at quadrature points that are away from $(0,0)$. Moreover, what we eventually need is derivatives with respect to the physical coordinates, for example,
$$\pd{\mathbf{B}_0 \circ s^{-1}(x,y)}{x} = \pd{\mathbf{B}_0}{u} \pd{u}{x} + \pd{\mathbf{B}_0}{v} \pd{v}{x},$$
where $s^{-1}$ is the inverse mapping of $s(u,v)$. The troublesome factor $2^n$, which may cause overflow when $n$ becomes too large, is canceled out with that from $\partial u/\partial x$ (which is $2^{-n}$) and does not cause any numerical issues. The same argument applies to higher order derivatives.
\end{myremark}

\begin{myremark}
In \cite{ref:stam98}, the eigen structure $(\mathbf{\Lambda},\mathbf{V})$ is precomputed for different valence numbers and stored in a file for repeated use. However, the same scheme cannot be applied to tHNUS because the subdivision matrix $\mathbf{S}_2$ depends on not only the valence number but also the surrounding knot intervals, leading to infinite possible cases of $\mathbf{S}_2$. Therefore, the eigen structure of $\mathbf{S}_2$ needs to be found in real time for every irregular element. Alternatively, we can directly perform matrix multiplications to compute $(\mathbf{S}_2)^{n-2}$, especially when the valence number is small and basis functions need to be computed at points other than $(0,0)$. This is indeed the case in IGA where evaluation is needed at quadrature points. In practice, we adopt a near-machine-precision tolerance (e.g., $10^{-13}$) to prevent a potential overflow issue.
%Moreover, the matrix multiplication is optimized in most matrix libraries such as Eigen \cite{ref:eigen} and does not induce noticeable overhead. In fact, we find in our numerical test that the time in evaluating basis functions on an irregular element is comparable to that on a regular element.
\end{myremark}

%\begin{myremark}
%Evaluation of the surface point $s(u,v)$ at $(0,0)$ in an irregular element (i.e., at the singularity) needs careful treatment. $s(0,0)$ can be obtained through taking the limit of $\mathbf{P}_{n}$, i.e., $\lim_{n\to\infty} \mathbf{P}_{n} =\lim_{n\to\infty} (\mathbf{S}_2)^{n-1} \mathbf{S}_1 \mathbf{P}_0$. To evaluate this limit, we need to eigen-decompose $\mathbf{S}_2$ such that $\mathbf{S}_2 = \mathbf{V} \mathbf{\Lambda} \mathbf{V}^{-1}$, where $\mathbf{\Lambda}$ is a diagonal matrix containing the eigenvalues of $\mathbf{S}_2$ and $\mathbf{V}$ is an invertible matrix with columns being the corresponding eigenvectors. Accordingly, we have $\lim_{n\to\infty} \mathbf{P}_{n} =\lim_{n\to\infty} \mathbf{V}\mathbf{\Lambda}^{n-1}\mathbf{V}^{-1} \mathbf{S}_1 \mathbf{P}_0$. Recall that all the eigenvalues are smaller than 1 (and greater than 0) except the first one $\lambda_1=1$. Therefore, the limit of $\mathbf{\Lambda}^{n-1}$ exists and thus we can find $s(0,0)=\lim_{n\to \infty} P_{n,1}$, where $P_{n,1}$ is the first entry of the vector $\mathbf{P}_n$. In fact, all the control points ($P_{n,1},\ldots,P_{n,N}$) in a polygonal face collapse to $s(0,0)$ in the limit.
%\end{myremark}

%However, $\mathbf{B}_0(u,v)$ is not well defined at $(0,0)$. In practice, . reparameterization based on characteristic map. In this paper, we use a tolerance in the order of $10^{-11}\sim 10^{-13}$ to avoid such situation.

\begin{myremark}
In the previous discussion, $\mathbf{B}_0$ is derived under the assumption that there is only one polygonal face next to an irregular element, which, however, is not a necessary condition. When an irregular element has multiple adjacent polygonal faces, we treat it as a macro element and pseudo-subdivide it once. Each of the resulting four subelements only has one polygonal face, where basis functions are defined according to our previous discussion. In other words, basis functions are well defined on each quarter of the original macro element. This extension follows the same idea proposed in \cite{ref:wei15b}, which extends Stam's derivation \cite{ref:stam98} to arbitrary unstructured quadrilateral meshes.
\end{myremark}

%\begin{myremark}
%. reparameterization based on characteristic map.
%\end{myremark}

\subsection{Quadrature}

To apply the standard Gauss quadrature rule, we need to guarantee that the involved basis functions are polynomials (rather than piecewise polynomials) on each integration cell. However, the functions in $\mathbf{B}_0$ are piecewise smooth polynomials defined on an infinite series of subdomains, i.e., $\{\Omega_k^n\}_{n=1}^{\infty}$ ($k=1,2,3$). The straightforward way is to apply the Gauss quadrature rule on each cell $\Omega_k^n$ up to a certain fine level, which was adopted in several subdivision-based isogeometric methods \cite{ref:tnguyen14, ref:wei15b}. We call such a quadrature \emph{the full quadrature scheme}. In our patch test, we observe that the solution achieves machine precision ($\sim 10^{-16}$) when the 4-point rule is used and the level $n$ is set to be $10$. As a result, a total number of 496 quadrature points are needed for a single irregular element. In contrast, only 16 Gauss quadrature points are used for a regular element.

Alternatively, we can ``brutally" apply the Gauss quadrature rule to the entire irregular element. In other words, only 16 Gauss quadrature points are placed on an irregular element. We observe that such a \emph{reduced quadrature scheme} does not influence convergence. In fact, it does not introduce noticeable numerical error in terms of the $L^2$- or $H^1$-norm error compared to the full quadrature. We will numerically compare the two schemes in the next section.

\subsection{Properties}

Now we briefly discuss several properties of tHNUS basis functions, including non-negative partition of unity, refinability (equivalent to nested spline spaces), and global linear independence. The non-negative partition of unity of tHNUS basis functions follows from the fact that all the entries in the subdivision matrix are non-negative and each row sum of the subdivision matrix is one. Refinability states that each basis function of a given mesh can be represented as a linear combination of those defined on a refined mesh. In fact, we can see this property in the derivation of $\mathbf{B}_0$, where basis functions are always expressed as linear combinations of functions in the refined meshes.

Finally, the global linear independence implies linear independence on the entire domain, and it can be easily shown under the mild assumption that each irregular element has at least one regular element as its direct neighbor. Under this assumption, every basis function has support on a certain regular element, where it is simply a B-spline. As B-splines are linearly independent on such an element, we can conclude that all the basis functions are linearly independent on the entire domain by going through all the regular elements. The proof on general meshes becomes more involving because we need to resolve different configurations of polygonal faces, or equivalently, configurations of extraordinary vertices in the input mesh. A complex configuration usually occurs when the mesh is very coarse such that many extraordinary vertices may be next to one another. When this is the case, we can perform global refinement to guarantee linear independence.

%To keep consistent with the definition of basis functions introduced in Section \ref{sec:defbf}, T-junctions need to be placed on the boundary of the 2-ring neighborhood of a polygonal face, as shown in Figure \cite{fig:tjunc}.

\section{Numerical examples}
\label{sec:result}

In this section, we present several numerical examples using tHNUS surfaces in both geometric modeling and IGA.

%Non-uniform knot intervals are adopted in all the examples.

\subsection{Geometric modeling with tHNUS surfaces}
We show some tHNUS limit surface examples and compare them with the existing non-uniform subdivision schemes.
We first show the graphs of blending functions for the extraordinary points (EPs) with different valences, such as valence-5 EP in
Figure~\ref{fig:val5g}, valence-6 EP in Figure~\ref{fig:val6g} and valence-7 EP in Figure~\ref{fig:val7g}.
As stated in~\cite{xin18}, the approaches in \cite{Sederberg98}, \cite{ref:cashman09} and \cite{kovacs2015dyadic}
produce limit surfaces with very similar quality in all the examples. Therefore, we only show the limit surface
comparisons in one example as shown in Figure~\ref{fig:val5g}. All the rest of the examples only show the limit surface of
the new tHNUS with different $\lambda$. We can observe that all different $\lambda$ can produce better shape quality
than those approaches in \cite{Sederberg98}, \cite{ref:cashman09} and \cite{kovacs2015dyadic}, but the small $\lambda$
produces worse shape quality surround the EPs, see Figures~\ref{fig:val6z} and \ref{fig:her} for the
details.
\begin{figure}[htbp]
\begin{center}
\begin{tabular}{ccc}
\includegraphics[height=0.18\textwidth]{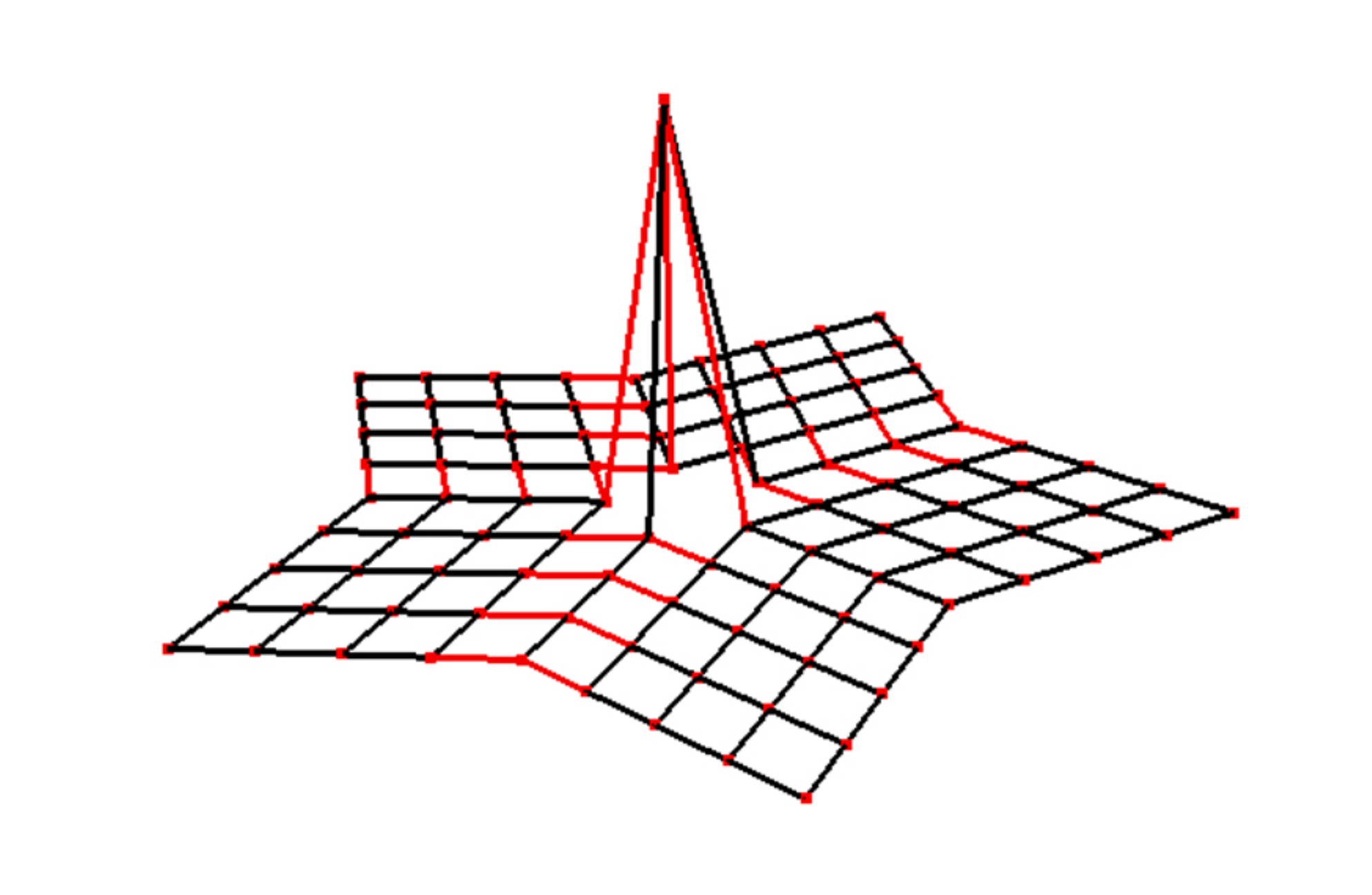}&
\includegraphics[height=0.18\textwidth]{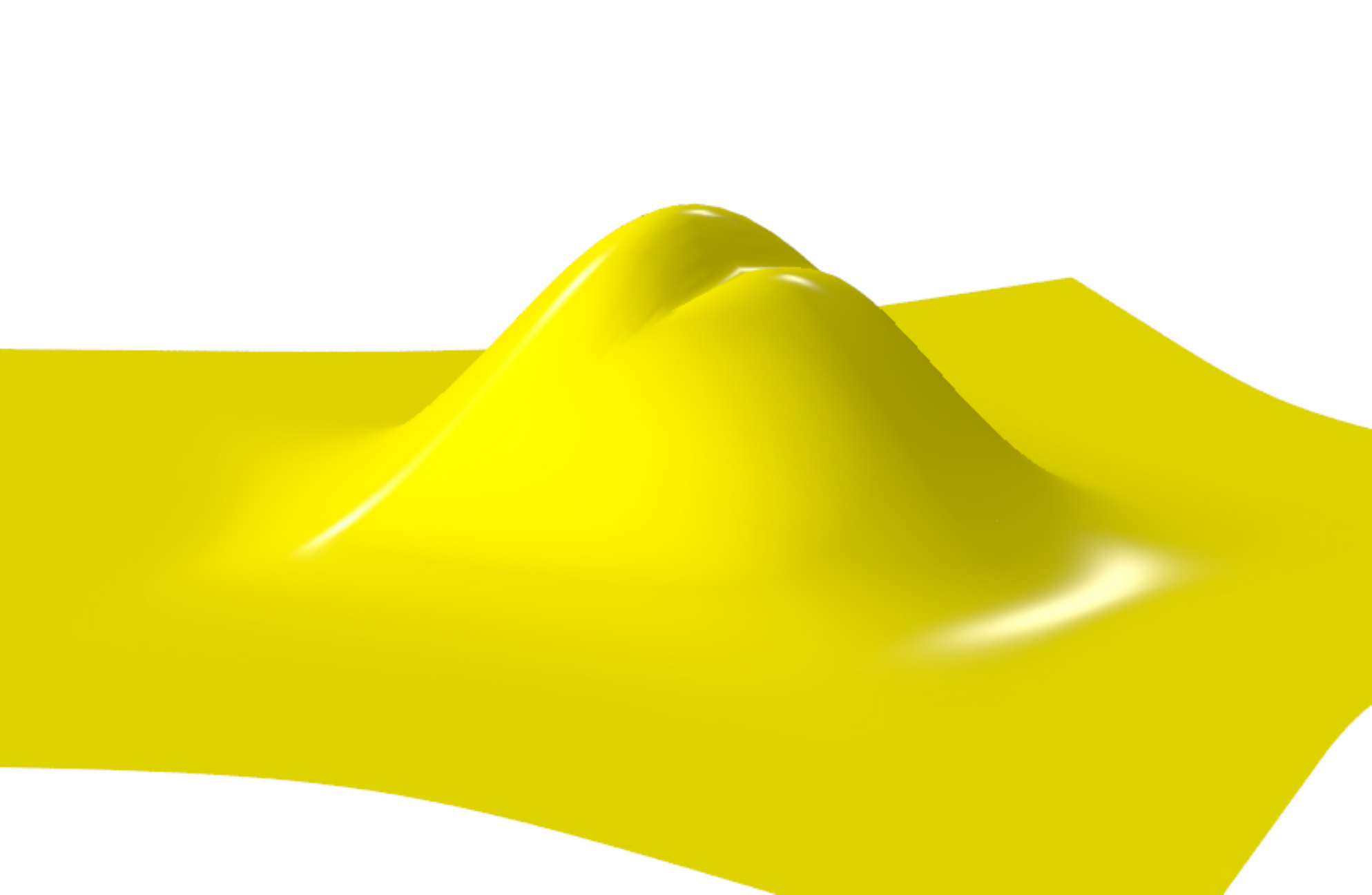}&
\includegraphics[height=0.18\textwidth]{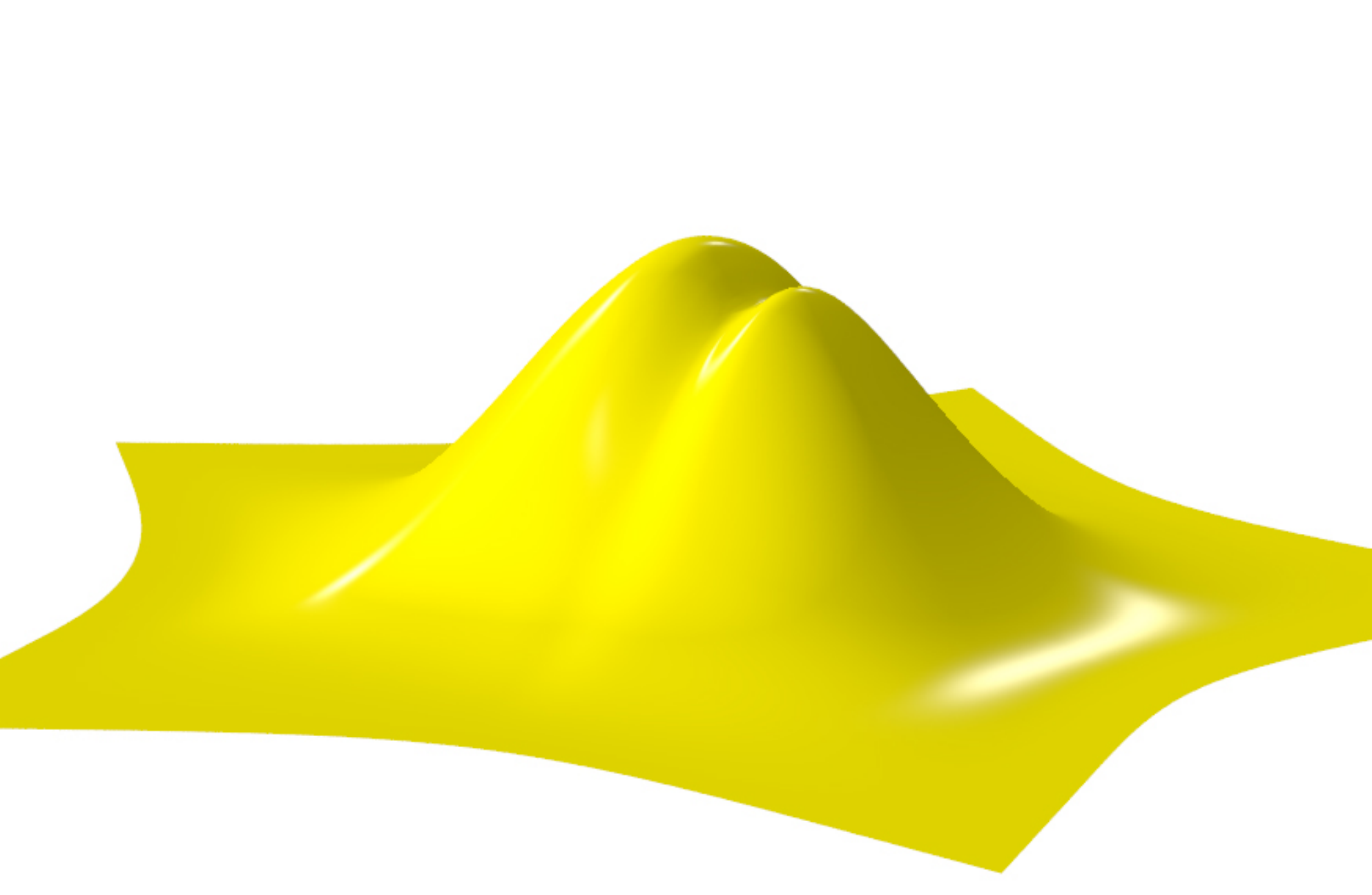}\\
(a) The control grid & (b) Result produced by~\cite{Sederberg98} & (c) Result produced by~\cite{ref:cashman09} \\
\includegraphics[height=0.18\textwidth]{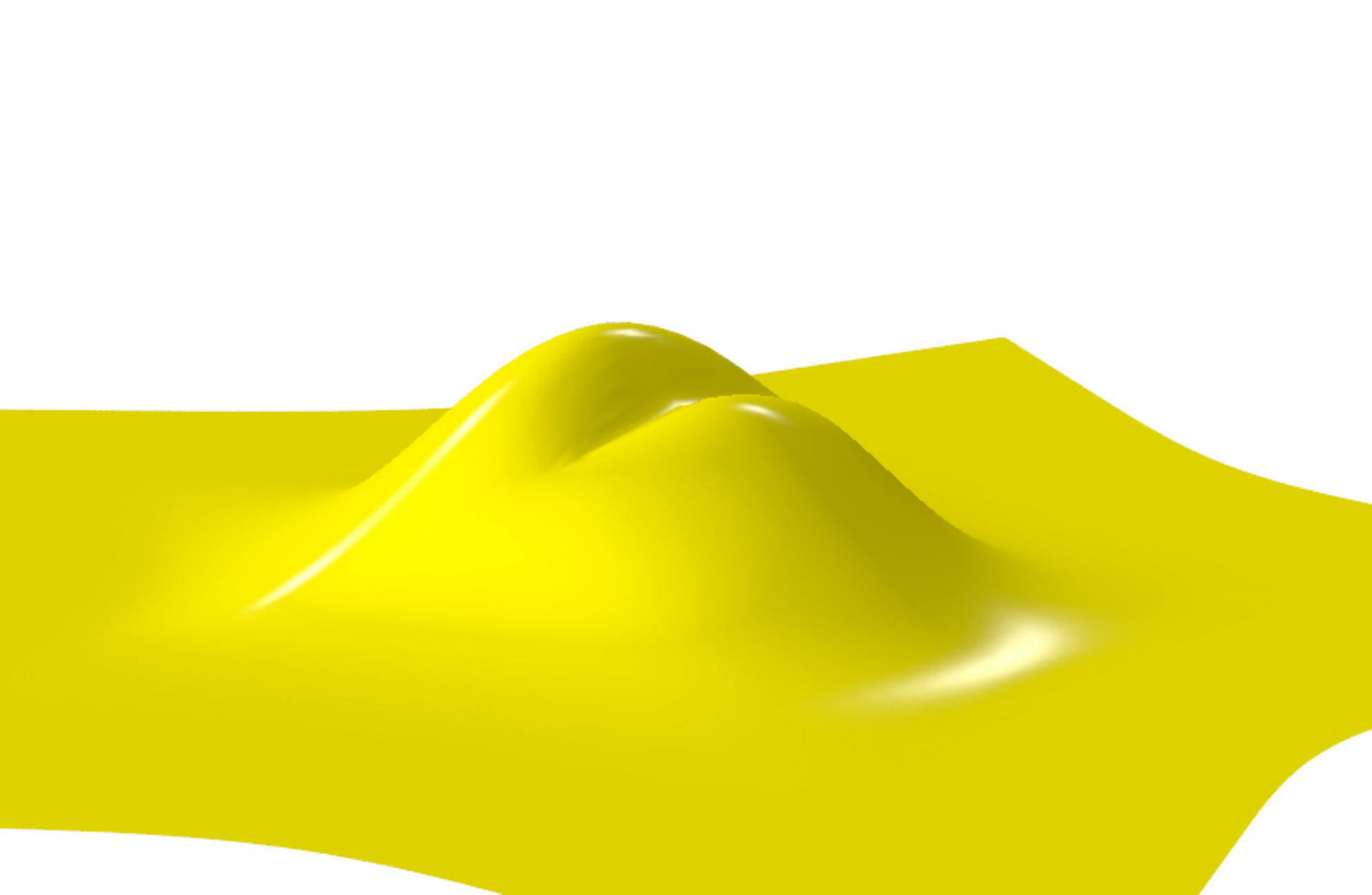}&
\includegraphics[height=0.18\textwidth]{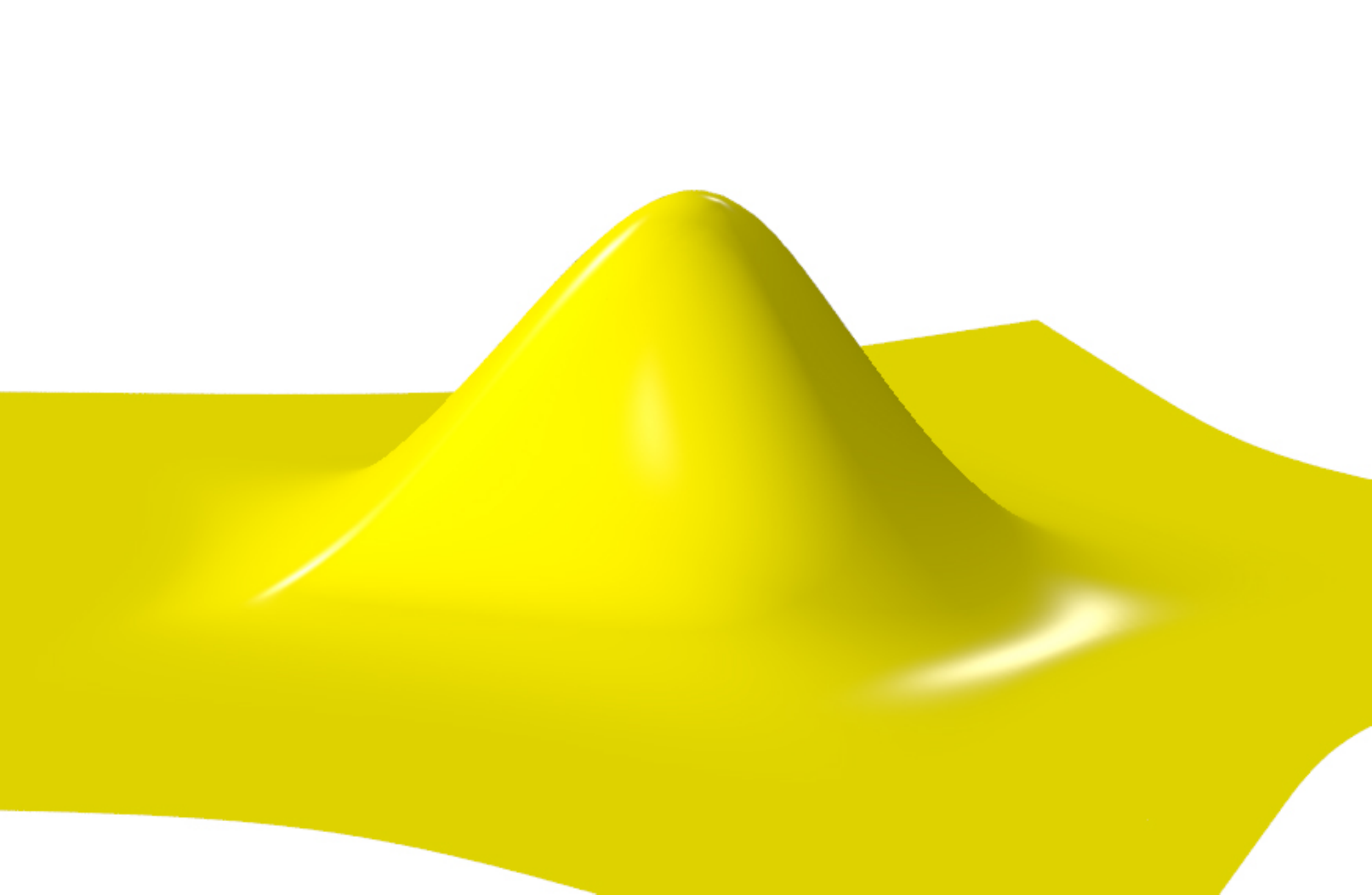}&
\includegraphics[height=0.18\textwidth]{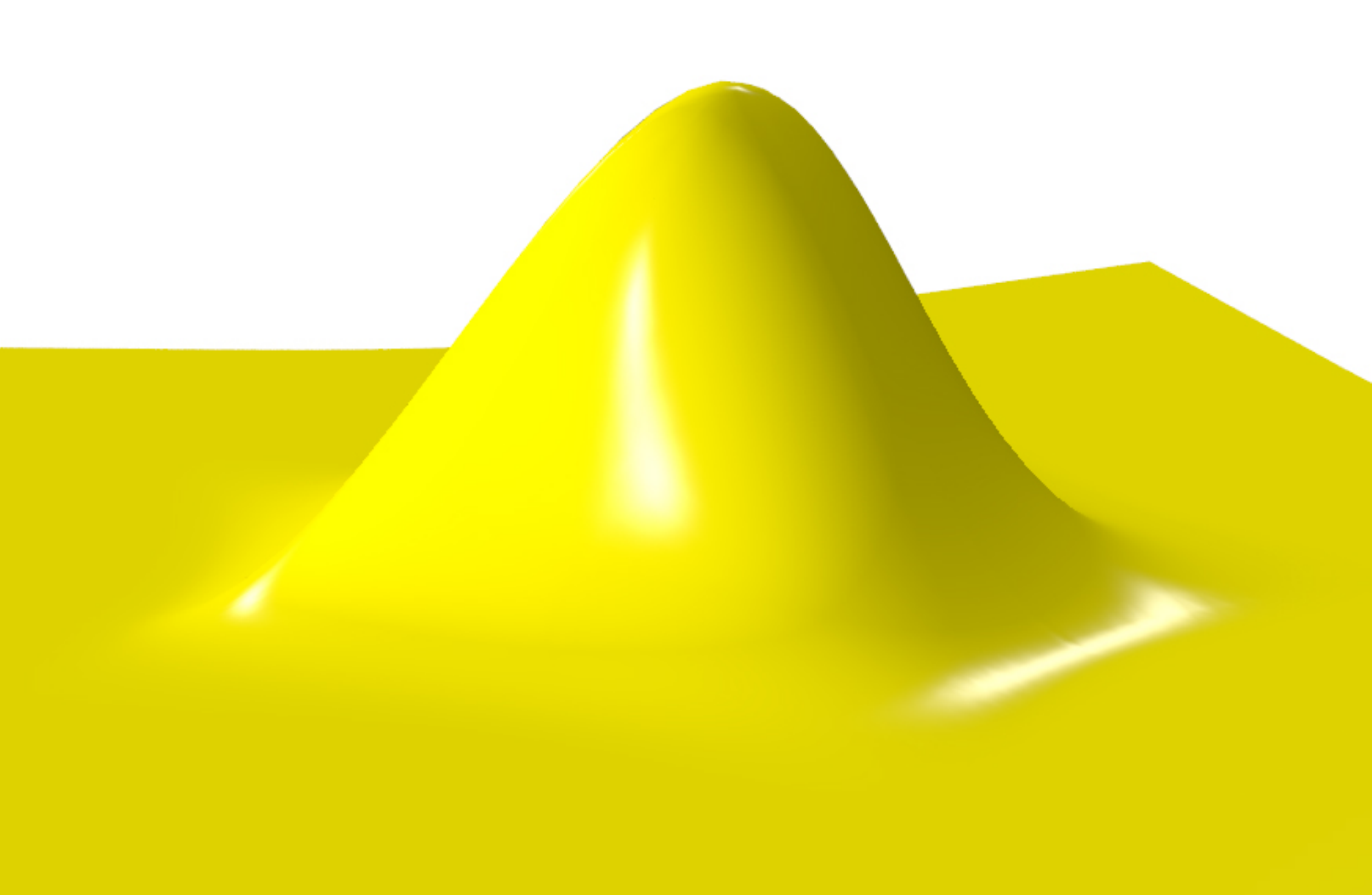}\\
(d) Result produced by~\cite{kovacs2015dyadic} & (e) Result produced by~\cite{xin16} & (f) Result produced by~\cite{xin18}\\
\includegraphics[width=0.24\textwidth]{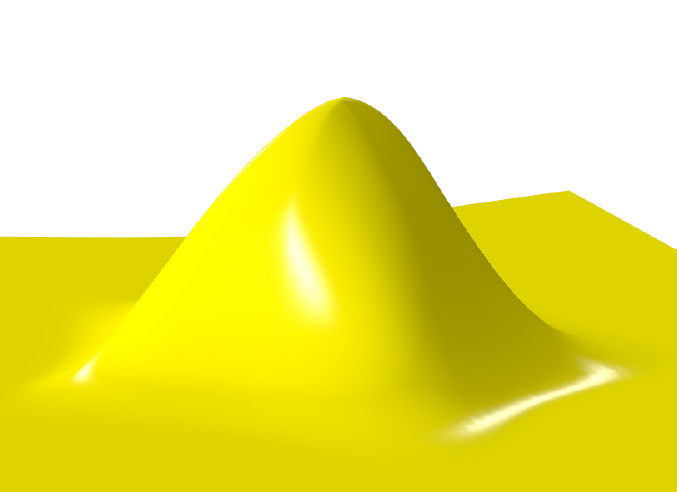}&
\includegraphics[width=0.24\textwidth]{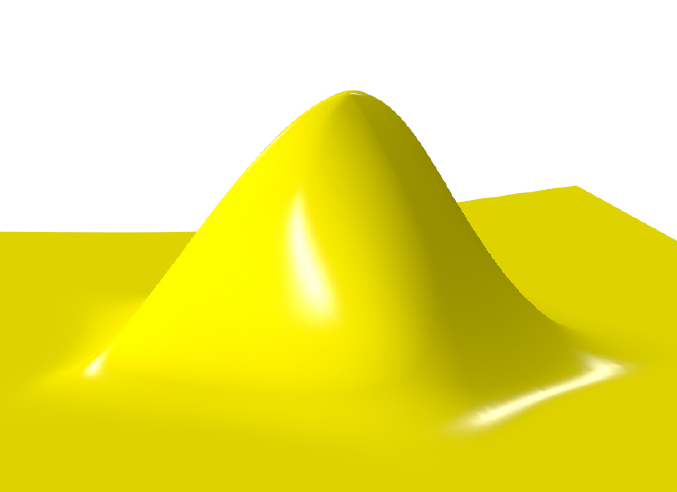}&
\includegraphics[width=0.24\textwidth]{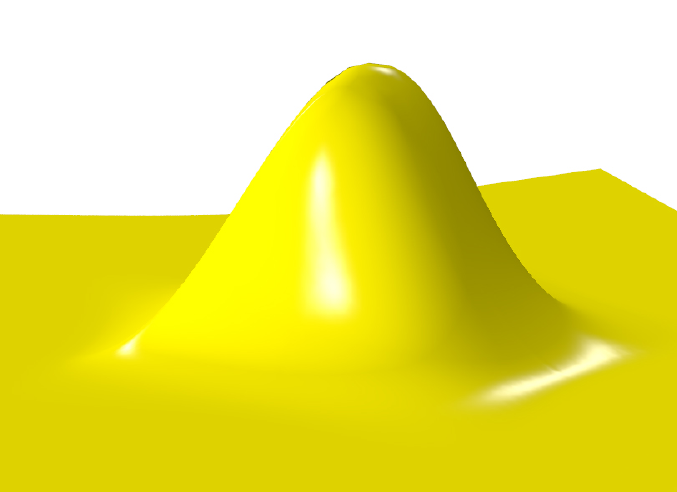}\\
(g) tHNUS with $\lambda = 0.26$ & (h) tHNUS with $\lambda = 0.35$ & (i) tHNUS with $\lambda = 0.65$
\end{tabular}
\end{center}
\caption{The blending function for a valence-5 non-uniform EP using different approaches, where the knot intervals of
the red edges are 10 and those of the other edges are 1.}
\label{fig:val5g}
\end{figure}

\begin{figure}[htbp]
\begin{center}
\begin{tabular}{ccc}
\includegraphics[width=0.24\textwidth]{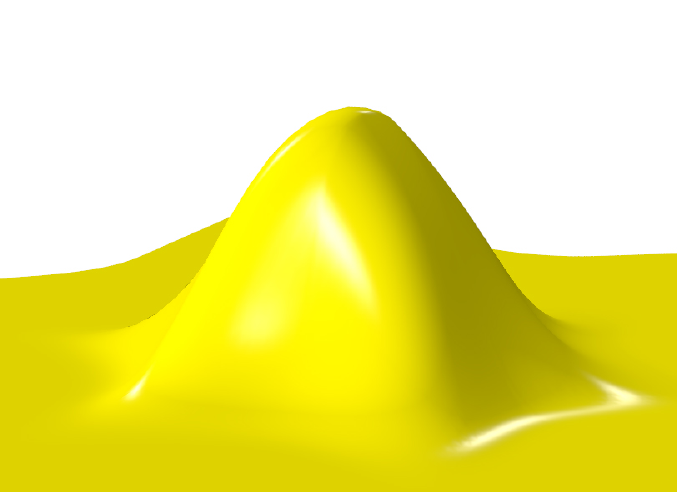}&
\includegraphics[width=0.24\textwidth]{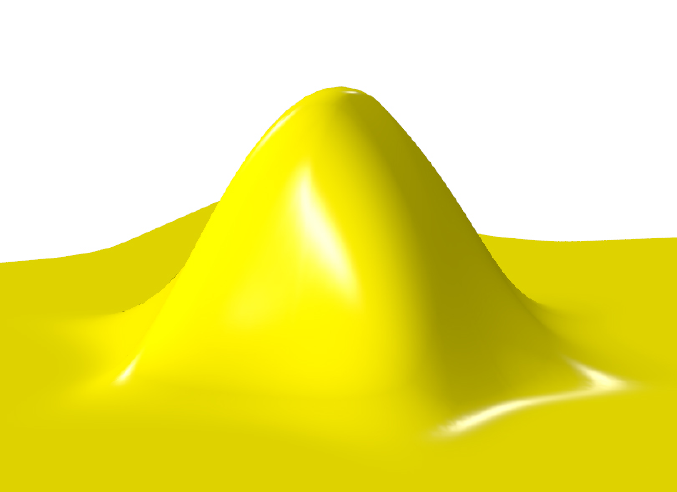}&
\includegraphics[width=0.24\textwidth]{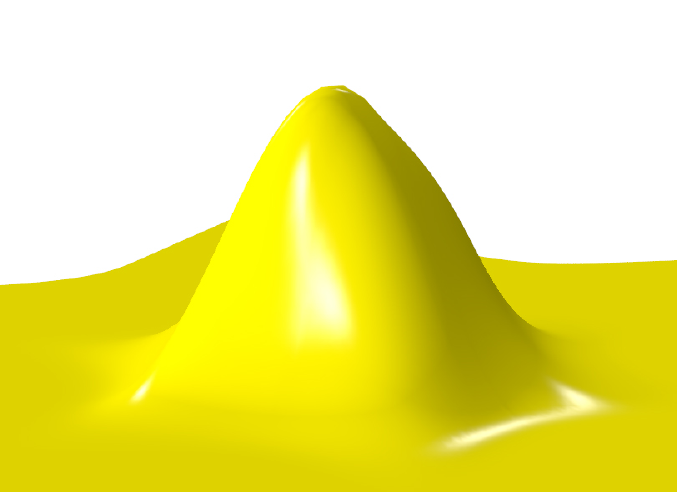}\\
(a) $\lambda = 0.26$ & (b) $\lambda = 0.35$ & (c) $\lambda = 0.65$
\end{tabular}
\end{center}
\caption{The blending function for a valence-6 non-uniform EP using different $\lambda$.}
\label{fig:val6g}
\end{figure}

\begin{figure}[htbp]
\begin{center}
\begin{tabular}{ccc}
\includegraphics[width=0.24\textwidth]{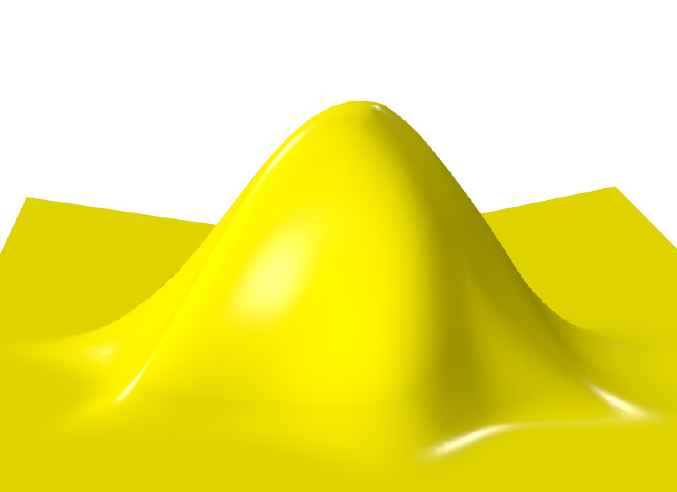}&
\includegraphics[width=0.24\textwidth]{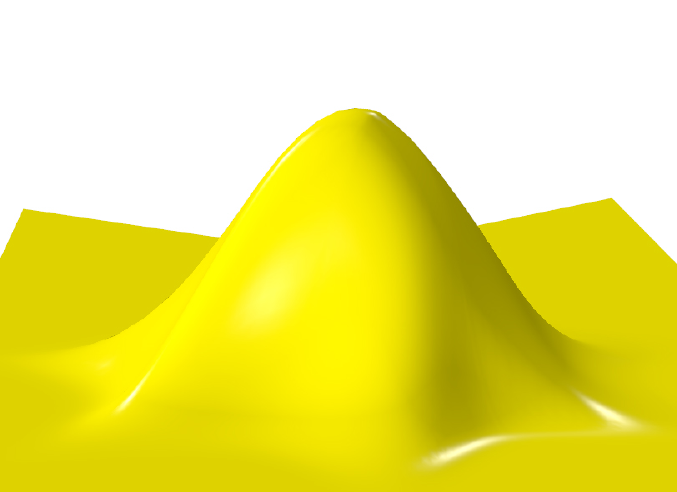}&
\includegraphics[width=0.24\textwidth]{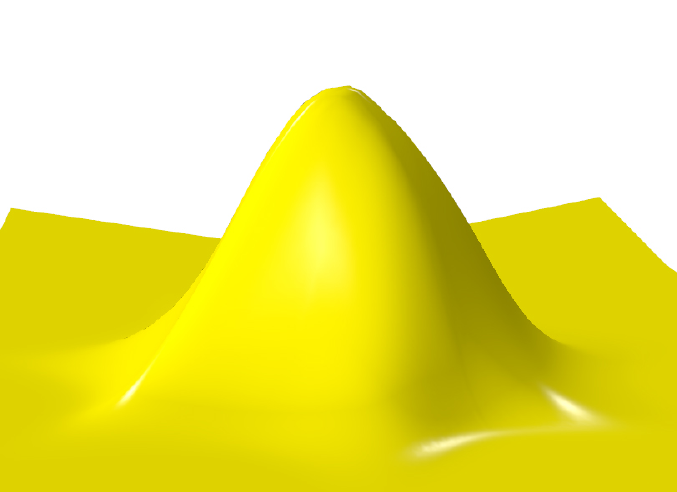}\\
(a) $\lambda = 0.26$ & (b) $\lambda = 0.35$ & (c) $\lambda = 0.65$
\end{tabular}
\end{center}
\caption{The blending function for a valence-7 non-uniform EP using different $\lambda$.}
\label{fig:val7g}
\end{figure}

\begin{figure}[htbp]
\begin{center}
\begin{tabular}{ccc}
\includegraphics[width=0.24\textwidth]{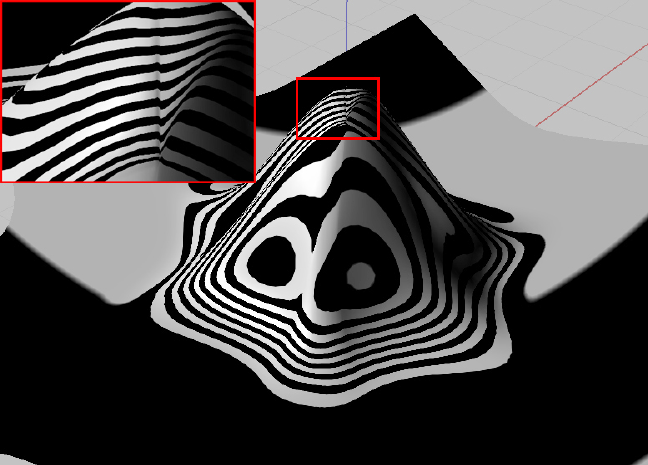}&
\includegraphics[width=0.24\textwidth]{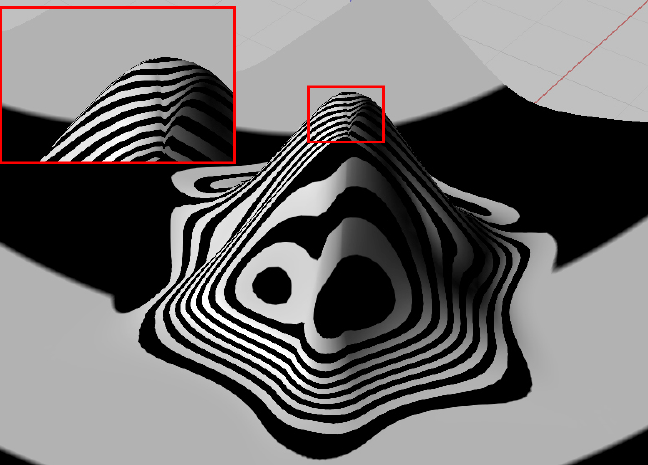}&
\includegraphics[width=0.24\textwidth]{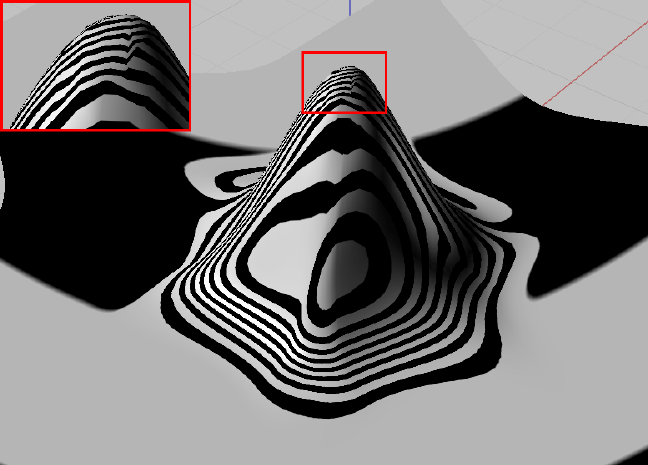}\\
(a) $\lambda = 0.26$ & (b) $\lambda = 0.35$ & (c) $\lambda = 0.65$
\end{tabular}
\end{center}
\caption{The blending function for a valence-6 non-uniform EP using different $\lambda$, where larger $\lambda$ produces more satisfactory reflection lines.}
\label{fig:val6z}
\end{figure}

\begin{figure}[htbp]
\begin{center}
\begin{tabular}{ccc}
\includegraphics[width=0.24\textwidth]{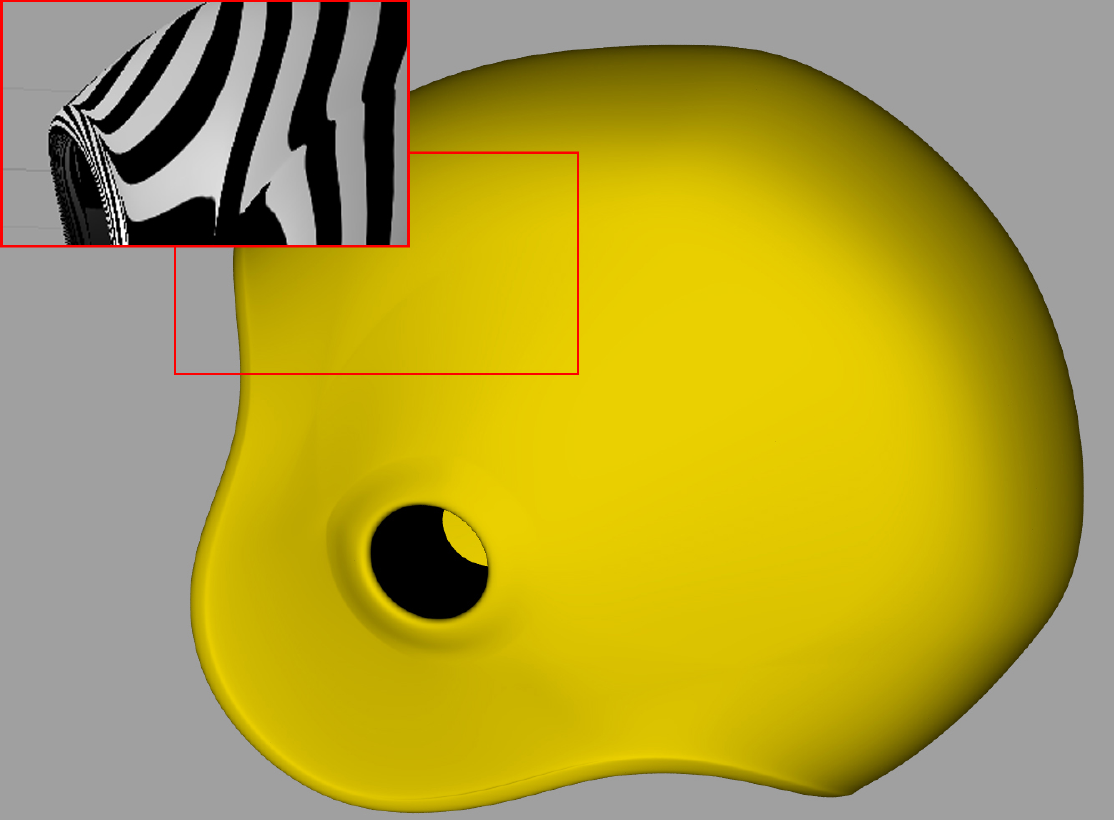}&
\includegraphics[width=0.24\textwidth]{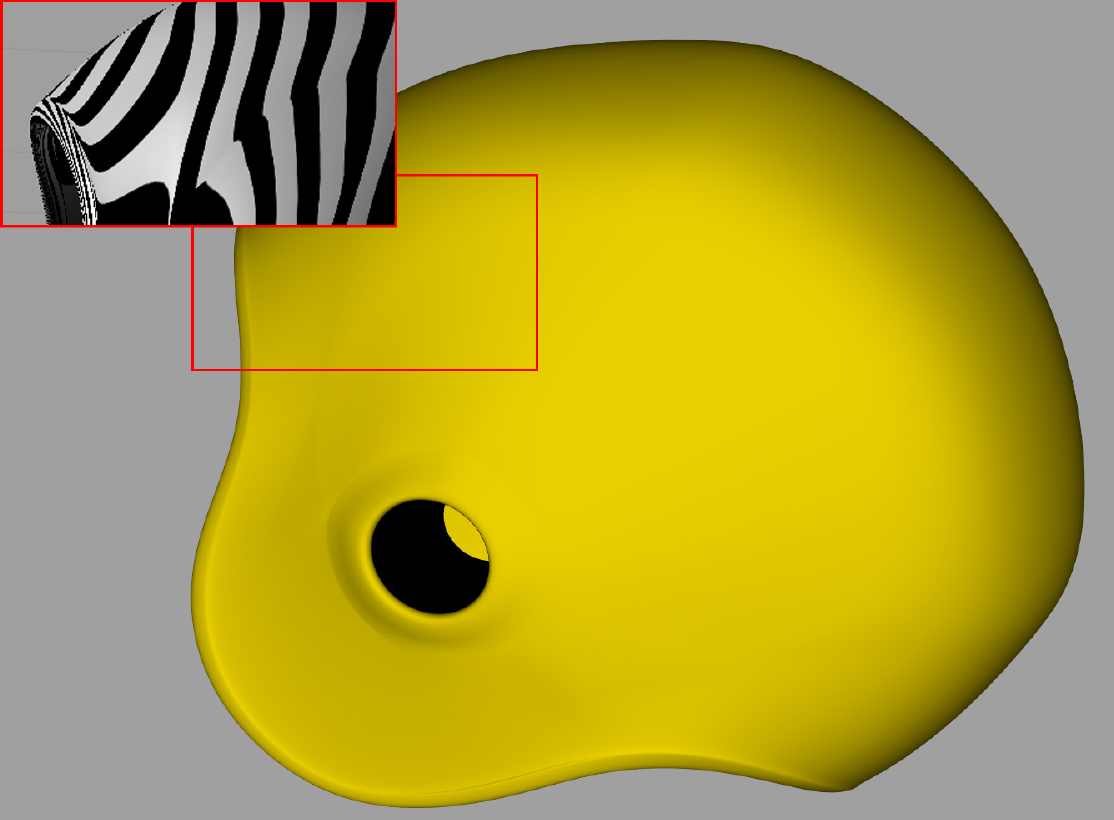}&
\includegraphics[width=0.24\textwidth]{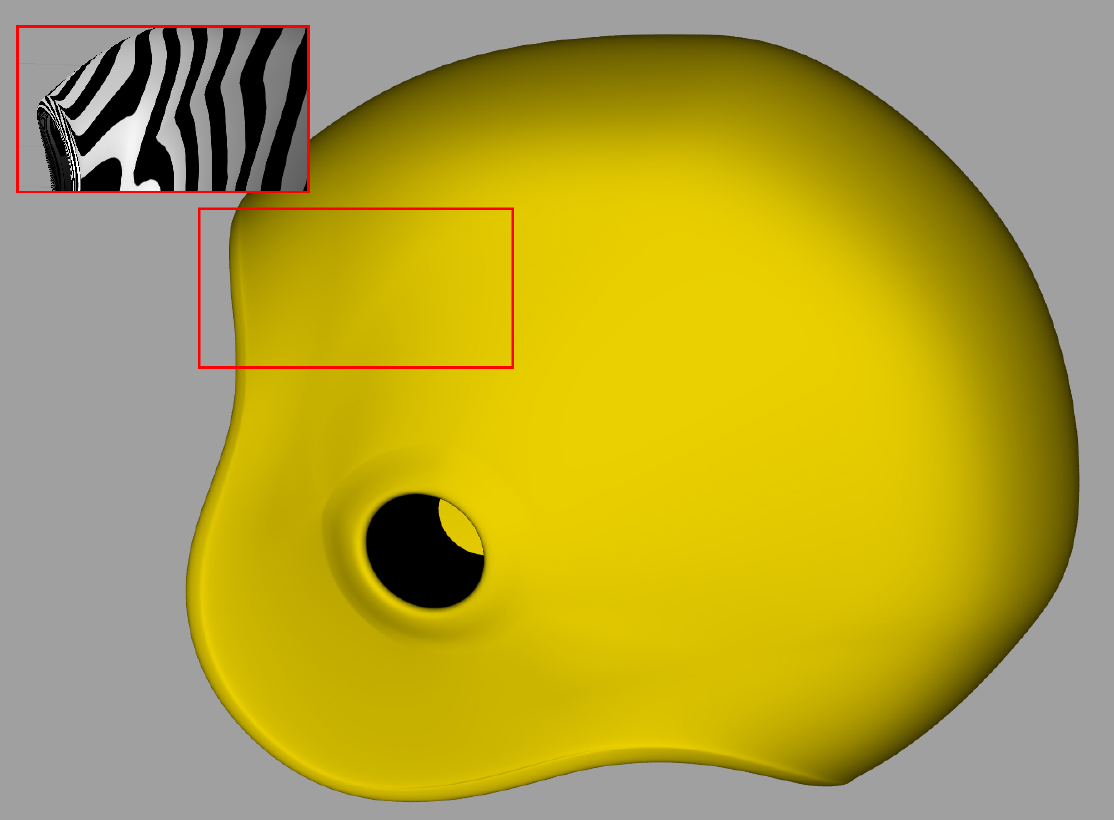}\\
(a) $\lambda = 0.26$ & (b) $\lambda = 0.35$ & (c) $\lambda = 0.65$
\end{tabular}
\end{center}
\caption{A comparison of different $\lambda$ applied to the helmet model. The artifact for the reflection lines exists for $\lambda = 0.26$.}
\label{fig:her}
\end{figure}

\subsection{IGA applications using tHNUS basis functions}

In this section, we test the performance of tHNUS basis functions in the context of IGA. We solve the Poisson's equation with several unstructured quadrilateral meshes as the input. We start with convergence tests on a unit square, whose input control mesh has two EPs, one of valence 3 and the other of valence 5; see Figure \ref{fig:meshterm}(a). These tests are aimed at studying: (1) the role of $\lambda$ in convergence, (2) the feasibility of using reduced quadrature, and (3) the influence of non-uniform parameterizations on convergence.

\begin{figure}[htb]
\centering
\begin{tabular}{cc}
\includegraphics[width=0.48\textwidth]{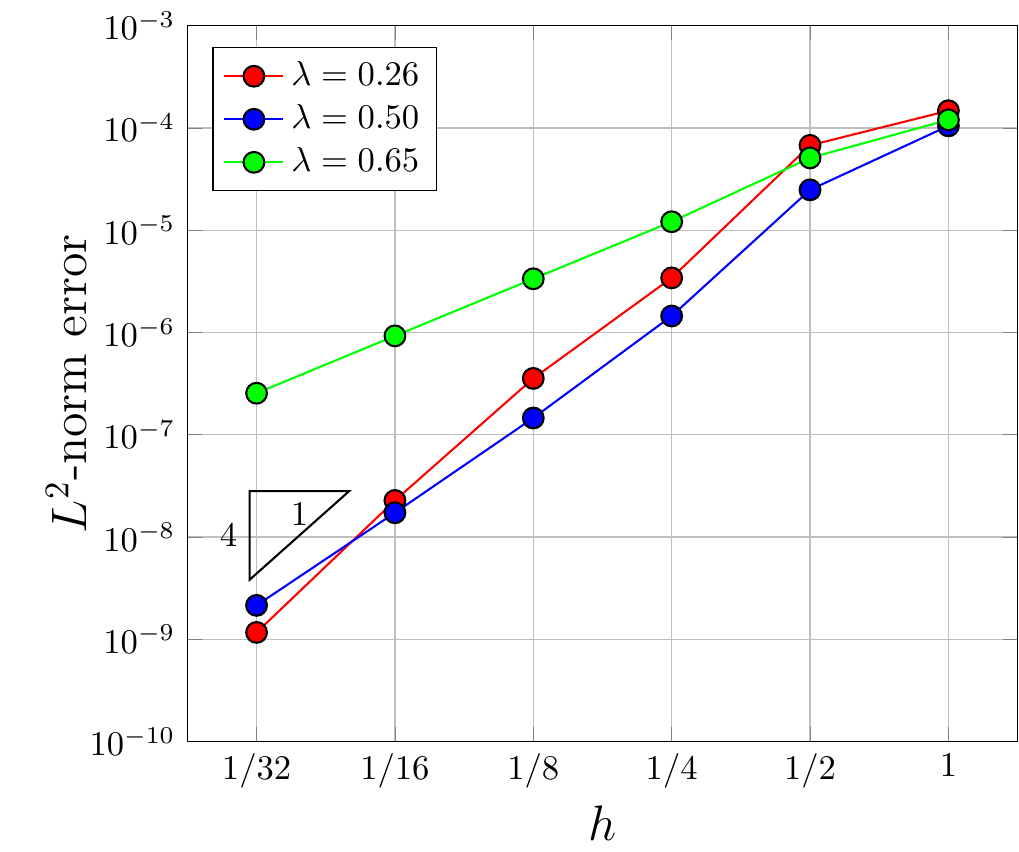}&
\includegraphics[width=0.48\textwidth]{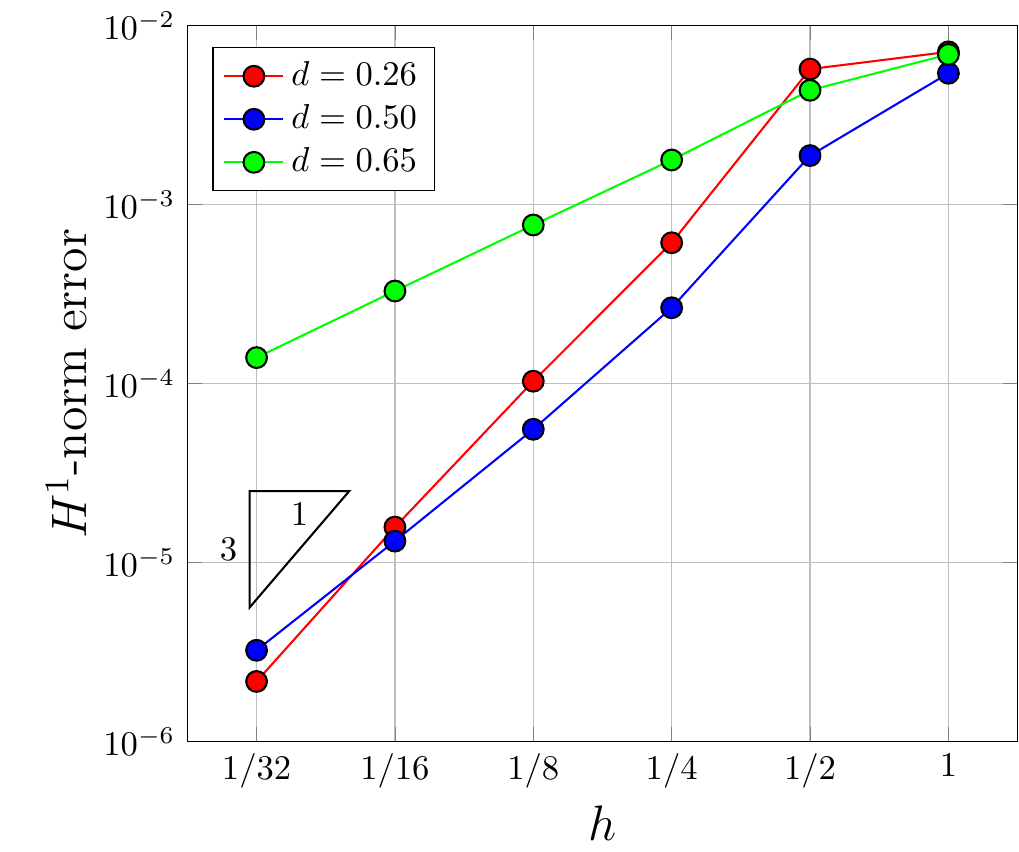}\\
(a) & (b)\\
\end{tabular}
\caption{Convergence plots using $\lambda=\{0.65,0.5,0.26\}$. Particularly, $\lambda=0.5$ corresponds to the original HNUS whereas $\lambda=0.26$ recovers optimal convergence rates.}
\label{fig:lambda}
\end{figure}

First, we study the influence of the tuning parameter $\lambda$ on convergence behavior, where we choose $\lambda$ to be $0.65$, $0.5$, and $0.26$. Recall that tHNUS is equivalent to the original HNUS when $\lambda=0.5$. We adopt uniform parameterization (i.e., same knot intervals) around EPs as well as full quadrature in this study. With the manufactured solution $u(x,y)=\sin(\pi x) \sin(\pi y)$, we summarize the convergence plots in Figure \ref{fig:lambda}. We observe that a smaller $\lambda$ delivers a better convergence behavior, and particularly, optimal convergence rates are achieved when $\lambda=0.26$. The tuned Catmull-Clark subdivision (with uniform parameterization) was studied in \cite{ref:ma19}, where optimal convergence rates in the $L^2$-norm were observed in the Poisson's problem when $\lambda=0.39$. It indicates that the tuning parameter in tHNUS plays a less sensitive role than that in \cite{ref:ma19} because tHNUS requires a smaller $\lambda$ to recover optimal convergence. The reason may be that $\lambda$ brings more vertices to move further towards each EP than in tHNUS. As a result, the tuned Catmull-Clark subdivision has a faster shrinkage in irregular regions. We will provide insights about why reducing $\lambda$ recovers optimal convergence later when we study the meshes with high-valence EPs.

%We recall that $\lambda$ is the subdominant eigenvalue of the subdivision matrix. The result here is also consistent with \cite{ref:ma19} in the sense that the convergence of a subdivision scheme is mostly affected by the subdominant eigenvalue. in each refined irregular face, only one vertex is influenced by $\lambda$ in tHNUS, whereas all the vertices are influenced in the tuned Catmull-Clark subdivision

%Note that although the isoparametric lines using tHNUS bend towards $s(0,0)$, it does not imply a degenerated parameterization. In contrast, the isoparametric lines using degenerated patches \cite{deepesh} originate from $s(0,0)$.

\begin{figure}[htb]
\centering
\begin{tabular}{cc}
\includegraphics[width=0.48\textwidth]{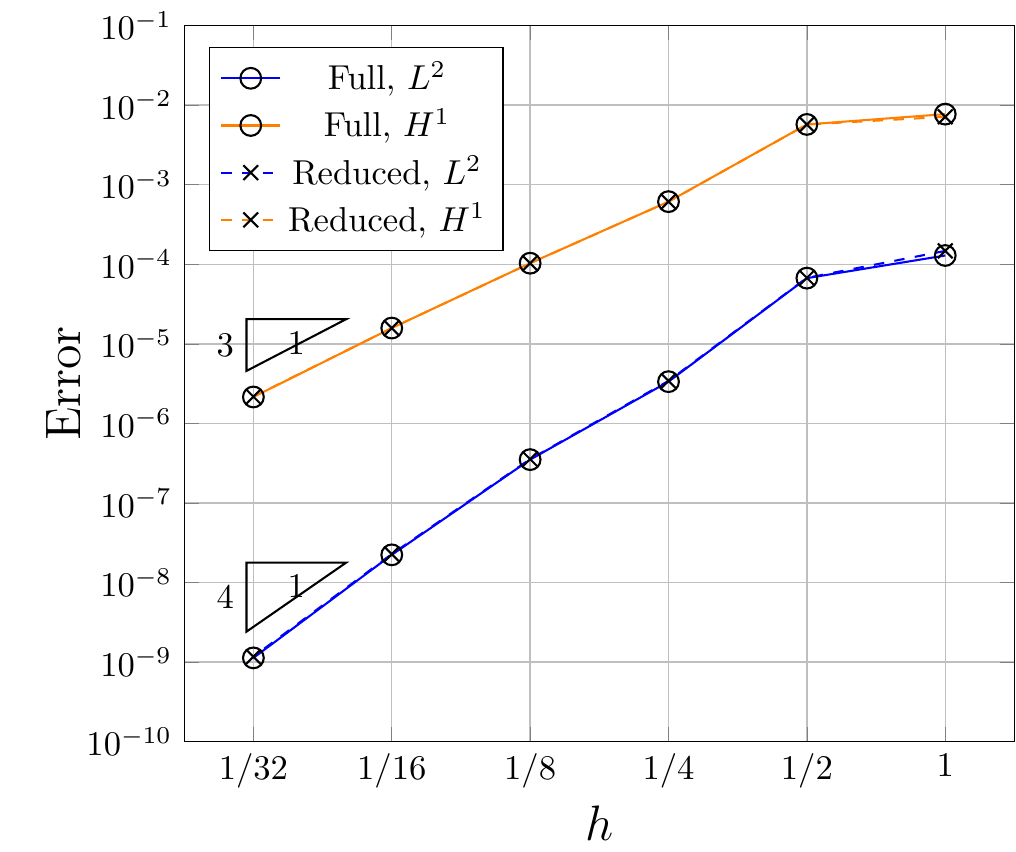}&
\includegraphics[width=0.48\textwidth]{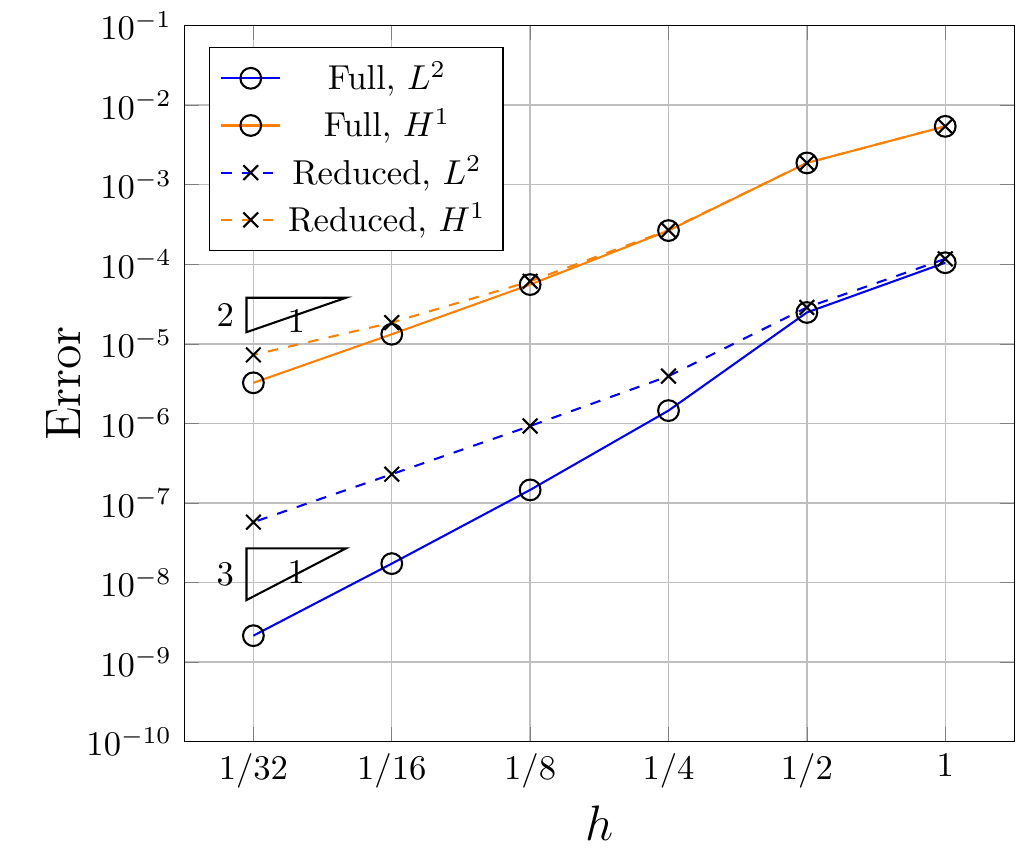}\\
(a) $\lambda=0.26$ & (b) $\lambda=0.5$\\
\end{tabular}
\caption{Convergence plots using full and reduced quadrature.}
\label{fig:quad}
\end{figure}

Second, we compare two quadrature schemes in irregular elements, full quadrature versus the reduced quadrature, under uniform parameterization with $\lambda=0.26$ and $\lambda=0.5$. We observe in Figure \ref{fig:quad}(a) that there is no noticeable difference in terms of both $L^2$- and $H^1$-norm errors. In other words, both quadrature schemes deliver the same level of accuracy when $\lambda=0.26$. In contrast, quadrature plays an important role when $\lambda=0.5$, where the full quadrature yields nearly one-order higher convergence rates than the reduced quadrature. This indicates that when $\lambda=0.26$, the corresponding basis functions (piecewise polynomials) in irregular elements can be better approximated by polynomials than those using $\lambda=0.5$, and 16 quadrature points seem to suffice to retain accuracy. However, further study is needed to fully understand the mechanism behind.

%This is because when $\lambda=0.26$, the underlying basis functions possess a better approximation property. the portion of irregular elements in the mesh is usually small, and moreover, it will be further reduced along with the mesh refinement. Therefore, the quadrature accuracy in irregular elements does not play a dominant role in the approximation error. Indeed, it is usually the basis functions in irregular elements that control the approximation accuracy. As the reduced quadrature suffices, we use it in our numerical tests.

%\cite{mikebarton}

\begin{figure}[htb]
\centering
\begin{tabular}{cc}
\includegraphics[width=0.35\textwidth]{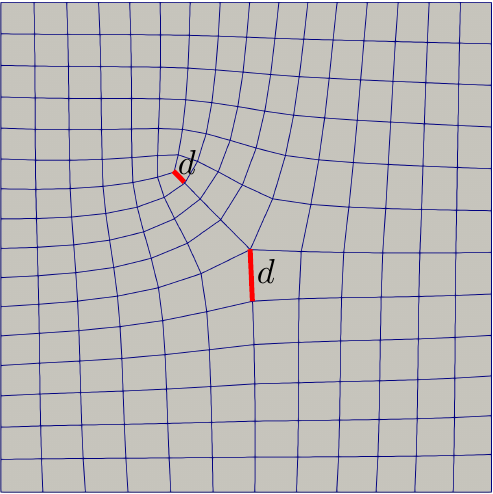} &
\includegraphics[width=0.35\textwidth]{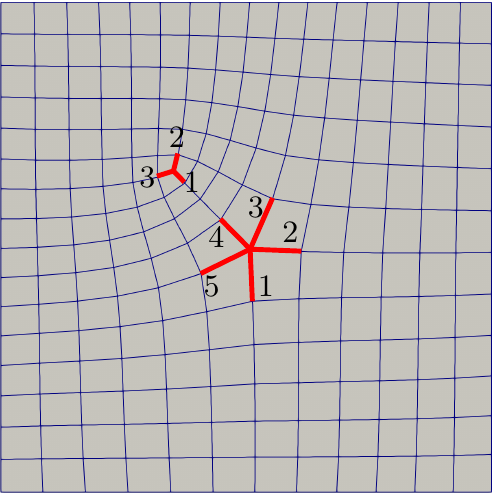}\\
(a) & (b)\\
\includegraphics[width=0.45\textwidth]{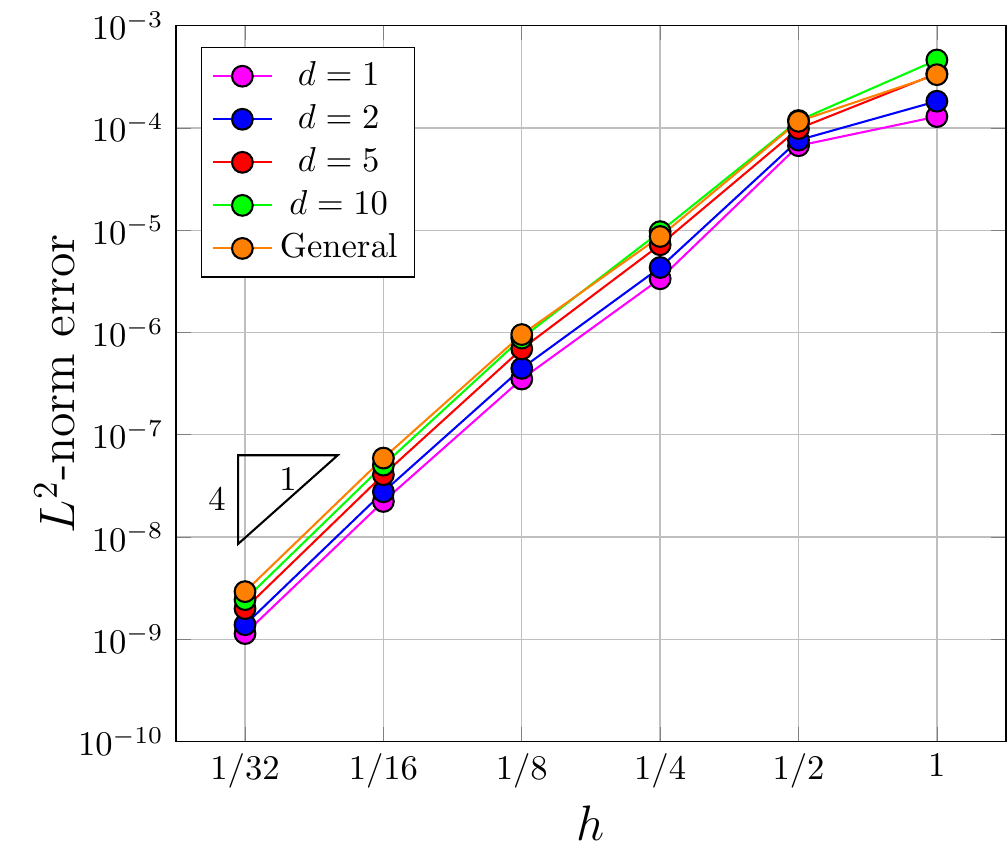} &
\includegraphics[width=0.45\textwidth]{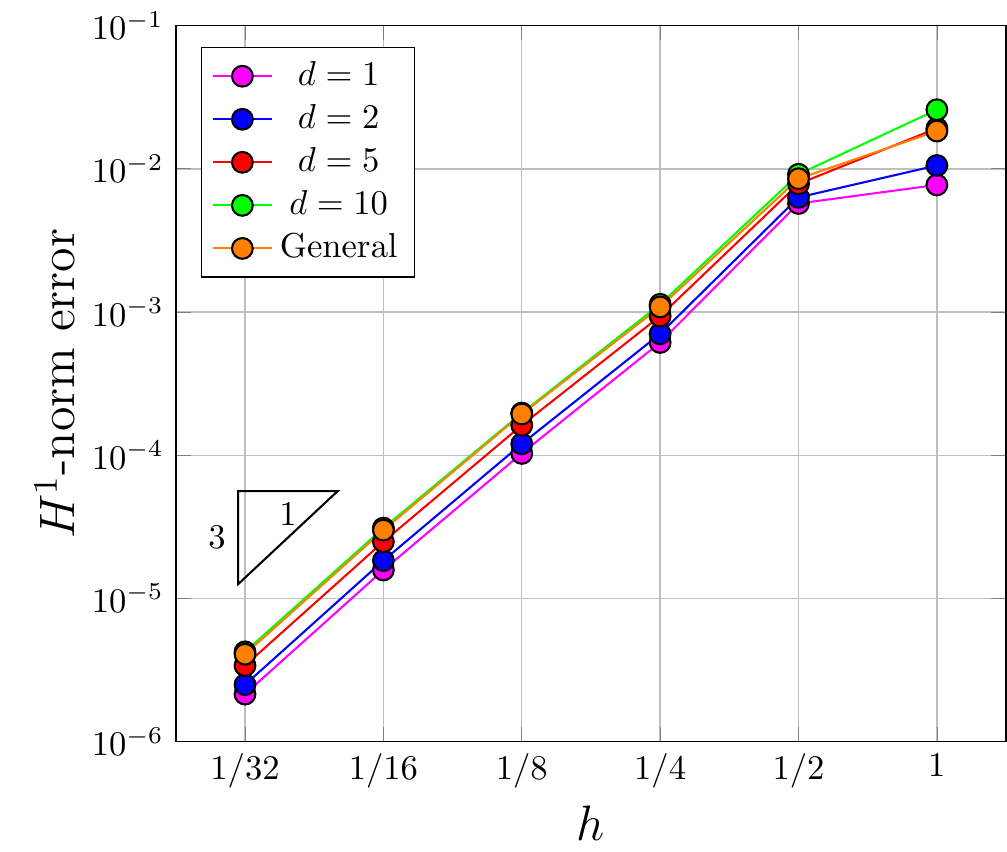}\\
(c) & (d)\\
\end{tabular}
\caption{Convergence plots under different non-uniform parameterizations. (a, b) The configurations of knot intervals around EPs, and (c, d) convergence plots in $L^2$- and $H^1$-norm errors.}
\label{fig:square_non_unif}
\end{figure}

Third, we study several different non-uniform parameterizations for convergence test, which can be obtained by assigning different knot intervals to the edges in the input control mesh. Semi-uniform knot intervals are usually adopted in the literature, where all the edges are assigned a unit knot interval except for those perpendicular to the boundary, which are assigned a zero knot interval. To have non-uniform parameterization around EPs, we modify the semi-uniform setting in two ways: (1) the knot interval (denoted by $d$) of highlighted spoke edges takes values $d\in\{1,2,5,10\}$; and (2) every spoke edge is assigned a different knot interval; see Figure \ref{fig:square_non_unif}(a, b). In both cases, we observe in Figure \ref{fig:square_non_unif}(c, d) that tHNUS basis functions can achieve optimal convergence rates with $\lambda=0.26$. We also observe that the convergence plots corresponding to a larger $d$ slightly shift up, meaning that larger difference in knot intervals yields larger approximation error. In other words, the ``distortion" in parameterization influences accuracy rather than convergence.

\begin{figure}[htb]
\centering
\begin{tabular}{ccc}
\includegraphics[width=0.3\textwidth]{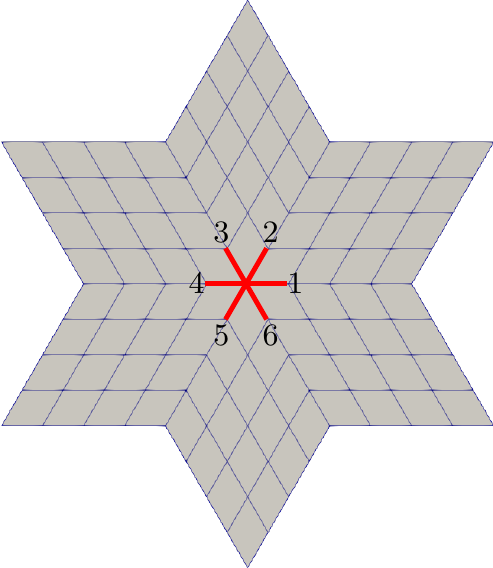} &
\includegraphics[width=0.3\textwidth]{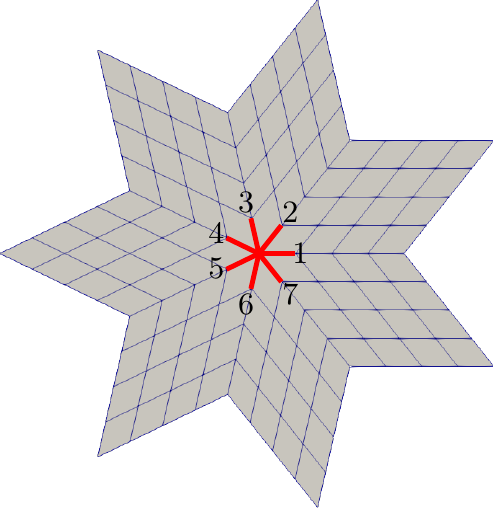} &
\includegraphics[width=0.3\textwidth]{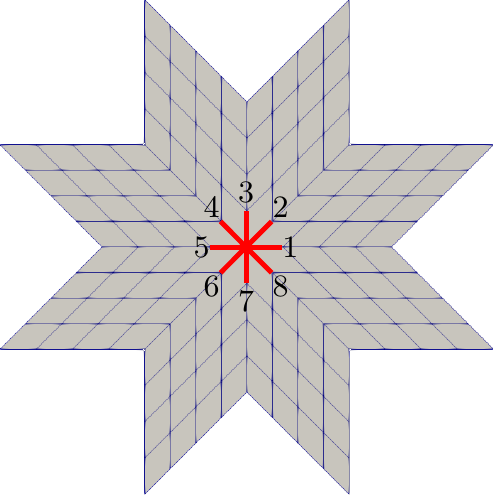}\\
(a) & (b) & (c)\\
\includegraphics[width=0.32\textwidth]{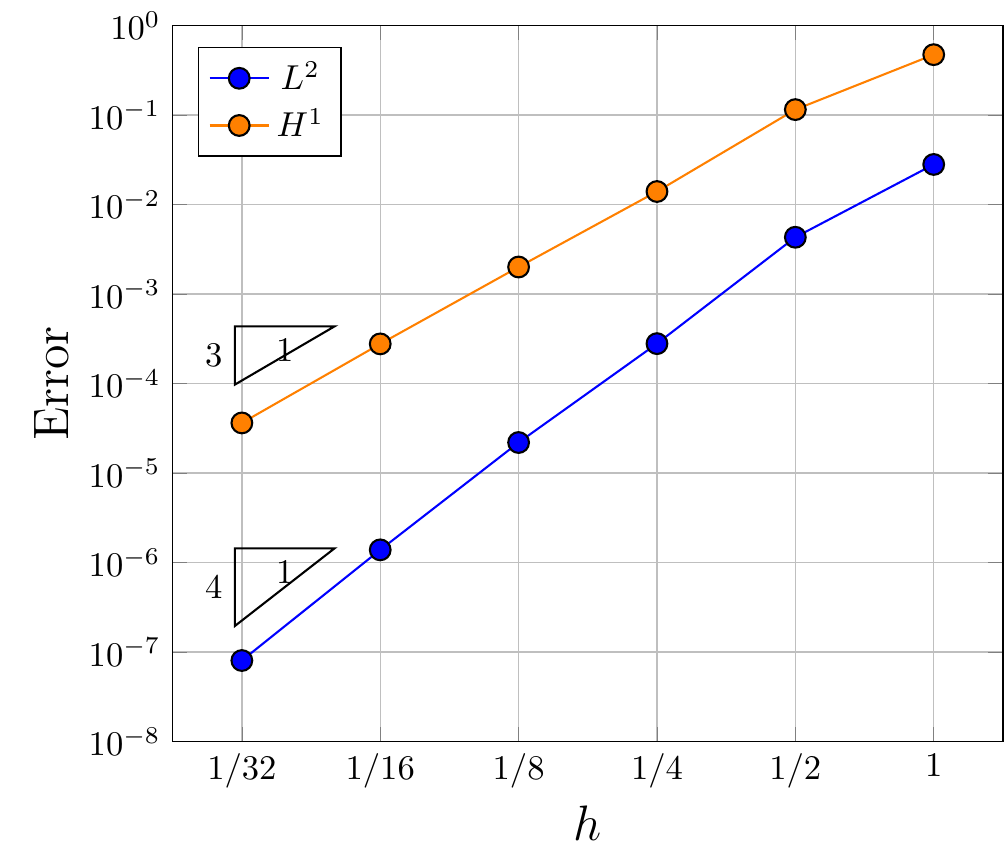} & \hspace{-4mm}
\includegraphics[width=0.32\textwidth]{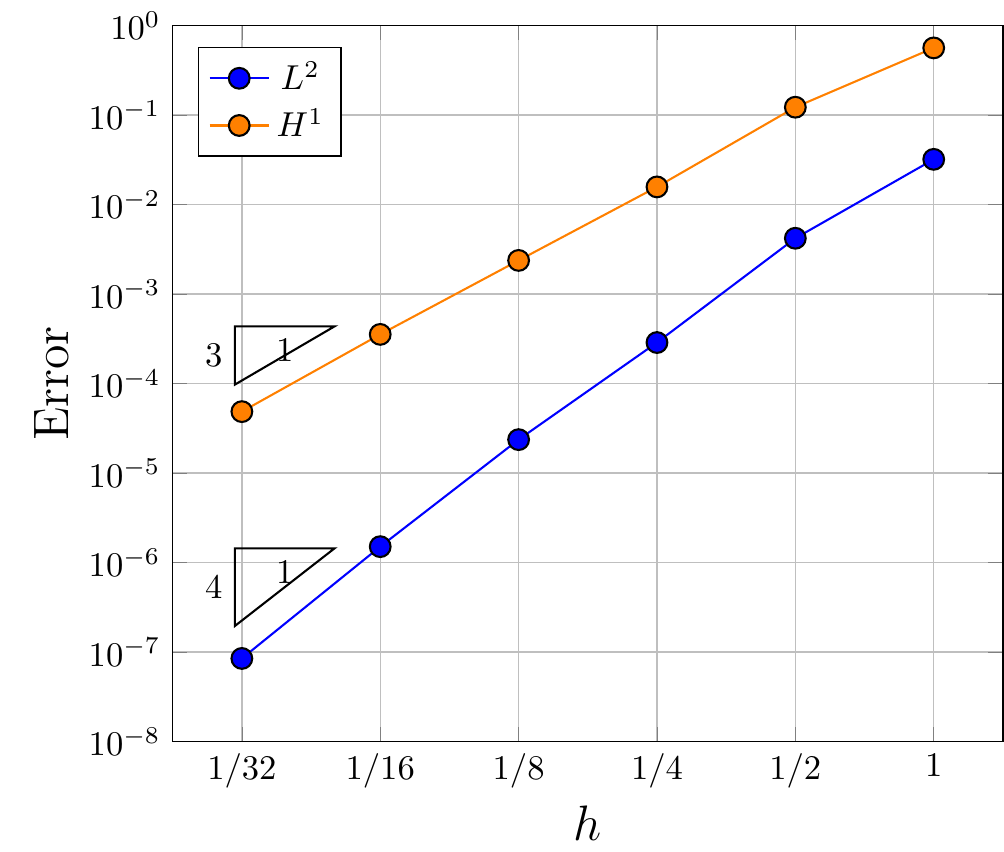} & \hspace{-4mm}
\includegraphics[width=0.32\textwidth]{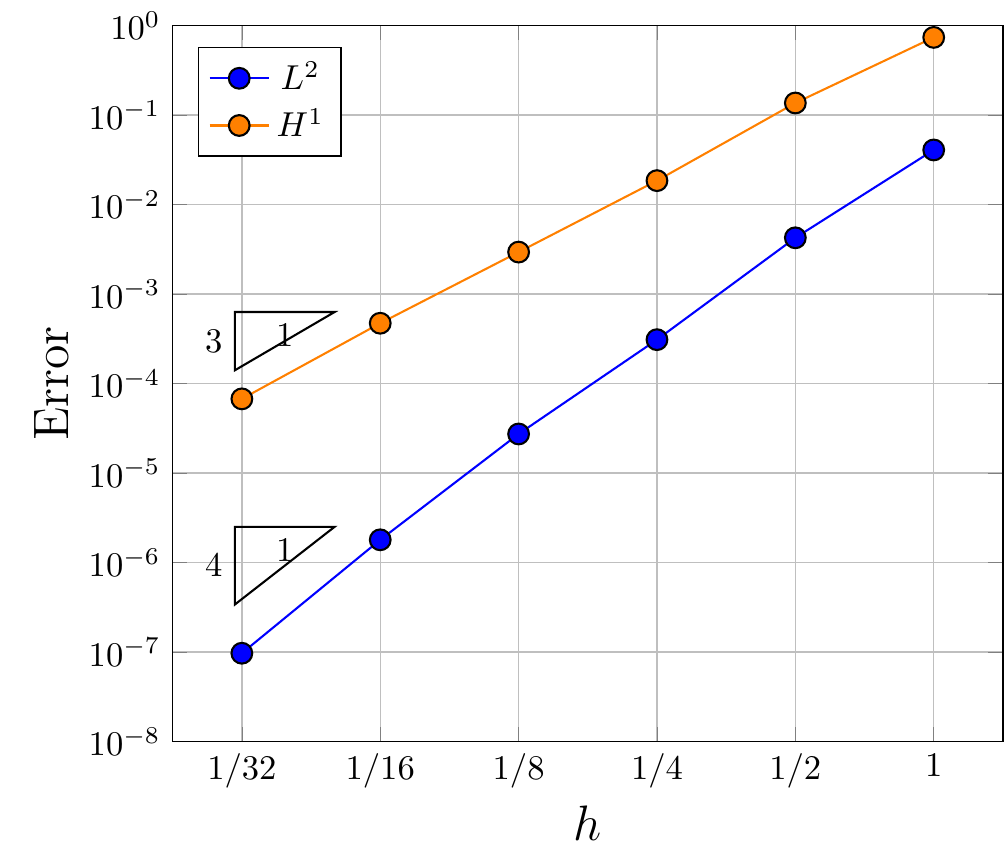}\\
(d) & (e) & (f)\\
\end{tabular}
\caption{Convergence plots using meshes with high-valence EPs. (a--c) The configurations of knot intervals around EPs, and (d--f) convergence plots corresponding to the input meshes in (a--c), respectively.}
\label{fig:highv}
\end{figure}

\begin{figure}[htb]
\centering
\includegraphics[width=0.6\textwidth]{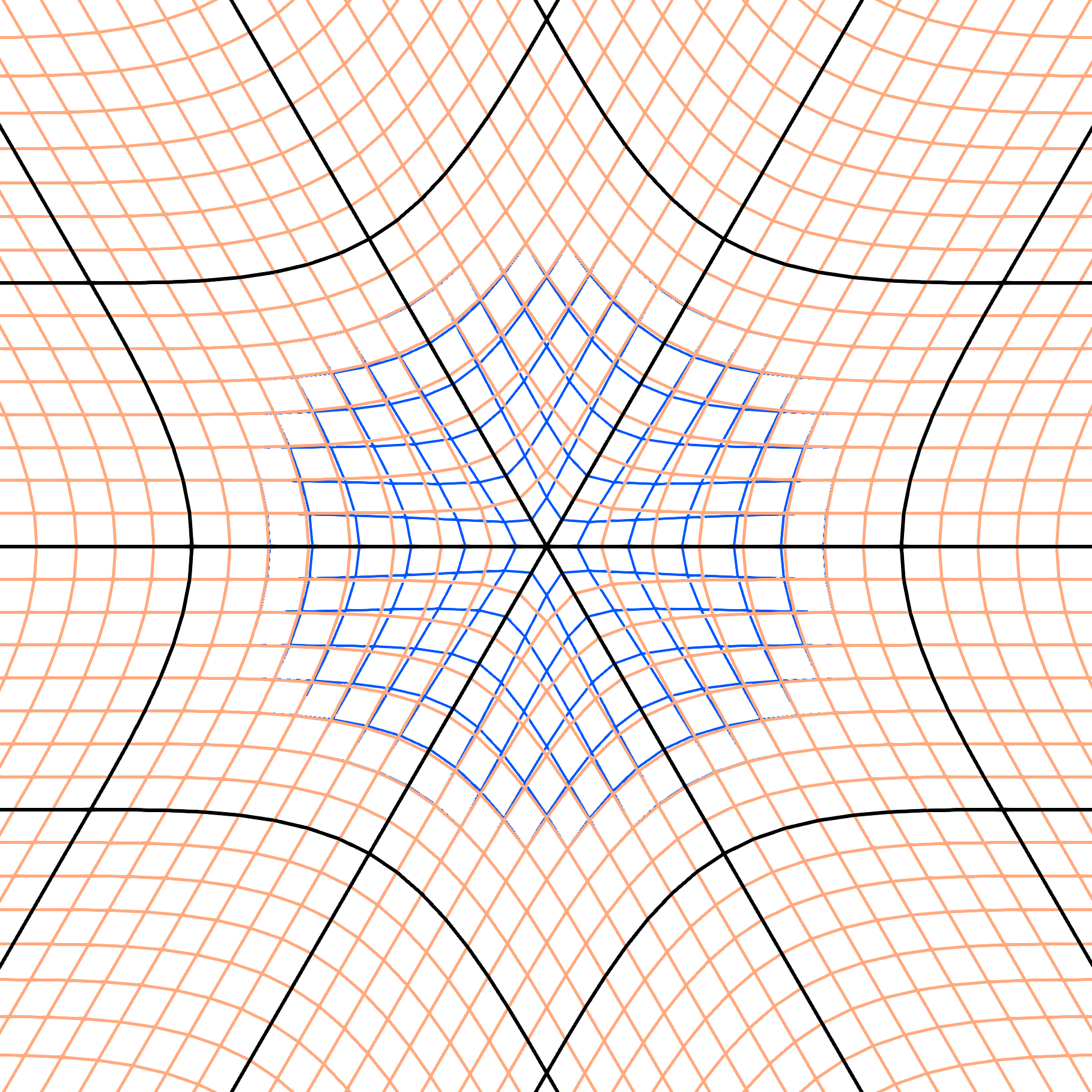}
\caption{Isoparametric lines around the valence-6 EP under uniform parameterization. Orange and blue curves are isoparametric lines corresponding to $\lambda=0.5$ and $\lambda=0.26$, respectively. Blue curves are not visible in most regions because they are overlaid with orange ones. Black curves indicate element boundaries in the physical domain.}
\label{fig:isoline}
\end{figure}

Now, we consider meshes with high-valence EPs (valence 6, 7 and 8), where each spoke edge is assigned a different knot interval. We again observe optimal convergence rates with $\lambda=0.26$; see Figure \ref{fig:highv}. Moreover, let us have a close look at how $\lambda$ influences parameterization around EPs. In particular, we compare isoparametric lines using two different $\lambda$'s ($0.5$ versus $0.26$) around a valence-6 EP, where the same input control mesh in Figure \ref{fig:highv}(a) is used in both cases. We find in Figure \ref{fig:isoline} that isoparametric lines are overlaid with one another in most regions, and with a smaller $\lambda$, the isoparametric lines (blue curves) are more bent towards the extraordinary surface point $s(0,0)$. Equivalently speaking, a smaller $\lambda$ yields smaller refined irregular elements in the physical domain. Therefore, the mesh around $s(0,0)$ becomes denser than that using a larger $\lambda$, and as a result, the asymptotic approximation error controlled by $s(0,0)$ can be reduced using such a denser mesh. Ideally, optimal convergence rates can be achieved by reducing $\lambda$, which indeed is the case in all our numerical tests when $\lambda=0.26$.

\begin{myremark}
Although the globally smooth tHNUS basis functions can be applied to solve 4th-order partial differential equations (PDEs), our preliminary tests only show suboptimal convergence in solving the biharmonic equation, where obtained convergence rates in terms of $L^2$-, $H^1$- and $H^2$-norm errors are around 2, 2, and 1, respectively. This is consistent with the result reported in \cite{ref:qzhang18}, where a thin-shell problem was solved and the reported convergence rates in $L^2$- and energy norm errors are 2 and 1, respectively. In other words, reducing $\lambda$ alone is not sufficient for high-order PDEs. We conjecture that to recover the optimal convergence in this case, we may need more degrees of freedom around EPs following similar ideas in \cite{c0spline, c03d}. However, this would further complicate the current subdivision framework, so we postpone related results in a follow-up work.
\end{myremark}

\section{Conclusions and future work}
\label{sec:con}

We have presented a tuned version of hybrid non-uniform subdivision, tHNUS, by introducing a parameter $\lambda \in (\frac{1}{4}, 1)$, which is also the second and third eigenvalues of the subdivision matrix. The tHNUS surface is proved to be $G^1$-continuous for any positive knot intervals and extraordinary vertices of any valence. The tHNUS surface has satisfactory shape quality for any $\lambda$ under non-uniform parameterization. However, the highest shape quality is achieved when $\lambda=0.5$. In other words, the original HNUS generally performs better in geometric modeling than tHNUS. On the other hand, tHNUS basis functions can achieve optimal convergence rates when $\lambda$ is reduced to $0.26$, regardless of which quadrature scheme is used and whether parameterizations are uniform or non-uniform around EPs.

In the future, we can extend tHNUS in the following three issues. First, converting an input quadrilateral mesh to its hybrid counterpart can be restricted locally to irregular regions without introducing zero-measure faces throughout the entire mesh. This can be done by allowing T-junctions~\cite{ref:sederberg04} in the hybrid mesh, but support of T-junctions in a hybrid mesh requires a much more sophisticated data structure to accommodate both polygonal faces and quadrilaterals with T-junctions. Second, tHNUS can be adapted to hierarchical splines \cite{ref:vuong11} due to its refinability property, where the initial level corresponds to the initial hybrid mesh. The construction of hierarchical tHNUS essentially follows those proposed for truncated hierarchical Catmull-Clark subdivision surfaces \cite{ref:wei15a, ref:wei15b}, but the differences lie in dealing with hybrid meshes and non-uniform knot intervals. Third, improving tHNUS to achieve optimal convergence rates in solving high-order PDEs is another challenging but very interesting direction to pursue. Currently, we can only show optimal convergence in solving the 2nd-order PDE. In the case of high-order PDEs, our preliminary tests suggest that it is not sufficient to tune $\lambda$ alone and additional treatment is needed. We plan to our investigation by adding more control points around extraordinary vertices.

\section*{Acknowledgements}

X. Wei was supported in part by ERC AdG project CHANGE n. 694515. X. Li was supported by the NSF of China (No.61872328), NKBRPC (2011CB302400), SRF for ROCS SE, and the Youth Innovation Promotion Association CAS. Y. Zhang was supported in part by the NSF grant CBET-1804929. T.J.R. Hughes was partially supported by the Office of Naval Research, USA (Grant Nos. N00014-17-1-2119 and N00014-13-1-0500).

\section*{Appendix}

The control points $P_i^{j,k}$ ($1\leq j,k\leq 3$) of the characteristic map are listed as follows.
\begin{tiny}
\begin{align*}
P_{i}^{1,1} &= -\frac{(\lambda-2) \left(\left(32 \lambda^3-100 \lambda^2+31 \lambda-2\right) p-12 \lambda^2 (v+w)\right)}{(\lambda-4) \lambda (4 \lambda-1) (8 \lambda-1)}\\
P_{i}^{1,2} &= \frac{6 \lambda^2 \left(\left(256 \lambda^3+136 \lambda^2-2251 \lambda-232\right) v+2 \left(128 \lambda^3-184 \lambda^2-149 \lambda+10\right) w\right)+\left(-4096 \lambda^6+18688 \lambda^5-14576 \lambda^4-22112 \lambda^3+8275 \lambda^2-608 \lambda+20\right) p}{(\lambda-4) \lambda (4 \lambda-1) (8 \lambda-1) (16 \lambda-1) (32 \lambda-1)}\\
P_{i}^{1,3} &= \frac{6 \lambda^2 \left(2 \left(320 \lambda^3-919 \lambda^2-800 \lambda+7\right) v+\left(160 \lambda^3-314 \lambda^2-13 \lambda+2\right) w\right)+\left(-2560 \lambda^6+15904 \lambda^5-29348 \lambda^4+3535 \lambda^3+928 \lambda^2-62 \lambda+2\right) p}{(\lambda-4) \lambda^2 (4 \lambda-1) (8 \lambda-1) (16 \lambda-1) (32 \lambda-1)}\\
P_{i}^{2,1} &= \frac{6 \lambda^2 \left(\left(256 \lambda^3-368 \lambda^2-298 \lambda+20\right) v+\left(256 \lambda^3+136 \lambda^2-2251 \lambda-232\right) w\right)+\left(-4096 \lambda^6+18688 \lambda^5-14576 \lambda^4-22112 \lambda^3+8275 \lambda^2-608 \lambda+20\right) p}{(\lambda-4) \lambda (4 \lambda-1) (8 \lambda-1) (16 \lambda-1) (32 \lambda-1)}\\
P_{i}^{2,2} &= \frac{6 \left(4096 \lambda^4+33024 \lambda^3-96320 \lambda^2-11271 \lambda-1160\right) \lambda^2 (v+w)+\left(-65536 \lambda^7-194560 \lambda^6+2362752 \lambda^5-4183824 \lambda^4+1165584 \lambda^3-71505 \lambda^2+976 \lambda+100\right) p}{(\lambda-4) \lambda (4 \lambda-1) (8 \lambda-1) (16 \lambda-1) (32 \lambda-1) (64 \lambda-1)} \\
P_{i}^{2,3} &= \frac{6 \lambda^2 \left(\left(74240 \lambda^4-139104 \lambda^3-145684 \lambda^2-8118 \lambda-7\right) v+\left(58880 \lambda^4-106224 \lambda^3-32914 \lambda^2-2208 \lambda-127\right) w\right)}{2 (\lambda-4) \lambda^2 (4 \lambda-1) (8 \lambda-1) (16 \lambda-1) (32 \lambda-1) (64 \lambda-1)}\\
&+\frac{\left(-942080 \lambda^7+4816384 \lambda^6-6465888 \lambda^5+394032 \lambda^4+345858 \lambda^3-33192 \lambda^2+1313 \lambda-1\right) p}{2 (\lambda-4) \lambda^2 (4 \lambda-1) (8 \lambda-1) (16 \lambda-1) (32 \lambda-1) (64 \lambda-1)}\\
P_{i}^{3,1} &= \frac{6 \lambda^2 \left(\left(160 \lambda^3-314 \lambda^2-13 \lambda+2\right) v+2 \left(320 \lambda^3-919 \lambda^2-800 \lambda+7\right) w\right)+\left(-2560 \lambda^6+15904 \lambda^5-29348 \lambda^4+3535 \lambda^3+928 \lambda^2-62 \lambda+2\right) p}{(\lambda-4) \lambda^2 (4 \lambda-1) (8 \lambda-1) (16 \lambda-1) (32 \lambda-1)} \\
P_{i}^{3,2} &= \frac{6 \lambda^2 \left(\left(58880 \lambda^4-106224 \lambda^3-32914 \lambda^2-2208 \lambda-127\right) v+\left(74240 \lambda^4-139104 \lambda^3-145684 \lambda^2-8118 \lambda-7\right) w\right)}{2 (\lambda-4) \lambda^2 (4 \lambda-1) (8 \lambda-1) (16 \lambda-1) (32 \lambda-1) (64 \lambda-1)}\\
&+\frac{\left(-942080 \lambda^7+4816384 \lambda^6-6465888 \lambda^5+394032 \lambda^4+345858 \lambda^3-33192 \lambda^2+1313 \lambda-1\right) p}{2 (\lambda-4) \lambda^2 (4 \lambda-1) (8 \lambda-1) (16 \lambda-1) (32 \lambda-1) (64 \lambda-1)}\\
P_{i}^{3,3} &= \frac{6 \left(25600 \lambda^4-30144 \lambda^3-61166 \lambda^2-5685 \lambda-236\right) \lambda^2 (v+w)+\left(-409600 \lambda^7+1923584 \lambda^6-1925280 \lambda^5-1010388 \lambda^4+470631 \lambda^3-36036 \lambda^2+1060 \lambda+16\right) p}{(\lambda-4) \lambda^2 (4 \lambda-1) (8 \lambda-1) (16 \lambda-1) (32 \lambda-1) (64 \lambda-1)}
\end{align*}
\end{tiny}

The expressions of all $S_{2}^{j, k}$ ($0\leq j,k\leq 3$) are listed as follows.
\begin{tiny}
\begin{align*}
S_{2}^{0,0} &= \frac{3 \lambda \left(28672 \lambda^4-44160 \lambda^3-104888 \lambda^2-4242 \lambda-59\right) v}{2 (\lambda-4) (4 \lambda-1) (8 \lambda-1) (16 \lambda-1) (32 \lambda-1) (64 \lambda-1)}+\frac{3 \lambda \left(-20480 \lambda^4+125568 \lambda^3+2296 \lambda^2-2886 \lambda-107\right) w}{2 (\lambda-4) (4 \lambda-1) (8 \lambda-1) (16 \lambda-1) (32 \lambda-1) (64 \lambda-1)}\\&+\frac{\left(-589824 \lambda^7+2557952 \lambda^6-1417856 \lambda^5-723088 \lambda^4+318940 \lambda^3-25264 \lambda^2+853 \lambda+3\right) p}{4 (\lambda-4) \lambda (4 \lambda-1) (8 \lambda-1) (16 \lambda-1) (32 \lambda-1) (64 \lambda-1)}\\
S_{2}^{1,0}&=-\frac{6 \lambda \left(14336 \lambda^4-53888 \lambda^3+58744 \lambda^2+2009 \lambda+30\right) v}{(\lambda-4) (4 \lambda-1) (8 \lambda-1) (16 \lambda-1) (32 \lambda-1) (64 \lambda-1)}-\frac{12 \lambda \left(5120 \lambda^4-11808 \lambda^3+1904 \lambda^2-234 \lambda+41\right) w}{(\lambda-4) (4 \lambda-1) (8 \lambda-1) (16 \lambda-1) (32 \lambda-1) (64 \lambda-1)}\\&+\frac{2 \left(81920 \lambda^7-467968 \lambda^6+851904 \lambda^5-540792 \lambda^4+116238 \lambda^3-7683 \lambda^2+211 \lambda+1\right) p}{(\lambda-4) \lambda (4 \lambda-1) (8 \lambda-1) (16 \lambda-1) (32 \lambda-1) (64 \lambda-1)}\\
S_{2}^{2,0}&=-\frac{\left(163840 \lambda^5-360448 \lambda^4+20512 \lambda^3+2732 \lambda^2+571 \lambda+47\right) w}{4 (\lambda-4) (4 \lambda-1) (8 \lambda-1) (16 \lambda-1) (32 \lambda-1) (64 \lambda-1)}-\frac{\left(458752 \lambda^6-1579008 \lambda^5+1185408 \lambda^4+775632 \lambda^3-25056 \lambda^2+567 \lambda-4\right) v}{8 (\lambda-4) \lambda (4 \lambda-1) (8 \lambda-1) (16 \lambda-1) (32 \lambda-1) (64 \lambda-1)}\\&+\frac{\left(2621440 \lambda^8-14696448 \lambda^7+25499648 \lambda^6-13590144 \lambda^5+1352256 \lambda^4+248436 \lambda^3-23683 \lambda^2+897 \lambda-5\right) p}{24 (\lambda-4) \lambda^2 (4 \lambda-1) (8 \lambda-1) (16 \lambda-1) (32 \lambda-1) (64 \lambda-1)}\\
S_{2}^{0,1} &=\frac{18 \lambda \left(2048 \lambda^4-1856 \lambda^3-7028 \lambda^2-669 \lambda-13\right) v}{(\lambda-4) (4 \lambda-1) (8 \lambda-1) (16 \lambda-1) (32 \lambda-1) (64 \lambda-1)}+\frac{18 \lambda \left(-2048 \lambda^4+12288 \lambda^3+1904 \lambda^2-556 \lambda-17\right) w}{(\lambda-4) (4 \lambda-1) (8 \lambda-1) (16 \lambda-1) (32 \lambda-1) (64 \lambda-1)}\\&-\frac{3 \left(32768 \lambda^7-129024 \lambda^6+47872 \lambda^5+108496 \lambda^4-38716 \lambda^3+3018 \lambda^2-95 \lambda-1\right) p}{(\lambda-4) \lambda (4 \lambda-1) (8 \lambda-1) (16 \lambda-1) (32 \lambda-1) (64 \lambda-1)}\\
S_{2}^{1,1} &= -\frac{6 \lambda \left(12288 \lambda^4-38912 \lambda^3+43904 \lambda^2+4127 \lambda+76\right) v}{(\lambda-4) (4 \lambda-1) (8 \lambda-1) (16 \lambda-1) (32 \lambda-1) (64 \lambda-1)}-\frac{12 \lambda \left(6144 \lambda^4-14080 \lambda^3+1036 \lambda^2-299 \lambda+80\right) w}{(\lambda-4) (4 \lambda-1) (8 \lambda-1) (16 \lambda-1) (32 \lambda-1) (64 \lambda-1)}\\&+\frac{2 \left(98304 \lambda^7-553984 \lambda^6+985152 \lambda^5-637296 \lambda^4+138348 \lambda^3-9009 \lambda^2+216 \lambda+4\right) p}{(\lambda-4) \lambda (4 \lambda-1) (8 \lambda-1) (16 \lambda-1) (32 \lambda-1) (64 \lambda-1)}\\
S_{2}^{2,1} &= -\frac{3 \left(32768 \lambda^5-93184 \lambda^4+70336 \lambda^3+61032 \lambda^2-210 \lambda+7\right) v}{2 (\lambda-4) (4 \lambda-1) (8 \lambda-1) (16 \lambda-1) (32 \lambda-1) (64 \lambda-1)}-\frac{3 \left(32768 \lambda^5-70656 \lambda^4-4768 \lambda^3+888 \lambda^2+236 \lambda+15\right) w}{2 (\lambda-4) (4 \lambda-1) (8 \lambda-1) (16 \lambda-1) (32 \lambda-1) (64 \lambda-1)}\\&+\frac{\left(524288 \lambda^8-2883584 \lambda^7+4827136 \lambda^6-2557248 \lambda^5+206544 \lambda^4+64080 \lambda^3-6056 \lambda^2+233 \lambda-1\right) p}{4 (\lambda-4) \lambda^2 (4 \lambda-1) (8 \lambda-1) (16 \lambda-1) (32 \lambda-1) (64 \lambda-1)}\\
S_{2}^{0,2}&=\frac{6 \lambda \left(4096 \lambda^4-1024 \lambda^3-19348 \lambda^2-1877 \lambda-117\right) v}{(\lambda-4) (4 \lambda-1) (8 \lambda-1) (16 \lambda-1) (32 \lambda-1) (64 \lambda-1)}+\frac{6 \lambda \left(-4096 \lambda^4+21888 \lambda^3+20552 \lambda^2-3786 \lambda-97\right) w}{(\lambda-4) (4 \lambda-1) (8 \lambda-1) (16 \lambda-1) (32 \lambda-1) (64 \lambda-1)}\\&+\frac{\left(-65536 \lambda^7+215040 \lambda^6+122880 \lambda^5-518496 \lambda^4+160368 \lambda^3-11706 \lambda^2+295 \lambda+9\right) p}{(\lambda-4) \lambda (4 \lambda-1) (8 \lambda-1) (16 \lambda-1) (32 \lambda-1) (64 \lambda-1)}\\
S_{2}^{1,2}&=-\frac{6 \lambda \left(8192 \lambda^4-28928 \lambda^3+34888 \lambda^2+4205 \lambda+228\right) v}{(\lambda-4) (4 \lambda-1) (8 \lambda-1) (16 \lambda-1) (32 \lambda-1) (64 \lambda-1)}-\frac{48 \lambda \left(1024 \lambda^4-2272 \lambda^3-868 \lambda^2-65 \lambda+39\right) w}{(\lambda-4) (4 \lambda-1) (8 \lambda-1) (16 \lambda-1) (32 \lambda-1) (64 \lambda-1)}\\&+\frac{2 \left(65536 \lambda^7-393216 \lambda^6+768768 \lambda^5-550584 \lambda^4+120330 \lambda^3-6531 \lambda^2+32 \lambda+12\right) p}{(\lambda-4) \lambda (4 \lambda-1) (8 \lambda-1) (16 \lambda-1) (32 \lambda-1) (64 \lambda-1)}\\
S_{2}^{2,2}&=-\frac{\left(32768 \lambda^5-100352 \lambda^4+78176 \lambda^3+82984 \lambda^2+1862 \lambda+7\right) v}{(\lambda-4) (4 \lambda-1) (8 \lambda-1) (16 \lambda-1) (32 \lambda-1) (64 \lambda-1)}-\frac{\left(32768 \lambda^5-63488 \lambda^4-49120 \lambda^3+2596 \lambda^2+845 \lambda+43\right) w}{(\lambda-4) (4 \lambda-1) (8 \lambda-1) (16 \lambda-1) (32 \lambda-1) (64 \lambda-1)}\\&+\frac{\left(524288 \lambda^8-2998272 \lambda^7+5300224 \lambda^6-2834112 \lambda^5-22440 \lambda^4+158118 \lambda^3-13784 \lambda^2+513 \lambda-1\right) p}{6 (\lambda-4) \lambda^2 (4 \lambda-1) (8 \lambda-1) (16 \lambda-1) (32 \lambda-1) (64 \lambda-1)}\\
S_{2}^{0,3}&=\frac{\left(393216 \lambda^7+233472 \lambda^6-2491392 \lambda^5-285864 \lambda^4-21918 \lambda^3-258 \lambda^2\right) v}{24 (\lambda-4) \lambda^2 (4 \lambda-1) (8 \lambda-1) (16 \lambda-1) (32 \lambda-1) (64 \lambda-1)}+\frac{\left(-393216 \lambda^7+1769472 \lambda^6+3983616 \lambda^5-415440 \lambda^4-81576 \lambda^3+1449 \lambda^2-12 \lambda\right) w}{24 (\lambda-4) \lambda^2 (4 \lambda-1) (8 \lambda-1) (16 \lambda-1) (32 \lambda-1) (64 \lambda-1)}\\&+\frac{\left(-1048576 \lambda^8+2555904 \lambda^7+6412288 \lambda^6-14197248 \lambda^5+3735840 \lambda^4-144900 \lambda^3-5573 \lambda^2+717 \lambda-1\right) p}{24 (\lambda-4) \lambda^2 (4 \lambda-1) (8 \lambda-1) (16 \lambda-1) (32 \lambda-1) (64 \lambda-1)}\\
S_{2}^{1,3}&=-\frac{\left(65536 \lambda^5-262144 \lambda^4+334720 \lambda^3+52988 \lambda^2+3718 \lambda+41\right) v}{2 (\lambda-4) (4 \lambda-1) (8 \lambda-1) (16 \lambda-1) (32 \lambda-1) (64 \lambda-1)}-\frac{\left(65536 \lambda^5-139264 \lambda^4-137984 \lambda^3-23368 \lambda^2+8098 \lambda-7\right) w}{2 (\lambda-4) (4 \lambda-1) (8 \lambda-1) (16 \lambda-1) (32 \lambda-1) (64 \lambda-1)}\\&+\frac{\left(1048576 \lambda^8-6782976 \lambda^7+14569472 \lambda^6-11210496 \lambda^5+2223408 \lambda^4-32748 \lambda^3-8632 \lambda^2+615 \lambda+1\right) p}{12 (\lambda-4) \lambda^2 (4 \lambda-1) (8 \lambda-1) (16 \lambda-1) (32 \lambda-1) (64 \lambda-1)}\\
S_{2}^{2,3}&=-\frac{\left(65536 \lambda^5-219136 \lambda^4+173056 \lambda^3+229220 \lambda^2+9745 \lambda+131\right) v}{3 (\lambda-4) (4 \lambda-1) (8 \lambda-1) (16 \lambda-1) (32 \lambda-1) (64 \lambda-1)}-\frac{\left(65536 \lambda^5-108544 \lambda^4-208832 \lambda^3-11944 \lambda^2+6694 \lambda+239\right) w}{3 (\lambda-4) (4 \lambda-1) (8 \lambda-1) (16 \lambda-1) (32 \lambda-1) (64 \lambda-1)}\\&+\frac{\left(1048576 \lambda^8-6291456 \lambda^7+11761664 \lambda^6-6099072 \lambda^5-1021344 \lambda^4+664680 \lambda^3-54730 \lambda^2+1881 \lambda+7\right) p}{18 (\lambda-4) \lambda^2 (4 \lambda-1) (8 \lambda-1) (16 \lambda-1) (32 \lambda-1) (64 \lambda-1)}
\end{align*}
\end{tiny}

The expressions of all $T_{2}^{j, k}$ ($0\leq j,k\leq 3$) are listed as follows.
\begin{tiny}
\begin{align*}
T_{2}^{0,0} &= -\frac{3 \lambda \left(20480 \lambda^4-125568 \lambda^3-2296 \lambda^2+2886 \lambda+107\right) v}{2 (\lambda-4) (4 \lambda-1) (8 \lambda-1) (16 \lambda-1) (32 \lambda-1) (64 \lambda-1)}-\frac{3 \lambda \left(-28672 \lambda^4+44160 \lambda^3+104888 \lambda^2+4242 \lambda+59\right) w}{2 (\lambda-4) (4 \lambda-1) (8 \lambda-1) (16 \lambda-1) (32 \lambda-1) (64 \lambda-1)}\\
&+\frac{\left(-589824 \lambda^7+2557952 \lambda^6-1417856 \lambda^5-723088 \lambda^4+318940 \lambda^3-25264 \lambda^2+853 \lambda+3\right) p}{4 (\lambda-4) \lambda (4 \lambda-1) (8 \lambda-1) (16 \lambda-1) (32 \lambda-1) (64 \lambda-1)}\\
T_{2}^{1,0}&=-\frac{18 \lambda \left(2048 \lambda^4-12288 \lambda^3-1904 \lambda^2+556 \lambda+17\right) v}{(\lambda-4) (4 \lambda-1) (8 \lambda-1) (16 \lambda-1) (32 \lambda-1) (64 \lambda-1)}-\frac{18 \lambda \left(-2048 \lambda^4+1856 \lambda^3+7028 \lambda^2+669 \lambda+13\right) w}{(\lambda-4) (4 \lambda-1) (8 \lambda-1) (16 \lambda-1) (32 \lambda-1) (64 \lambda-1)}\\&-\frac{3 \left(32768 \lambda^7-129024 \lambda^6+47872 \lambda^5+108496 \lambda^4-38716 \lambda^3+3018 \lambda^2-95 \lambda-1\right) p}{(\lambda-4) \lambda (4 \lambda-1) (8 \lambda-1) (16 \lambda-1) (32 \lambda-1) (64 \lambda-1)}\\
T_{2}^{2,0}&=-\frac{6 \lambda \left(4096 \lambda^4-21888 \lambda^3-20552 \lambda^2+3786 \lambda+97\right) v}{(\lambda-4) (4 \lambda-1) (8 \lambda-1) (16 \lambda-1) (32 \lambda-1) (64 \lambda-1)}-\frac{6 \lambda \left(-4096 \lambda^4+1024 \lambda^3+19348 \lambda^2+1877 \lambda+117\right) w}{(\lambda-4) (4 \lambda-1) (8 \lambda-1) (16 \lambda-1) (32 \lambda-1) (64 \lambda-1)}\\&+\frac{\left(-65536 \lambda^7+215040 \lambda^6+122880 \lambda^5-518496 \lambda^4+160368 \lambda^3-11706 \lambda^2+295 \lambda+9\right) p}{(\lambda-4) \lambda (4 \lambda-1) (8 \lambda-1) (16 \lambda-1) (32 \lambda-1) (64 \lambda-1)}\\
T_{2}^{3, 0} &=-\frac{\left(-65536 \lambda^5-38912 \lambda^4+415232 \lambda^3+47644 \lambda^2+3653 \lambda+43\right) w}{4 (\lambda-4) (4 \lambda-1) (8 \lambda-1) (16 \lambda-1) (32 \lambda-1) (64 \lambda-1)}-\frac{\left(131072 \lambda^6-589824 \lambda^5-1327872 \lambda^4+138480 \lambda^3+27192 \lambda^2-483 \lambda+4\right) v}{8 (\lambda-4) \lambda (4 \lambda-1) (8 \lambda-1) (16 \lambda-1) (32 \lambda-1) (64 \lambda-1)}\\&+\frac{\left(-1048576 \lambda^8+2555904 \lambda^7+6412288 \lambda^6-14197248 \lambda^5+3735840 \lambda^4-144900 \lambda^3-5573 \lambda^2+717 \lambda-1\right) p}{24 (\lambda-4) \lambda^2 (4 \lambda-1) (8 \lambda-1) (16 \lambda-1) (32 \lambda-1) (64 \lambda-1)}\\
T_{2}^{0, 1} &= -\frac{12 \lambda \left(5120 \lambda^4-11808 \lambda^3+1904 \lambda^2-234 \lambda+41\right) v}{(\lambda-4) (4 \lambda-1) (8 \lambda-1) (16 \lambda-1) (32 \lambda-1) (64 \lambda-1)}-\frac{6 \lambda \left(14336 \lambda^4-53888 \lambda^3+58744 \lambda^2+2009 \lambda+30\right) w}{(\lambda-4) (4 \lambda-1) (8 \lambda-1) (16 \lambda-1) (32 \lambda-1) (64 \lambda-1)}\\&+\frac{2 \left(81920 \lambda^7-467968 \lambda^6+851904 \lambda^5-540792 \lambda^4+116238 \lambda^3-7683 \lambda^2+211 \lambda+1\right) p}{(\lambda-4) \lambda (4 \lambda-1) (8 \lambda-1) (16 \lambda-1) (32 \lambda-1) (64 \lambda-1)}\\
T_{2}^{1,1}&=-\frac{12 \lambda \left(6144 \lambda^4-14080 \lambda^3+1036 \lambda^2-299 \lambda+80\right) v}{(\lambda-4) (4 \lambda-1) (8 \lambda-1) (16 \lambda-1) (32 \lambda-1) (64 \lambda-1)}-\frac{6 \lambda \left(12288 \lambda^4-38912 \lambda^3+43904 \lambda^2+4127 \lambda+76\right) w}{(\lambda-4) (4 \lambda-1) (8 \lambda-1) (16 \lambda-1) (32 \lambda-1) (64 \lambda-1)}\\&+\frac{2 \left(98304 \lambda^7-553984 \lambda^6+985152 \lambda^5-637296 \lambda^4+138348 \lambda^3-9009 \lambda^2+216 \lambda+4\right) p}{(\lambda-4) \lambda (4 \lambda-1) (8 \lambda-1) (16 \lambda-1) (32 \lambda-1) (64 \lambda-1)}\\
T_{2}^{2,1}&=-\frac{48 \lambda \left(1024 \lambda^4-2272 \lambda^3-868 \lambda^2-65 \lambda+39\right) v}{(\lambda-4) (4 \lambda-1) (8 \lambda-1) (16 \lambda-1) (32 \lambda-1) (64 \lambda-1)}-\frac{6 \lambda \left(8192 \lambda^4-28928 \lambda^3+34888 \lambda^2+4205 \lambda+228\right) w}{(\lambda-4) (4 \lambda-1) (8 \lambda-1) (16 \lambda-1) (32 \lambda-1) (64 \lambda-1)}\\&+\frac{2 \left(65536 \lambda^7-393216 \lambda^6+768768 \lambda^5-550584 \lambda^4+120330 \lambda^3-6531 \lambda^2+32 \lambda+12\right) p}{(\lambda-4) \lambda (4 \lambda-1) (8 \lambda-1) (16 \lambda-1) (32 \lambda-1) (64 \lambda-1)}\\
T_{2}^{3,1} &= -\frac{\left(65536 \lambda^5-139264 \lambda^4-137984 \lambda^3-23368 \lambda^2+8098 \lambda-7\right) v}{2 (\lambda-4) (4 \lambda-1) (8 \lambda-1) (16 \lambda-1) (32 \lambda-1) (64 \lambda-1)}-\frac{\left(65536 \lambda^5-262144 \lambda^4+334720 \lambda^3+52988 \lambda^2+3718 \lambda+41\right) w}{2 (\lambda-4) (4 \lambda-1) (8 \lambda-1) (16 \lambda-1) (32 \lambda-1) (64 \lambda-1)}\\&+\frac{\left(1048576 \lambda^8-6782976 \lambda^7+14569472 \lambda^6-11210496 \lambda^5+2223408 \lambda^4-32748 \lambda^3-8632 \lambda^2+615 \lambda+1\right) p}{12 (\lambda-4) \lambda^2 (4 \lambda-1) (8 \lambda-1) (16 \lambda-1) (32 \lambda-1) (64 \lambda-1)}\\
T_{2}^{0,2} &= \frac{\left(-983040 \lambda^7+2162688 \lambda^6-123072 \lambda^5-16392 \lambda^4-3426 \lambda^3-282 \lambda^2\right) v}{24 (\lambda-4) \lambda^2 (4 \lambda-1) (8 \lambda-1) (16 \lambda-1) (32 \lambda-1) (64 \lambda-1)}+\frac{\left(-1376256 \lambda^7+4737024 \lambda^6-3556224 \lambda^5-2326896 \lambda^4+75168 \lambda^3-1701 \lambda^2+12 \lambda\right) w}{24 (\lambda-4) \lambda^2 (4 \lambda-1) (8 \lambda-1) (16 \lambda-1) (32 \lambda-1) (64 \lambda-1)}\\&+\frac{\left(2621440 \lambda^8-14696448 \lambda^7+25499648 \lambda^6-13590144 \lambda^5+1352256 \lambda^4+248436 \lambda^3-23683 \lambda^2+897 \lambda-5\right) p}{24 (\lambda-4) \lambda^2 (4 \lambda-1) (8 \lambda-1) (16 \lambda-1) (32 \lambda-1) (64 \lambda-1)}\\
T_{2}^{1,2} &=-\frac{3 \left(32768 \lambda^5-70656 \lambda^4-4768 \lambda^3+888 \lambda^2+236 \lambda+15\right) v}{2 (\lambda-4) (4 \lambda-1) (8 \lambda-1) (16 \lambda-1) (32 \lambda-1) (64 \lambda-1)}-\frac{3 \left(32768 \lambda^5-93184 \lambda^4+70336 \lambda^3+61032 \lambda^2-210 \lambda+7\right) w}{2 (\lambda-4) (4 \lambda-1) (8 \lambda-1) (16 \lambda-1) (32 \lambda-1) (64 \lambda-1)}\\&+\frac{\left(524288 \lambda^8-2883584 \lambda^7+4827136 \lambda^6-2557248 \lambda^5+206544 \lambda^4+64080 \lambda^3-6056 \lambda^2+233 \lambda-1\right) p}{4 (\lambda-4) \lambda^2 (4 \lambda-1) (8 \lambda-1) (16 \lambda-1) (32 \lambda-1) (64 \lambda-1)}\\
T_{2}^{2,2} &= -\frac{\left(32768 \lambda^5-63488 \lambda^4-49120 \lambda^3+2596 \lambda^2+845 \lambda+43\right) v}{(\lambda-4) (4 \lambda-1) (8 \lambda-1) (16 \lambda-1) (32 \lambda-1) (64 \lambda-1)}-\frac{\left(32768 \lambda^5-100352 \lambda^4+78176 \lambda^3+82984 \lambda^2+1862 \lambda+7\right) w}{(\lambda-4) (4 \lambda-1) (8 \lambda-1) (16 \lambda-1) (32 \lambda-1) (64 \lambda-1)}\\&+\frac{\left(524288 \lambda^8-2998272 \lambda^7+5300224 \lambda^6-2834112 \lambda^5-22440 \lambda^4+158118 \lambda^3-13784 \lambda^2+513 \lambda-1\right) p}{6 (\lambda-4) \lambda^2 (4 \lambda-1) (8 \lambda-1) (16 \lambda-1) (32 \lambda-1) (64 \lambda-1)}\\
T_{2}^{3,2} &=-\frac{\left(65536 \lambda^5-108544 \lambda^4-208832 \lambda^3-11944 \lambda^2+6694 \lambda+239\right) v}{3 (\lambda-4) (4 \lambda-1) (8 \lambda-1) (16 \lambda-1) (32 \lambda-1) (64 \lambda-1)}-\frac{\left(65536 \lambda^5-219136 \lambda^4+173056 \lambda^3+229220 \lambda^2+9745 \lambda+131\right) w}{3 (\lambda-4) (4 \lambda-1) (8 \lambda-1) (16 \lambda-1) (32 \lambda-1) (64 \lambda-1)}\\&+\frac{\left(1048576 \lambda^8-6291456 \lambda^7+11761664 \lambda^6-6099072 \lambda^5-1021344 \lambda^4+664680 \lambda^3-54730 \lambda^2+1881 \lambda+7\right) p}{18 (\lambda-4) \lambda^2 (4 \lambda-1) (8 \lambda-1) (16 \lambda-1) (32 \lambda-1) (64 \lambda-1)}
\end{align*}
\end{tiny}

The following mathematica code has been used to compute the above control points.
\begin{lstlisting}
clear;
a1 = 2*(1 - la)/la;
a2 = 6*(2 - la)*(1 - la)/(8*la - 1)/la;
a3 = 2*(2 - la)*(1 - la)*(1 + la)/(8*la - 1)/la/la;
b2 = 6*(1 - 2*la)/(16*la - 1);
b3 = (1 - 2*la)*(1 + 2*la)/la/(16*la - 1);
p1 = p + a1*v; p2 = (1 + b2)*p + (a1 + a2 - b2) * v;
p3 = (1 + b2 + b3)*p + (a1 + a2 + a3 - b2 - b3) * v;
p4  = p + a1*w; p5 = (1 + b2)*p + (a1 + a2 - b2) * w;
p6 = (1 + b2 + b3)*p + (a1 + a2 + a3 - b2 - b3) * w;
T = Solve[a*la == a*la*la/4 + p1 * la*(2 - la)/4 + p4 * la*(2 - la)/4 + 
      p * (2 - la)*(2 - la)/4 && (b - a)*la == ( a + p1*3)/8 + (p1*3 + p2*3 + a + b)/32 - 
      a*la*3/4 && (d - a)*la == ( a + p4*3)/8 + (p4*3 + p5*3 + a + d)/32 - 
      a*la*3/4 && (c - b)*la == (p1*3 + p2*3 + a + b)*3/32 - a*la/4 - ( a + p1*3)/8 && 
      (g - d)*la == (p4*3 + p5*3 + a + d)*3/32 - a*la/4 - ( a + p4*3)/8 && 
      (e - a)*la == (p1*30 + p2 * 6 + p4 * 30 + a * 120 + b * 22 + p5 * 6 + 
      d * 22 + e * 4)/256 - (a*la*la + p1 * la*(2 - la) + p4 * la*(2 - la) + 
      p * (2 - la)*(2 - la))*15/64 && (f - c)*la == (a + b)/4 + (a + b + d + e)/16 - 
      (p1 * 3 + p2 * 3 + a + b)*3/32 && (h - g)*la == (a + d)/4 + (a + b + d + e)/16 - 
      (p4 * 3 + p5 * 3 + a + d)*3/32 && (k - f)*la == (a + b + d + e)*3/16 - (a + b)/4 - 
      (p1 * 3 + p2 * 3 + a + b)/32, {a, b, c, d, e, f, g, h, k}];
a = T[[1, 1, 2]]; b = T[[1, 2, 2]]; c = T[[1, 3, 2]];
d = T[[1, 4, 2]]; e = T[[1, 5, 2]]; f = T[[1, 6, 2]];
g = T[[1, 7, 2]]; h = T[[1, 8, 2]]; k = T[[1, 9, 2]];
b11 = (a * 4 + b * 2 + d * 2 + e  * 1 )/9;
b21 = (a * 2 + b * 4 + d * 1 + e  * 2 )/9;
b12 = (a * 2 + b * 1 + d * 4 + e  * 2 )/9;
b22 = (a * 1 + b * 2 + d * 2 + e  * 4 )/9;
b00 = (p + p1 + p4  + a )/4;
b10 = (p1 * 2 + p2 * 1 + a * 2 + b * 1)/6;
b20 = (p1 * 1 + p2 * 2 + a * 1 + b * 2)/6;
b30 = (p2 * 2 + p3 * 1 + b * 2 + c * 1)/6;
b01 = (p4 * 2 + p5 * 1 + a * 2 + d * 1)/6;
b02 = (p4 * 1 + p5 * 2 + a * 1 + d * 2)/6;
b03 = (p5 * 2 + p6 * 1 + d * 2 + g * 1)/6;
b31 = (b* 4 + c * 2 + e * 2 + f  * 1 )/9;
b32 = (b* 2 + c * 1 + e * 4 + f  * 2 )/9;
b33 = (e* 4 + f * 2 + h * 2 + k  * 1 )/9;
b13 = (d* 4 + e * 2 + g * 2 + h  * 1 )/9;
b23 = (d* 2 + e * 4 + g * 1 + h  * 2 )/9;
b00 = (b00 + b10 + b01 + b11)/4;
b30 = (b30 + b20 + b21 + b31)/4;
b03 = (b03 + b13 + b02 + b12)/4;
b33 = (b33 + b32 + b23 + b22)/4;
b10 = (b11 + b10)/2;b20 = (b21 + b20)/2;
b13 = (b12 + b13)/2;b23 = (b22 + b23)/2;
b01 = (b11 + b01)/2;b02 = (b12 + b02)/2;
b31 = (b31 + b21)/2;b32 = (b32 + b22)/2;
Collect[Simplify[{{b00, b10, b20, b30}, {b01, b11, b21, b31}, {b02, b12, b22, b32}, {b03, b13, b23, b33}}], {p, v, w}]
Collect[Simplify[{{b01 - b00, b11 - b10, b21 - b20, b31 - b30}, {b02 - b01, b12 - b11, b22 - b21, 
    b32 - b31}, {b03 - b02, b13 - b12, b23 - b22, b33 - b32}}], {p, v,w}]
Collect[Simplify[{{b10 - b00, b20 - b10, b30 - b20}, {b11 - b01, b21 - b11, b31 - b21}, {b12 - b02, b22 - b12, 
    b32 - b22}, {b13 - b03, b23 - b13, b33 - b23}}], {p, v, w}]
\end{lstlisting}

\bibliographystyle{plain}
\bibliography{InterSubd}
\end{document}